\documentclass[a4paper,12pt]{amsart}
\usepackage[utf8]{inputenc}
\usepackage{amsmath}
\usepackage{amsfonts}
\usepackage{amsthm}
\usepackage{amssymb}
\usepackage[mathscr]{euscript}
\usepackage{xcolor}
\usepackage{enumerate}
\usepackage{tikz}

\usepackage[
margin=1in,
marginpar=2cm,
includefoot,
footskip=30pt,
]{geometry}

\newtheorem{theorem}[equation]{Theorem}
\newtheorem{lemma}[equation]{Lemma}
\newtheorem{proposition}[equation]{Proposition}
\newtheorem{corollary}[equation]{Corollary}

\theoremstyle{definition}
\newtheorem{definition}[equation]{Definition}

\theoremstyle{remark}
\newtheorem{example}[equation]{Example}
\newtheorem{remark}[equation]{Remark}

\numberwithin{equation}{section}

\newcommand{\osf}{{\normalfont \textsf{X}}}

\newcommand{\lang}{\CL_{\osf}}
\newcommand{\algshift}{\CA_R(\osf)}
\newcommand{\ualgshift}{\TCA_R(\osf)}
\newcommand{\dalgshift}{\CD_R(\osf)}
\newcommand{\udalgshift}{\TCD_R(\osf)}

\newcommand{\calgshift}{\CO_{\osf}}

\newcommand{\alf}{\mathscr{A}}

\newcommand{\N}{\mathbb{N}}
\newcommand{\F}{\mathbb{F}}

\newcommand{\CA}{\mathcal{A}}
\newcommand{\CB}{\mathcal{B}}

\newcommand{\CD}{\mathcal{D}}
\newcommand{\CE}{\mathcal{E}}
\newcommand{\CF}{\mathcal{F}}
\newcommand{\CG}{\mathcal{G}}

\newcommand{\CK}{\mathcal{K}}
\newcommand{\CL}{\mathcal{L}}
\newcommand{\CM}{\mathcal{M}}
\newcommand{\CO}{\mathcal{O}}
\newcommand{\CZ}{\mathcal{Z}}
\newcommand{\powerset}[1]{\text{$P(#1)$}}								

\newcommand{\Img}{\operatorname{Im}}

\newcommand{\TCA}{\widetilde{\CA}}
\newcommand{\TCB}{\mathcal{U}}
\newcommand{\TCD}{\widetilde{\CD}}

\newcommand{\HTCB}{\widehat{\TCB}}
\newcommand{\HCB}{\widehat{\CB}}

\newcommand{\hsig}{\widehat{\sigma}}
\newcommand{\hh}{\widehat{h}}

\newcommand{\tauh}{\widehat{\tau}}
\newcommand{\varphih}{\widehat{\varphi}}

\newcommand{\lspace}{(\CE,\CL,\CB)}
\newcommand{\tlspace}{(\CE,\CL,\TCB)}

\newcommand{\nn}{\mathbb{N}}
\newcommand{\zn}{\mathbb{Z}}

\newcommand{\scj}{\subseteq}

\newcommand{\eword}{\omega}

\newcommand{\vecspan}{\operatorname{span}}

\newcommand{\Lc}{\operatorname{Lc}}
\newcommand{\dom}{\operatorname{dom}}
\newcommand{\supp}{\operatorname{supp}}

\title{Algebras of one-sided subshifts over arbitrary alphabets}
\author[G. Boava]{Giuliano Boava}
\author[G.G. de Castro]{Gilles G. de Castro}
\author[D. Gonçalves]{Daniel Gonçalves}
\address[Giuliano Boava, Gilles G. de Castro and Daniel Gonçalves]{Departamento de Matem\'atica, Universidade Federal de Santa Catarina, 88040-970 Florian\'opolis SC, Brazil. }
\email{g.boava@ufsc.br \\ gilles.castro@ufsc.br \\ daemig@gmail.com}
\author[D.W. van Wyk]{Daniel W. van Wyk}
\address[Daniel W. van Wyk]{Department of Mathematics and Statistics, Amherst College, Amherst, MA 01002 USA, and  \\
Department of Mathematics and Applied Mathematics, University of the Free State, Park West, Bloemfontein, 9301, South Africa.}
\email{dvanwyk@amherst.edu }

\begin{document}


\thanks{The second and third authors were partially supported by Capes-Print Brazil. The third author was partially supported by CNPq - Conselho Nacional de Desenvolvimento Científico e Tecnológico - Brazil.
 }

\begin{abstract}
We introduce two algebras associated with a subshift over an arbitrary alphabet. One is unital, and the other not necessarily. We focus on the unital case and describe a conjugacy between Ott-Tomforde-Willis subshifts in terms of a homeomorphism between the Stone duals of suitable Boolean algebras, and in terms of a diagonal-preserving isomorphism of the associated unital algebras. For this, we realise the unital algebra associated with a subshift as a groupoid algebra and a partial skew group ring.
\end{abstract}

\maketitle

\section{Introduction}

The rich interplay between non-commutative algebras and symbolic dynamics dates back to the seminal work of Cuntz and Krieger \cite{CuntzKrieger}, where a C*-algebra (now called the Cuntz-Krieger algebra) is associated with a subshift of finite type (given by a finite matrix). Among its applications, these algebras are invariants for shift conjugacy and are essential in the study of (two-sided) continuous orbit equivalence, see \cite{MatuiMatsumoto}. Moreover, in \cite{MatsumotoConjugacy}, Matsumoto shows that two one-sided subshifts associated with irreducible and nonpermutation \{0, 1\}-matrices are topologically conjugate if and only if there is a diagonal-preserving isomorphism, that commutes with the diagonal action, between their associated C*-algebras. 

A subshift of finite type over a finite alphabet can be seen as the edge subshift associated with a graph. In \cite{BrixCarlsen1}, one-sided conjugacy of subshifts of finite type is characterised in terms of the Cuntz–Krieger algebra's diagonal and a completely positive map. Orbit equivalence of subshifts associated with directed graphs is characterised using graph C*-algebras in \cite{BCW}. The algebraic analogues of graph C*-algebras, called Leavitt path algebras, have attracted researchers' attention from a broad spectrum of Mathematics, as these algebras proved to have close connections with symbolic dynamics and, of course, their analytical counterparts, graph C*-algebras. Most of the invariance results mentioned above have algebraic counterparts obtained via groupoid techniques, as in \cite{GroupoidLeavitt}. 
Moreover, important problems in symbolic dynamics, such as the Williams problem regarding shift equivalence and strong shift equivalence, can be recast as a problem in graph algebras and their graded theory, \cite{HazratClassificationLPA}.

Recently, in \cite{BrixCarlsen}, Brix and Carlsen characterise conjugacy of subshifts (not necessarily of finite type) over finite alphabets in terms of the C*-algebras that Carlsen defines in \cite{CarlsenShift}. The critical problem in this setting is that, for a general subshift, the shift map is not a local homeomorphism. Hence, the established theory of Deaconu-Renault systems (see \cite{Renault00}) does not apply. Brix and Carlsen circumvent this problem by introducing the notion of a cover space. They describe the conjugacy of subshifts in terms of the conjugacy of their cover spaces, which in turn are described in terms of the associated groupoids and C*-algebras.

The main purpose of this paper is to obtain an algebraic description of conjugacy between subshifts over an arbitrary alphabet (including infinite alphabets). For finite alphabets, our conjugacy results may be interpreted as purely algebraic versions of the C*-algebraic results in \cite{BrixCarlsen}. Our approach consists of introducing algebras associated with subshifts over arbitrary alphabets and suitable topological spaces that, although different from the cover spaces in \cite{BrixCarlsen}, play the same role as these spaces.

In the context of a subshift $\osf$ defined over a finite alphabet, the C*-algebra $\calgshift$ studied in \cite{BrixCarlsen} was originally defined in \cite{CarlsenShift} in terms of a C*-correspondence. 
In order to formulate a suitable definition for the purely algebraic counterpart, we adapt the universal property established in \cite[Theorem~7.2]{CarlsenShift}. To accommodate infinite alphabets as well, we make slight modifications to the relations presented in \cite[Theorem~7.2]{CarlsenShift}. However, if we were to define a C*-algebra using our resulting relations, as outlined in Definition~\ref{diachuvoso}, this definition would subsume $\calgshift$ for finite alphabets. Furthermore, our definition readily extends to arbitrary alphabets. The authors will explore these C*-algebras for one-sided subshifts over arbitrary alphabets in an upcoming paper. We decided to first focus on the purely algebraic setting since, even for finite alphabets, associated algebras were not yet defined. This decision was motivated by the study of Leavitt path algebras, which, although connected to the corresponding graph C*-algebras, presents techniques and results that are often different. For instance, there are graphs with isomorphic graph C*-algebras, but there does not exist a *-isomorphism between their associated Leavitt path algebras over the ring of integers \cite{MR3517565}.

An important step toward our goals is to give a groupoid model for the subshift algebra. In Theorem~\ref{melaoflutuante}, we prove that our subshift algebra is isomorphic to the Leavitt labelled path algebra of a certain labelled graph, which can be used in conjunction with \cite{MR4583730} to describe the subshift algebra as a partial skew group ring and as a Steinberg algebra. The results of \cite{MR4583730} are based on \cite{BoavaDeCastroMortari1,Gil3,GillesDanie} and use the theory of inverse semigroups. An alternative approach is to use Boolean dynamical systems and results from \cite{COP,GillesEunJi2,GillesEunJi1}, which utilise inverse semigroups and topological correspondences. However, we have decided to take a more direct and self-contained approach. For the Steinberg algebra description, we use a groupoid whose unit space is just the Stone dual of the Boolean algebra that appears in the definition of the subshift algebra. For the partial skew group ring, we define a topological partial action using the same space and a set-theoretical partial action, which generalises that of \cite{MishaRuy} from subshifts of finite alphabets to arbitrary alphabets.

In symbolic dynamics, the study of subshifts over infinite alphabets is the subject of intense research, with practical applications \cite{KitchensBook, LindMarcus}. The main difficulty for infinite alphabets with the discrete topology is that the infinite product space is not compact (not even locally compact). This difficulty has motivated several approaches for subshifts over infinite alphabets, such as countable Markov subshifts (where the full shift is the usual shift with the product topology) or compactifications of subshifts, among others; see \cite{OTW} for an overview. In \cite{OTW}, Ott, Tomforde, and Willis (from now on referenced as OTW) introduce a new subshift associated with an infinite countable alphabet. They show that conjugacy of subshifts associated with infinite graphs implies isomorphism of the graph C*-algebras. In \cite{MR3938320}, ultragraphs are used to propose a notion of subshifts of finite type over infinite alphabets. This notion is closely related to the ideas of OTW, and \cite{MR3938320} shows that conjugacy of ultragraph shifts implies isomorphism of their algebras. The relation between ultragraphs and OTW subshifts is studied in \cite{MR3600124}, and many results for ultragraphs shifts are described in \cite{MR4443753}. Nevertheless, the connection between the above notions of subshifts over infinite alphabets and their non-commutative algebras is, in general,  not fully described. With our algebras, we describe such a connection. We describe the conjugacy of OTW subshifts in terms of the conjugacy of our analogues of cover spaces. These, in turn, imply a diagonal-preserving graded isomorphism of the subshift algebras. Therefore, we establish a two-way connection between OTW subshifts and non-commutative algebras.

The problem of topological conjugacy for subshifts over finite alphabet was explored in \cite{BrixCarlsen,CarlsenShift,MishaRuy}. Carlsen proved in \cite[Theorem~8.6]{CarlsenShift} that $\calgshift$ is an invariant for topological conjugacy, whereas Brix and Carlsen established a complete characterisation of topological conjugacy in terms of $\calgshift$ in \cite[Theorem~4.4]{BrixCarlsen}. One of the main results of our paper, Theorem~\ref{theone}, adds different characterisations of topological conjugacy of subshifts over finite alphabets to those in \cite[Theorem~4.4]{BrixCarlsen}. Our result is also a generalisation for subshifts over infinite alphabets. As for the results of Dokuchaev and Exel in \cite{MishaRuy}, they study two C*-algebras associated with a subshift $\osf$, namely $\CM_X$ and $\calgshift$ (the second being the same as the one above by \cite[Theorem~10.2]{MishaRuy}). The algebra $\CM_X$ is defined using operators on a Hilbert space. It can be seen as a quotient of $\CO_X$ and there is no known description of it via a universal property. Dokuchaev and Exel prove that $\CM_X$ is an invariant for topological conjugacy, but they mention that their method ``does not seem appropriate'' to give a different proof for \cite[Theorem~8.6]{CarlsenShift}. In a way, our Theorem~\ref{isogroupoid} can be considered the missing piece needed to establish that $\CO_X$ serves as an invariant for topological conjugacy when employing a partial action approach.

Before we describe the structure of the paper, we point out that we define two algebras associated with a subshift: one is unital by definition, and the other not necessarily. When the latter is unital, they coincide. Otherwise, its unitization coincides with a unital subshift algebra (see Proposition~\ref{jabutirapido}). We focus on the unital case because it is related to OTW subshifts (see Proposition~\ref{algebra OTW} and Theorem~\ref{theone}). The non-unital case is interesting when studying Leavitt path algebras of graphs with infinite vertices since these are not unital. Under certain conditions on the graph, we show that the algebra of the corresponding edge subshift is isomorphic to the Leavitt path algebra of the graph (see Proposition \ref{LPA}). We also consider the case of Leavitt path algebras of ultragraphs (see Proposition \ref{LPAU}).

We now give a detailed overview of the paper. In Section~\ref{eleicoes}, we provide the reader with preliminary definitions and auxiliary results used throughout. In Subsection~\ref{symbolic}, we present basic elements of symbolic dynamics and define subshifts over an arbitrary alphabet. Following~\cite{OTW}, we recall the definition of OTW subshifts in Subsection~\ref{OTW}. In Subsection~\ref{s:stone}, we develop auxiliary results concerning Boolean algebras, filters, and algebras generated by idempotents. In Subsection~\ref{mariokart}, we recall the definition of a Leavitt labelled path algebras, as in~\cite{MR4583730}, which we later connect with subshift algebras.

In Section~\ref{unital}, we define the unital subshift algebra and describe some of its properties, which include a $\mathbb{Z}$-grading and its realisation as a Leavitt labelled path algebra.

As mentioned before, we define two algebras associated with a subshift. The definition of the second algebra (which is not necessarily unital) is presented in Section~\ref{s:nonunital}. Key results in this section are the descriptions, under mild conditions, of Leavitt path algebras associated with graphs or ultragraphs as subshift algebras (Propositions~\ref{LPA} and \ref{LPAU}).

In Section~\ref{skewrings}, we give two descriptions of the unital subshift algebra as a partial skew group ring. One description arises from a set-theoretic partial action and the other from a topological partial action (Theorems~\ref{thm:set-theoretic-partial-action} and \ref{thm:shift_alg_partil_skew}).

From the partial skew group ring characterisation mentioned above, in Section~\ref{s:groupoid picture}, we define a groupoid and describe unital subshifts algebras as Steinberg algebras (Theorem~\ref{Steinberg}). 

Finally, in Section~\ref{s:conjugacy}, we describe a conjugacy between OTW subshifts in terms of a homeomorphism between the Stone duals of the Boolean algebras used in the definition of the unital algebras and in terms of a diagonal-preserving isomorphism of the associated unital algebras (Theorem~\ref{theone}).

\section{Preliminaries}\label{eleicoes}

In this section, we establish notation and present some results that we require in this paper. Firstly, we fix basic terminology and notation related to symbolic dynamics. Secondly, we recall the definition of Ott-Tomforde-Willis (OTW) subshift. Following that, we revisit the Stone Duality Theorem, which is used in Section~\ref{skewrings}. We finish this section stating the definition of a Leavitt labelled path algebras, as in \cite{MR4583730}.

Throughout the paper, $R$ stands for a commutative unital ring, $\nn=\{0,1,2,\ldots\}$ and $\nn^*=\{1,2,\ldots\}$.

\subsection{Symbolic Dynamics}\label{symbolic}

Let $\alf$ be a non-empty set, called an \emph{alphabet}, and let $\sigma$ be the \emph{shift map} on $\alf^\N$, that is, $\sigma$ is the map from $\alf^\N$ to $\alf^\N$ given by $\sigma(x)=(y_n)$, where $x=(x_n)$ and  $y_n=x_{n+1}$. Elements of $\alf^*:=\bigcup_{k=0}^\infty \alf^k$ are called \emph{blocks} or \emph{words}, and $\omega$ stands for the empty word. We also set $\alf^+=\alf^*\setminus\{\eword\}$. Given $\alpha\in\alf^*\cup\alf^{\N}$, $|\alpha|$ denotes the length of $\alpha$ and for $1\leq i,j\leq |\alpha|$, we define $\alpha_{i,j}:=\alpha_i\cdots\alpha_j$ if $i\leq j$, and $\alpha_{i,j}=\eword$ if $i>j$. If moreover $\beta\in\alf^*$, then $\beta\alpha$ denotes the usual concatenation. A subset $\osf \subseteq \alf^\N$ is \emph{invariant} for $\sigma$ if $\sigma (\osf)\subseteq \osf$. For an invariant subset $\osf \subseteq \alf^\N$, we define $\CL_n(\osf)$ as the set of all words of length $n$ that appear in some sequence of $\osf$, that is, $$\CL_n(\osf):=\{(a_0\ldots a_{n-1})\in \alf^n:\ \exists \ x\in \osf \text{ s.t. } (x_0\ldots x_{n-1})=(a_0\ldots a_{n-1})\}.$$ Clearly, $\CL_n(\alf^\N)=\alf^n$ and we always have that $\CL_0(\osf)=\{\omega\}$.
The \emph{language} of $\osf$ is the set $\lang$, which consists of all finite words that appear in some sequence of $\osf$, that is,
$$\lang:=\bigcup_{n=0}^\infty\CL_n(\osf).$$

Given $F\subseteq \alf^*$, we define the \emph{subshift} $\osf_F\subseteq \alf^\N$ as the set of all sequences $x$ in $\alf^\N$ such that no word of $x$ belongs to $F$. Usually, the set $F$ will not play a role, so we will say $\osf$ is a subshift with the implication that $\osf=\osf_F$ for some $F$. We also point out that subshifts are also called shift spaces in the literature.

Next, we define the key sets that will be used in the definition of the algebra associated with a subshift and describe some of their properties.

\begin{definition}\label{Arapaima}
Let $\osf$ be a subshift for an alphabet $\alf$. Given $\alpha,\beta\in \lang$, define \[C(\alpha,\beta):=\{\beta x\in\osf:\alpha x\in\osf\}.\] In particular, we denote $C(\eword,\beta)$ by $Z_{\beta}$ and call it a \emph{cylinder set}. Moreover, we denote $C(\alpha,\eword)$ by $F_{\alpha}$ and call it a \emph{follower set}. Notice that $\osf=C(\eword,\eword)$.
\end{definition}

\subsection{Ott-Tomforde-Willis subshifts}\label{OTW}

In this subsection, we briefly recall the construction of subshifts over an arbitrary alphabet as done in \cite{OTW}. 

If $\alf$ is a finite alphabet, we define $\Sigma_\alf=\alf^\N$. If $\alf$ is an infinite alphabet, define a new symbol $\infty$, not in $\alf$, let $\tilde \alf:=\alf\cup\{\infty\}$ and 
\[
\Sigma_\alf = \{(x_i)_{i\in \N}\in \tilde \alf^\N: x_i = \infty \text{ implies }x_{i+1}= \infty \}.
\]
In both cases,
\[
\Sigma_\alf^{\text{inf}} = \alf^\N, \text{ and }\Sigma_\alf^{\text{fin}} = \Sigma_\alf\setminus\Sigma_\alf^{\text{inf}}.
\]
When the alphabet is infinite, the set $\Sigma_\alf^{\text{fin}}$ is identified with the finite sequences in $\alf$ via the identification
\begin{equation}\label{xfin}
(x_0x_1\dots x_k\infty\infty\infty\dots)\equiv (x_0x_1\ldots x_k).    
\end{equation}

The sequence $(\infty \infty\infty\ldots)$ is denoted by $\vec{0}$ and is called the {\em empty sequence}.

Next, we recall the construction of OTW-subshifts (OTW stands for Ott-Tomforde-Willis).

\begin{definition}
Let $F \subseteq  \alf^*$. We define
\begin{align*}
\osf^{inf}_F &:= \{ x \in \alf^\N : \text{ no block of $x$ is in $F$} \} \\
\osf^{fin}_F &:= \{ x \in \Sigma_\alf^{\text{fin}} : \text{ there are infinitely many $a \in \alf$ for which}\\
& \qquad \qquad \qquad \qquad \text{there exists $y \in \alf^\N$ such that $xay \in \osf^{inf}_F$}\}. \\
\end{align*}
The \emph{OTW-subshift associated with $F$} is defined as $\osf_F^{OTW}:= \osf^{inf}_F \cup \osf^{fin}_F$. The shift map $\sigma:\osf_F^{OTW}\to \osf_F^{OTW}$ is defined as
\[\sigma(x)=\begin{cases}
x_1x_2\ldots, & \text{if }x=x_0x_1x_2\ldots\in\osf_F^{inf} \\
x_1\ldots x_{n-1}, & \text{if }x=x_0\ldots x_{n-1}\in\osf_F^{fin}\text{ and }n\geq 2 \\
\vec{0}, & \text{if }x=x_0\in\osf_F^{fin}\text{ or }x=\vec{0}.
\end{cases}
\]
In the case that $F=\emptyset$, we have that $\osf_F^{OTW}=\Sigma_\alf$, which we call the OTW full shift.
\end{definition}

\begin{remark}\label{computer}
Notice that $\osf^{inf}_F$ coincides with the subshift associated with $F$ in Subsection~\ref{symbolic}. As in the case of a subshift, we omit the subscript $F$ and write $\osf^{OTW}$ for $\osf^{OTW}_F$.
Similarly to what is done in Section \ref{symbolic}, we can define the language of an OTW-subshift, denoted by $\mathcal{L}_{\osf^{OTW}}$. Note that $\lang$ and $\mathcal{L}_{\osf^{OTW}}$ are the same and therefore there is no ambiguity in writing $\alpha\in\lang$ when working in $\osf^{OTW}$. Moreover, using the identification given by \eqref{xfin}, we may view $\osf^{fin}$ as a subset of $\lang$.
\end{remark}

We also need the notion of follower and generalised cylinder sets for OTW-subshifts.

\begin{definition}\label{peixe}
Let $\osf^{OTW}$ be an OTW-subshift and $\alpha\in \lang$. The \emph{follower set} of $\alpha$, denoted by $\CF(\alpha)$, is defined as the set \[\CF(\alpha):=\{y\in \osf^{OTW}: \alpha y \in \osf^{OTW}\}. \] For a finite set $F\subseteq \alf$ and $\alpha\in  \lang$, we define the \emph{generalised cylinder set} as
\[
\CZ(\alpha,F):=\{y\in\osf^{OTW}:\ y_i=\alpha_i\ \forall \, 1\leq i\leq |\alpha|,\ y_{|\alpha|+1}\notin F\}.
\]
To simplify notation, we denote  $\CZ(\alpha,\emptyset)$ by $\CZ(\alpha)$.
\end{definition}

Endowed with the topology generated by the generalised cylinders, $\osf^{OTW}$ is a compact totally disconnected Hausdorff space and, in this topology, the generalised cylinders are compact and open. We note that the shift map is continuous everywhere, except possibly at $\vec{0}$ (if $\vec{0}\in\osf^{OTW}$). Also, if the alphabet is finite, then $\osf^{OTW}$ is the usual subshift and the topology given by the generalised cylinders is the product topology (see \cite[Remark~2.26]{OTW}).

The following two lemmas will be useful in the upcoming work.

\begin{lemma}\label{lem:backandforth}
Let $\osf^{OTW}$ be an OTW-subshift. Let $F,G\scj \alf$ be finite sets, $\alpha,\beta\in\lang$ and let $n\in\nn$ be such that $n\leq|\alpha|$. Then,\begin{enumerate}[(i)]
    \item\label{i:forward} $\sigma^n(\CZ(\alpha,F))=\CZ(\alpha_{n+1,|\alpha|},F)\cap\CF(\alpha_{1,n})$;
    \item\label{i:backward} $\CZ(\alpha,F)\cap\sigma^{-n}(\CZ(\beta,G))$ is either empty or is equal to $\CZ(\gamma,H)$ for some $\gamma\in\lang$ and $H\scj\alf$ finite.
\end{enumerate}
\end{lemma}

\begin{proof}
\eqref{i:forward} This is immediate from the definitions of the generalised cylinders and $\sigma$.

\eqref{i:backward} Let $Y=\CZ(\alpha,F)\cap\sigma^{-n}(\CZ(\beta,G))$. We divide the proof into two cases. 

First, we suppose that $n=|\alpha|$. If $\beta_1\in F$ or $\alpha\beta\notin \lang$ then $Y=\emptyset$, otherwise we have that $Y=\CZ(\alpha\beta,G)$.

Secondly, suppose that $n<|\alpha|$. If $|\beta|>|\alpha|-n$, then for $Y$ not to be empty we must have that $\beta_{1,|\alpha|-n}=\alpha_{n+1,|\alpha|}$ and $\beta_{|\alpha|-n+1}\notin F$. In this case, we have that $Y=\CZ(\alpha_{1,n}\beta,G)$. If $|\beta|=|\alpha|-n$ and $Y$ is not empty, then $\beta=\alpha_{n+1,|\alpha|}$, in which case, $Y=\CZ(\alpha,F\cup G)$. Finally, if $|\beta|<|\alpha|-n$ and $Y$ is not empty, then $\beta=\alpha_{n+1,n+|\beta|}$ and $\alpha_{n+|\beta|+1}\notin G$. In this case $Y=\CZ(\alpha,F)$.
\end{proof}

In the next lemma, we compare the cylinder and follower sets in the OTW subshift $\osf^{OTW}$, denoted by $\CZ$ and $\CF$, as in Definition~\ref{peixe} with the cylinder and follower sets in the subshift $\osf$, denoted by $Z$ and $F$, as in Definition~\ref{Arapaima}.

\begin{lemma}\label{l:CFZ}
Let $\osf^{OTW}$ be an OTW-subshift. Let $F\scj \alf$ be a finite set and $\alpha,\beta\in\lang$. Then
\begin{enumerate}[(i)]
    \item $\CZ(\alpha,F)\cap\osf^{inf}=Z_{\alpha}\setminus\left(\bigcup_{a\in F}Z_{
    \alpha a}\right)$;
    \item $\CF(\alpha)\cap \osf^{inf}=F_{\alpha}$;
    \item $\CZ(\beta)\cap\sigma^{-|\beta|}(\CF(\alpha))\cap\osf^{inf}=C(\alpha,\beta)$.
\end{enumerate}
\end{lemma}

\begin{proof}
The proof is straightforward and it is left to the reader.
\end{proof}

\subsection{Stone duality}\label{s:stone}

We define a \emph{Boolean algebra} as a distributive lattice $\CB$ with least element and which is relatively complemented. We do not assume that $\CB$ has a maximum element and denote the join and meet by $\cup$ and $\cap$, respectively. When $\CB$ has a maximum element, we say that $\CB$ is a \emph{unital Boolean algebra}.

A \emph{filter} is a non-empty proper subset $\xi$ of $\CB$ such that $A\cap B\in\xi$ for all $A,B\in\xi$, and if $A,B\in\CB$ are such that $A\scj B$ and $A\in\xi$, then $B\in\xi$. An \emph{ultrafilter} is a filter that is maximal with respect to inclusions. In a Boolean algebra, an ultrafilter is equivalently defined as being a \emph{prime filter}, that is, a filter $\xi$ such that if $A,B\in\CB$ are such that $A\cup B\in\xi$, then $A\in\xi$ or $B\in\xi$.

The set of all ultrafilters of $\CB$ will be denoted by $\HCB$. For each $A\in\CB$, we define $O_A:=\{\xi\in\HCB:A\in\xi\}$. We have that $O_{A\cap B}=O_A\cap O_B$, $O_{A\cup B}=O_A\cup O_B$ and $O_{A\setminus B}=O_A\setminus O_B$ for every $A,B\in\CB$. In particular, the family $\{O_A\}_{A\in\CB}$ is a basis for a topology on $\HCB$. We will always assume that $\HCB$ has this topology, and we call $\HCB$ the \emph{Stone dual} of $\CB$.

In what follows, for a topological space $X$, we denote by $\CK(X)$ the set of all compact-opens subsets of $X$. For convenience, we state below the well known Stone duality theorem.

\begin{theorem}[Stone duality]\label{thm:stone}
    Let $\CB$ be a Boolean algebra. Then, $\HCB$ is a Hausdorff space and $\{O_A\}_{A\in\CB}$ is the set of all compact-open sets of $\CB$. Moreover, the map $A\in\CB\mapsto O_A\in\CK(\HCB)$ is an isomorphism of Boolean algebras. Reciprocally, if $X$ is a Hausdorff space such that $\CK(X)$ is a basis for its topology, then the map that sends $x\in X$ to the ultrafilter $\xi_x=\{A\in\CK(X):x\in A\}$ is a homeomorphism between $X$ and $\widehat{\CK(X)}$, with inverse given by $\xi\in\widehat{\CK(X)}\mapsto \cap_{A\in\xi}A\in X$. Moreover, if $\Phi:\CB_1\to\CB_2$ is an isomorphism of Boolean algebras, then the map $\widehat{\Phi}:\widehat{\CB_2}\to\widehat{\CB_1}$, given by $\widehat{\Phi}(\xi)=\phi^{-1}(\xi)$, is a homeomorphism.
\end{theorem}

A Hausdorff space with a basis of compact-open sets will be called a \emph{Stone space} (it is sometimes called a \emph{Boolean space}). For a Stone space $X$, we let $\Lc(X,R)$ denote the $R$-algebra of compactly supported locally constant functions, with pointwise operations. For a non-zero element $f\in\Lc(X,R)$, its image is a finite set. Let $\{r_1,\ldots,r_n\}$ be the set of non-zero elements in the image of $f$. Then, for each $i=1,\ldots,n$, we have that $A_i=f^{-1}(r_i)$ is compact-open. Moreover, $f=\sum_{i=1}^n r_i1_{A_i}$, where $1_{A_i}$ represents the characteristic function of $A_i$.

Suppose that $\CA$ is a commutative $R$-algebra. The set $E(\CA)$ of idempotents of $\CA$ is a Boolean algebra, where the order is given by $e\leq f$ if $ef=e$ for $e,f\in E(\CA)$. In this case $e\cap f = ef$ and $e \cup f = e + f - ef$ for $e,f\in E(\CA)$. The following theorem is essentially due to Keimel \cite{Keimel}.

\begin{theorem}\label{IsoLc}
    Let $\CA$ be a commutative $R$-algebra generated by its idempotent elements. Suppose that for $r\in R$ and $e\in E(\CA)$, $re=0$ implies that $r=0$ or $e=0$. Then, there is an isomorphism $\Phi:\CA\to\Lc(\widehat{E(\CA)},R)$ of $R$-algebras such that $\Phi(e)=1_{O_e}$ for every $e\in E(\CA)$.
\end{theorem}

If $R$ and $\CA$ satisfy the hypothesis of Theorem~\ref{IsoLc} and $\CA\neq 0$, then $R$ must be an indecomposable ring, that is, $0$ and $1$ are its only idempotents. Indeed if $r\in E(R)\setminus\{0\}$ and $e\in E(\CA)\setminus\{0\}$, then $(1-r)re=0$, from where it follows that $r=1$.

Sometimes, it is useful to work with another topological space in place of $\widehat{E(\CA)}$. We define the \emph{character space} of $\CA$ as the set $\widehat{\CA}$ of all $R$-algebra homomorphisms from $\CA$ to $R$. The topology on $\widehat{\CA}$ is the one induced from the product $R^{\CA}$, where $R$ has the discrete topology. In other words, it is the topology of pointwise convergence. Note that a subbase for the topology on $\widehat{\CA}$ is given by the clopen sets $U_{a,r}:=\{\phi\in\widehat{\CA}:\phi(a)=r\}$ for $a\in \CA$ and $r\in R$.

\begin{proposition}\label{jamesbond}
On the conditions of Theorem~\ref{IsoLc}, $\widehat{E(\CA)}$ is homeomorphic to $\widehat{\CA}$.
\end{proposition}

\begin{proof}
If $\phi\in\widehat{\CA}$, it is straightforward to check that $\xi_{\phi}:=\{e\in E(\CA):\phi(e)=1\}$ is a prime filter and therefore an ultrafilter. We then obtain a map $\Psi:\widehat{\CA}\to\widehat{E(\CA)}$, given by $\Psi(\phi)=\xi_\phi$, which is injective because $\CA$ is generated by $E(\CA)$. For surjectivity, take $\xi\in\widehat{E(\CA)}$ and let $\phi_{\xi}:A\to R$ be given by $\phi_{\xi}(a)=\Phi(a)(\xi)$, where $\Phi$ is the isomorphism of Theorem~\ref{IsoLc}. Since $\Phi(e)=1_{O_e}$ for all $e\in E(\CA)$, we have that $\Psi(\phi_{\xi})=\xi$.

For each $e\in E(\CA)$, we have that $\Psi^{-1}(O_e)=U_{e,1}$ and hence $\Psi$ is continuous. To show that $\Psi$ is open, we first observe that $U_{0,r}$ is either the empty set (if $r\neq 0$) or $\widehat{\CA}$ (if $r=0$), and hence $\Psi(U_{0,r})$ is open for every $r\in\CA$. Now, given $a\in\CA\setminus\{0\}$, we will describe $U_{a,r}$ and $\Psi(U_{a,r})$. First, we let $\{r_1,\ldots,r_n\}$ be the distinct non-zero elements of the image of $\Phi(a)$. For each $i=1,\ldots,n$, we let $e_i\in E(\CA)$ be such that $O_{e_i}=\Phi(a)^{-1}(r_i)$. Then $a=\sum_{i=1}^n r_ie_i$ and $e_ie_j=0$ if $i\neq j$. As observed above, $R$ must be an indecomposable ring and, therefore, $\phi(e)$ is either $0$ or $1$ for each $e\in E(\CA)$ and $\phi\in\widehat{\CA}$. Hence, for $\phi\in\widehat{\CA}$, $\phi(a)=0$ or $\phi(a)=r_i$ for the unique $i$ such that $\phi(e_i)=1$. There are a few possibilities for $U_{a,r}$, which we describe next. If $r=r_i$ for some $i=1,\ldots,n$, then $U_{a,r}=U_{e_i,1}$ and $\Psi(U_{a,r})=O_{e_i}$. If $r=0$, then $U_{a,r}=\widehat{\CA}\setminus\bigcup_{i=1}^nU_{e_i,1}$ and $\Psi(U_{a,r})=\widehat{E(\CA)}\setminus\bigcup_{i=1}^nO_{e_i}$. Finally, if $r\notin\{0,r_1,\ldots,r_n\}$, then $U_{a,r}=\emptyset$ and $\Psi(U_{a,r})=\emptyset$. We have proved that $\Psi(U_{a,r})$ is open for every $a\in\CA$ and $r\in R$, and hence $\Psi$ is an open map.
\end{proof}

For indecomposable rings, we can also describe the idempotents of $\Lc(X,R)$ for a Stone space $X$.

\begin{lemma}\label{idempotents Lc}
If $R$ is also an indecomposable ring and $X$ is a Stone space, then the idempotents of $\Lc(X,R)$ must be of the form $1_A$ for some compact-open subset $A$ of $X$.
\end{lemma}

\begin{proof}
Obviously $0=1_{\emptyset}$. For a non-zero function $f\in\Lc(X,R)$, we write $f=\sum_{i=1}^nr_i1_{A_i}$, where $r_1,\ldots,r_n$ are the distinct non-zero elements in the image of $f$ and $A_i=f^{-1}(r_i)$. Suppose that $f$ is idempotent. Then, $f=f^2=\sum_{i=1}^nr_i^21_{A_i}$, which implies that $r_i^2=r_i$ for each $i$. Since $R$ is indecomposable and $r_i\neq 0$ for all $i$, we have that $n=1$ and $r_1=1$. Hence $f=1_{A_1}$.
\end{proof}

For a homeomorphism $h:X\to Y$ between Stone spaces, it is easy to see that $\phi:\Lc(Y,R)\to\Lc(X,R)$, given by $\phi(f)=f\circ h$, is an isomorphism of $R$-algebras, with inverse given by $\phi(g)=g\circ h^{-1}$. We now prove a converse for this result in the case of indecomposable rings. First, we need a lemma.

\begin{lemma}\label{Dual Lc(X,R)}
Let $X$ be a Stone space, suppose that $R$ is also an indecomposable ring and let $\CA=\Lc(X,R)$. Then, the map that sends $x\in X$ to $\{1_A:A\in\CK(X),x\in A\}\in\widehat{E(\CA)}$ is a homeomorphism.
\end{lemma}

\begin{proof}
Consider the map $\phi:\CK(X)\to E(\CA)$ given by $\phi(A)=1_A$. Clearly, it is an injective homomorphism of Boolean algebras. By Lemma~\ref{idempotents Lc}, this map is also surjective. By the Stone duality, we have homeomorphisms $\psi:X\to\widehat{\CK(X)}$ given by $\psi(x)=\{A\in\CK(X):x\in A\}$ and $\widehat{\phi^{-1}}:\widehat{\CK(X)}\to \widehat{E(\CA)}$ given by $\widehat{\phi^{-1}}(\xi)=\phi(\xi)$. Composing these two maps we conclude that the map that sends $x\in X$ to $\{1_A:A\in\CK(X),x\in A\}\in\widehat{E(\CA)}$ is a homeomorphism.
\end{proof}

\begin{proposition}\label{givemeanh}
    Let $X$ and $Y$ be Stone spaces and suppose that $R$ is also an indecomposable ring. If $\phi:\Lc(Y,R)\to\Lc(X,R)$ is an isomorphism of $R$-algebras, then there exists a homeomorphism $h:X\to Y$, such that $\phi(f)=f\circ h$ for all $f\in\Lc(Y,R)$.
\end{proposition}

\begin{proof}
Let $\CA_X=\Lc(X,R)$ and $\CA_Y=\Lc(Y,R)$. We note that $\phi$ restricts to an isomorphism of Boolean algebras between $E(\CA_Y)$ and $E(\CA_X)$. We let $\widehat{\phi}:\widehat{E(\CA_X)}\to \widehat{E(\CA_Y)}$ be the corresponding homeomorphism. We also let $h_X:X\to \widehat{E(\CA_X)}$ and $h_Y:Y\to \widehat{E(\CA_Y)}$ be the homeomorphisms given by Lemma \ref{Dual Lc(X,R)}. We then define $h=h_Y^{-1}\circ\widehat{\phi}\circ h_X$, which is a homeomorphism between $X$ and $Y$.

Given $x\in X$, let us describe $h(x)$. First, we let  $\xi=\widehat{\phi}(h_X(x))=\{\phi^{-1}(1_A):A\in\CK(X),x\in A\}$. By Lemma \ref{idempotents Lc}, there is an ultrafilter $\eta$ of $\CK(Y)$ such that $\xi=\{1_B:B\in\eta\}$. Then, $h(x)=\bigcap_{B\in\eta}B$.

Let $f\in\Lc(Y,R)$. If $f=0$, then clearly $\phi(f)=f\circ h$. Suppose that $f\neq 0$ and write $f=\sum_{i=1}^nr_i1_{B_i}$, where $r_1,\ldots,r_n$ are the non-zero elements in the image of $f$ and $B_i=f^{-1}(r_i)$. For each $i$, by Lemma \ref{idempotents Lc}, we have that $\phi(1_{B_i})=1_{A_i}$ for some $A_i\in \CK(X)$. In this case, $\phi(f)=\sum_{i=1}^n r_i1_{A_i}$. It follows from the description of $h$ that $x\in A_i$ if and only if $h(x)\in B_i$, for every $x\in X$ and each $i$. Hence
\[(f\circ h)(x)=\sum_{i=1}^nr_i1_{B_i}(h(x))=\sum_{i=1}^nr_i1_{A_i}(x)=\phi(f)(x),\]
for each $x\in X$. Therefore, $\phi(f)=f\circ h$ for every $f\in\Lc(Y,R)$.
\end{proof}

\subsection{Leavitt labelled path algebras}\label{mariokart}

In this subsection, we recall the definition of Leavitt labelled path algebras, as defined in \cite{MR4583730}.

A (directed) \emph{graph} is a quadruple $(\CE^0,\CE^1,s,r)$ where $\CE^0,\CE^1$ are sets,  $s:\CE^1\to\CE^0$ and $r:\CE^1\to\CE^0$ are maps.

Given a set $\alf$, which is thought of as a set of letters, an (edge-)\emph{labelling} on a graph $\CE$ is a surjective map $\CL:\CE^1\to\alf$. We call the pair $(\CE,\CL)$ a \emph{labelled graph}.	A \emph{path} $\lambda$ on $\CE$ is a sequence (finite or infinite) of edges $\lambda=\lambda_1\ldots\lambda_n(\ldots)$ such that $r(\lambda_i)=s(\lambda_{i+1})$ $\forall i$.  We can extend the map $\CL$ to any path $\lambda$ by $\CL(\lambda)=\CL(\lambda_1)\ldots\CL(\lambda_n)(\ldots)$. An element $\alpha=\CL(\lambda)$ is called a \emph{labelled path}. We also include the empty word $\eword$ as a labelled path. For $A\scj\CE^0$, we define {$\CL(A\CE^1)=\{\CL(e): e\in\CE^1,\ s(e)\in A\}$}.

For $\alpha\in\alf^*$ and $A\in\powerset{\CE^0}$ (where $\powerset{\CE^0}$ denotes the powerset of $\CE^0)$, the \emph{relative range of $\alpha$ with respect to $A$} is
{\[r(A,\alpha)=\{r(\lambda):\lambda\in\CE^{\ast},\ \CL(\lambda)=\alpha,\ s(\lambda)\in A\}\]}
if $\alpha\in\alf^+$ and $r(A,\alpha)=A$ if $\alpha=\eword$. We define $r(\alpha):=r(\CE^0,\alpha)$.  Note that $r(A,a)\neq\emptyset$ if and only if $a\in\CL(A\CE^1)$.

\begin{definition}
	A \emph{normal labelled space} is a triple $\lspace$, where $(\CE,\CL)$ is a labelled graph and $\CB\subseteq\powerset{\CE^0}$ is a Boolean algebra such that
	\begin{itemize}
		\item $r(\alpha)\in\CB$ and $r(A,\alpha)\in\CB$ for all $\alpha\in\alf^+$ and $A\in\CB$,
		\item {$r(A\cap B,\alpha)=r(A,\alpha)\cap r(B,\alpha)$} for all $\alpha\in\alf^+$ and $A,B\in\CB$.
	\end{itemize} 

	We say that $A\in\CB$ is \emph{regular} if for all $B\in\CB\setminus\{\emptyset\}$ such that $B\scj A$, we have that $0<|\CL(B\CE^1)|<\infty$. The set of regular sets is denoted by \emph{$\CB_{reg}$}. 
\end{definition}

\begin{definition}\label{def:LPA}
Let $\lspace$ be a normal labelled space. The \emph{Leavitt labelled path algebra associated with $\lspace$ with coefficients in $R$}, denoted by $L_R\lspace$, is the universal $R$-algebra with generators $\{p_A : A\in \CB\}$ and $\{s_a,s_a^* : a\in\alf\}$ subject to the relations
\begin{enumerate}[(i)]
	\item $p_{A\cap B}=p_Ap_B$, $p_{A\cup B}=p_A+p_B-p_{A\cap B}$ and $p_{\emptyset}=0$, for every $A,B\in\CB$;
	\item $p_As_a=s_ap_{r(A,a)}$ and $s_a^{*}p_A=p_{r(A,a)}s_a^{*}$, for every $A\in\CB$ and $a\in\alf$;
	\item $s_a^*s_a=p_{r(a)}$ and $s_b^*s_a=0$ if $b\neq a$, for every $a,b\in\alf$;
	\item $s_as_a^*s_a=s_a$ and $s_a^{*}s_as_a^{*}=s_a^{*}$ for every $a\in\alf$;
	\item For every $A\in\CB_{reg}$,
	\[p_A=\sum_{a\in\CL(A\CE^1)}s_ap_{r(A,a)}s_a^*.\]
\end{enumerate}
\end{definition}

\section{Unital algebras of subshifts}\label{unital}

In this section, we define and study a unital algebra associated with a general subshift $\osf$. To define this algebra, we associate a projection with each element in the Boolean algebra generated by $C(\alpha,\beta)$, $\alpha,\beta\in\lang$ (see Definition~\ref{Arapaima}).

\begin{definition}\label{diachuvoso}
Let $\osf$ be a subshift. Define $\TCB$ to be the Boolean algebra of subsets of $\osf$ generated by all $C(\alpha,\beta)$ for $\alpha,\beta\in\lang$, that is, $\TCB$ is the collections of sets obtained from finite unions, finite intersections, and complements of the sets $C(\alpha,\beta)$.
\end{definition}

\begin{remark}\label{U_elements}
Notice that each element of $\TCB$ is a finite union of elements of the form \[C(\alpha_1,\beta_1)\cap\ldots\cap C(\alpha_n,\beta_n)  \cap C(\mu_1,\nu_1)^c\cap \ldots \cap C(\mu_m,\nu_m)^c.\]
\end{remark}

Next, we define the unital $R$-algebra associated with the subshift $\osf$.

\begin{definition}\label{gelado}
Let $\osf$ be a subshift. We define the \emph{unital subshift algebra} $\ualgshift$ as the universal unital $R$-algebra  with generators $\{p_A: A\in\TCB\}$ and $\{s_a,s_a^*: a\in\alf\}$, subject to the relations:
\begin{enumerate}[(i)]
    \item $p_{\osf}=1$, $p_{A\cap B}=p_Ap_B$, $p_{A\cup B}=p_A+p_B-p_{A\cap B}$ and $p_{\emptyset}=0$, for every $A,B\in\TCB$;
    \item $s_as_a^*s_a=s_a$ and $s_a^*s_as_a^*=s_a^*$ for all $a\in\alf$;
    \item $s_{\beta}s^*_{\alpha}s_{\alpha}s^*_{\beta}=p_{C(\alpha,\beta)}$ for all $\alpha,\beta\in\lang$, where $s_{\eword}:=1$ and, for $\alpha=\alpha_1\ldots\alpha_n\in\lang$, $s_\alpha:=s_{\alpha_1}\cdots s_{\alpha_n}$ and $s_\alpha^*:=s_{\alpha_n}^*\cdots s_{\alpha_1}^*$.
\end{enumerate}
\end{definition}

 \begin{remark}\label{sera_frio?}
From Item (iii) in Definition~\ref{gelado}, taking $\beta = \omega$, we obtain that $s_\alpha^* s_\alpha = p_{C(\alpha,\omega)}=p_{F_\alpha}$, for all $\alpha \in \lang$. Taking $\alpha = \omega$, we obtain that $s_\beta s_\beta^* = p_{C(\omega,\beta)}= p_{Z_\beta}$ for all $\beta \in \lang$.
\end{remark}

\begin{remark}\label{muito_frio} A unital C*-algebra associated with a subshift $\osf$ can be defined in the same manner as we defined $\ualgshift$, replacing the sentence ``universal unital $R-$algebra'' with ``universal unital C*-algebra''. Such a C*-algebra generalises, to the infinite alphabet case, the C*-algebra associated with a subshift over a finite alphabet defined by Carlsen \cite{CarlsenShift}.  Most of the analysis we do regarding $\ualgshift$ in this paper passes on to the C*-algebraic version, $C^*(\osf)$, but usually not automatically. We intend to study $C^*(\osf)$ (and its non-unital version) in a follow-up paper.
\end{remark}

In the next result, we describe multiplicative properties of elements of  $\ualgshift$.

\begin{proposition}\label{lemma.algebra.unital} Let $\osf$ be a subshift and $\ualgshift$ its associated unital subshift algebra. Then,
\begin{enumerate}[(i)]
    \item $s_a^*s_b= \delta_{a,b} p_{F_a}$, for all $a,b\in \alf$;
    \item $s_\alpha^*s_\alpha$ and $s_\beta^*s_\beta$ commute for all $\alpha,\beta\in\lang$;
    \item $s_\alpha^*s_\alpha$ and $s_\beta s_\beta^*$ commute for all $\alpha,\beta\in\lang$;
    \item $s_\alpha s_\beta=0$ for all $\alpha,\beta\in\lang$ such that $\alpha\beta\notin\lang$;
    \item $\ualgshift$ is generated as an $R$-algebra by the set $\{s_a, s_a^*: a\in\alf\}\cup\{1\}$.
\end{enumerate}
\end{proposition}

\begin{proof}
For (i), notice that, from Remark~\ref{sera_frio?} and Item~(ii) in Definition~\ref{gelado}, we can write $s_a^* s_b= s_a^*p_{Z_a}p_{Z_b} s_b = s_a^* p_{Z_a\cap Z_b} s_b$, from where the result follows.

Items (ii) and (iii) follow from Remark~\ref{sera_frio?} and Item~(i) in Definition~\ref{gelado}.

For (iv), let $\alpha,\beta\in\lang$ be such that $\alpha\beta\notin\lang$. In this case $F_{\alpha}\cap Z_{\beta}=\emptyset$. Then
\[s_\alpha s_\beta=s_\alpha s_\alpha^*s_\alpha s_\beta s_\beta^*s_\beta=s_\alpha p_{F_\alpha}p_{Z_\beta}s_\beta=s_\alpha p_{F_\alpha\cap Z_\beta}s_\beta=0.\]

For (v), notice that the projection associated to every generator of the Boolean  algebra $\TCB$ can be written as  $p_{C(\alpha,\beta)} = s_{\beta}s^*_{\alpha}s_{\alpha}s^*_{\beta}$ (from Item~(iii) in Definition~\ref{gelado}), which implies that it belongs to the $R$-algebra generated by the set $\{s_a, s_a^* : a\in\alf\}$. The result now follows from Item~(i) in Definition~\ref{gelado}.
\end{proof}

The grading of a combinatorial algebra (such as the Leavitt path algebra of a graph, ultragraph, labelled graph, etc.) plays a key role in its study; see \cite{MR4634309, Hazrat98, Vas} for a few examples. Below we recall the notion of grading and describe a $\mathbb {Z}$-grading of $\ualgshift$.

\begin{definition}\label{dfn:grading}
A \textit{$\mathbb{Z}$-graded ring} is a ring $S$ with a collection of additive subgroups $\{S_n\}_{n\in \mathbb{Z}}$ of $S$ such that 
    \begin{enumerate}
        \item $S=\bigoplus_{n\in \mathbb{Z}}S_n$, and 
        \item $S_mS_n \scj S_{m+n}$ for all $m,n\in \mathbb{Z}$.
    \end{enumerate}
The subgroup $S_n$ is called the \textit{homogeneous component of $S$ of degree $n$}, and the collection $\{S_n\}_{n\in \mathbb{Z}}$ is called a \emph{$\mathbb{Z}$-grading} of $S$.
\end{definition}

\begin{definition}
If $S$ is a $\mathbb{Z}$-graded ring, then an ideal $I\scj S$ is a \textit{$\mathbb{Z}$-graded ideal} if $I=\bigoplus_{n\in \mathbb{Z}} (I\cap S_n)$. If $\phi:S\to T$ is a ring homomorphism between   $\mathbb{Z}$-graded rings, then $\phi$ is \textit{$\mathbb{Z}$-graded homomorphism} if $\phi(S_n)\scj  T_n$ for every $n\in \mathbb{Z}$.
\end{definition}

\begin{proposition}\label{prop:grading} Let $\osf$ be a subshift. The unital subshift algebra $\ualgshift$ is $\zn$-graded, with grading given by
    \[\ualgshift_n = \vecspan_R\{s_\alpha p_A s_\beta^* : \alpha,\beta \in \lang,\ A\in\TCB \ \mbox{and} \ |\alpha|-|\beta|=n\}.\]
\end{proposition}
\begin{proof}
The proof is routine and we omit it (see \cite[Corollary 2.1.5]{AbrAraMol} or \cite[Proposition 3.8]{MR4583730} for examples of the techniques employed).
\end{proof}

As in \cite{MR3614028}, we define a normal labelled space $\tlspace$ associated with a subshift $\osf$ as follows: the graph $\CE$ is given by $\CE^0=\osf$, $\CE^1=\{(x,a,y)\in\osf\times\alf\times\osf: x=ay\}$, $s(x,a,y)=x$ and $r(x,a,y)=y$. The labelling map is given by $\CL(x,a,y)=a$, and the accommodating family $\TCB$ is the Boolean algebra defined above. Then the triple $\tlspace$ is a normal labelled space \cite[Lemma 5.5]{MR3614028}.

We notice that the above graph has no sinks and for $A\in \TCB$ we have $\CL(A\CE^1)=\{a\in\alf: Z_a\cap A\neq\emptyset\}$. This implies that $A\in\TCB_{reg}$ if and only if $Z_a\cap A\neq\emptyset$ for finitely many $a\in\alf$. In particular, if the alphabet is finite, then all elements of $\TCB$ are regular. Also, since $\osf=\bigsqcup_{a\in\alf}Z_a$, we have that if $A\in\TCB$, then $A=\bigsqcup_{a\in\CL(A\CE^1)}Z_a\cap A$.

For $\alpha,\beta\in\lang$ such that $\beta\neq\eword$ and $a\in\alf$, we have from  \cite[Equation (6)]{MR3614028} that the relative range for sets $C(\alpha,\beta)$ is given by 

\begin{equation}\label{eq:rel.range}
r(C(\alpha,\beta),a)=\begin{cases}
C(a,\eword)\cap C(\alpha,\beta_2\ldots\beta_{|\beta|}) & \text{if } \beta=a\beta_2,\ldots,\beta_{|\beta|}, \\
\emptyset & \text{otherwise}.
\end{cases}    
\end{equation}

Also,
\[r(C(\alpha,\eword),a)=\begin{cases}
C(\alpha a,\eword) & \text{if }\alpha a\in \lang, \\
\emptyset & \text{if }\alpha a\notin \lang.
\end{cases}\]
More generally, \[r(A,\alpha)=\{x\in\osf: \alpha x\in A\}.\] This last equality implies that $r(\alpha)=r(\osf,\alpha)=F_{\alpha}=C(\alpha,\eword)$.
Furthermore, the following holds. 

\begin{lemma}\label{aguaquecida}
Let $\osf$ be a subshift, $A_i\in \TCB$, $i=1,2$, and $a\in \alf$. Then,
\begin{enumerate}[(a)]
 \item $r(A_1,a)\cap r(A_2,a)=r(A_1\cap A_2,a)$;
 \item $r(A_1,a)\cup r(A_2,a)=r(A_1\cup A_2,a)$.
\end{enumerate}

\end{lemma}
\begin{proof}
The proof is straightforward.
\end{proof}

\begin{theorem}\label{melaoflutuante}
Let $\osf$ be a subshift and $\tlspace$ be the labelled space defined above. Then, $\ualgshift\cong L_R\tlspace$ as $\zn$-graded $R$-algebras. 
\end{theorem}

\begin{proof}
We use $t_a$ and $q_A$ for the generators of $L_R\tlspace$. Let us first a build a map from $L_R\tlspace$ to $\ualgshift$ by showing that the elements $\{p_A: A\in\TCB\}$ and $\{s_a,s_a^*: a\in\alf\}$ satisfy the relations defining $L_R\tlspace$.

Items (i) and (iv) of the definition of a Leavitt labelled path algebra are immediate.

To check Item~(iii), notice that we have  $s_a^*s_a=p_{F_a}=p_{r(a)}$. If $a\neq b$ then
\[s_b^*s_a=s_b^*s_bs_b^*s_as_a^*s_a=s_b^*p_{Z_b}p_{Z_a}s_a=s_b^*p_{Z_b\cap Z_a}s_a=s_b^*p_{\emptyset}s_a=0.\]

Next, we consider Item~(ii). We prove it first for elements of $\TCB$ of the form $C(\alpha,\beta)$. Suppose first that $\beta=\eword$. Then,
\[    p_{C(\alpha,\eword)}s_a =s_{\alpha}^*s_\alpha s_a=s_{\alpha}^*s_\alpha s_as_a^*s_a=s_as_a^*s_{\alpha}^*s_\alpha s_a=s_ap_{r(C(\alpha,\eword),a)},
\]
where the last equality holds if $\alpha a\in\lang$ by definition, and if $\alpha a\notin\lang$ because in this case $s_\alpha s_a=0$ and $r(C(\alpha,\eword),a)=\emptyset$. Analogously $s_a^*p_{C(\alpha,\eword)}=p_{r(C(\alpha,\eword),a)}s_a^*$. Suppose now that $\beta\neq\eword$. If $a\neq\beta_1$ then, using (iii) and the fact that $r(C(\alpha,\beta),a)=\emptyset$, we have that
\[p_{C(\alpha,\beta)}s_a=s_{\beta}s_{\alpha}^*s_\alpha s_{\beta}^* s_a=0=s_ap_{r(C(\alpha,\beta),a)}.\]
If $a=\beta_1$, then using (iii), we have that
\begin{align*}
    p_{C(\alpha,\beta)}s_a &=s_{\beta}s_{\alpha}^*s_\alpha s_{\beta}^* s_a \\
    &=s_as_{\beta_2}\cdots s_{\beta_n}s_{\alpha}^*s_\alpha s_{\beta_n}^*\cdots s_{\beta_2}^*s_a^*s_a\\
    &=s_ap_{C(\alpha,\beta_2\ldots\beta_n)}p_{C(a,\eword)} \\
    &=s_ap_{r(C(\alpha,\beta),a)}.
\end{align*}

Hence, Item~(ii) is proved for elements of $\TCB$ of the form $C(\alpha,\beta)$. Next, we argue that it is also valid for any $A\in \TCB$. If $A$ is of the form $A=\cap_{i=1}^{n}C(\alpha_i,\beta_i)$ the result follows from Item (i) in Definition~\ref{gelado}, the fact that (ii) is proved for elements of the form $C(\alpha,\beta)$ and from Item~(a) of Lemma~\ref{aguaquecida}. For elements of $\TCB$ that involve only $C(\alpha,\beta)^c$ we notice that $p_{C(\alpha,\beta)^c}s_a= (1-p_{C(\alpha,\beta)})s_a = s_a-s_ap_{r(C(\alpha,\beta),a)} = s_a p_{r(C(\alpha,\beta),a)^c}$. Finally, the intersections of terms of the form $C(\alpha,\beta)$ and $C(\alpha,\beta)^c$ are dealt with in a similar manner, and their finite unions are dealt with by applying Item (i) in Definition~\ref{gelado}, the fact that (ii) is proved for elements of the form $C(\alpha,\beta)$ and Item~(b) of Lemma~\ref{aguaquecida}.

For (v), we use (ii) and the fact that if $A\in\TCB_{reg}$ then $A$ can be decomposed as a finite disjoint union $A=\bigsqcup_{a\in\CL(A\CE^1)}A\cap Z_a$, and hence
\[p_A=\sum_{a\in\CL(A\CE^1)}p_{A\cap Z_a}=\sum_{a\in\CL(A\CE^1)}p_Ap_{Z_a}=\sum_{a\in\CL(A\CE^1)}p_As_as_a^*=\sum_{a\in\CL(A\CE^1)}s_ap_{r(A,a)}s_a^*.\]
This completes the proof that the generating sets of $\ualgshift$ satisfy the defining relations of $L_R\tlspace$. Thus, by the universal property of $L_R\tlspace$, there is a homomorphism from $L_R\tlspace$ into  $\ualgshift$ mapping $q_A$ to $p_A$, $t_a$ to $s_a$ and $t^*_a$ to $s^*_a$.

Now, let us build a map from $\ualgshift$ to $L_R\tlspace$ by showing that the elements $\{q_A: A\in\TCB\}$ and $\{t_a,t_a^*: a\in\alf\}$ satisfies the relations defining $\ualgshift$.

For (i), it remains to show that $t_{\osf}=1$. This follows from \cite[Corollary 6.5]{MR4583730} since $\osf$ is the top element of $\TCB$ and $r(\osf,a)=r(a)$ for all $a\in\alf$. Item (ii) is immediate.

For (iii), we fix $\alpha\in\lang$ and prove that $t_{\beta}t_{\alpha}^*t_{\alpha}t_{\beta}^*=q_{C(\alpha,\beta)}$ using induction on $|\beta|$. If $|\beta|=0$, then $\beta=\eword$ and $t_{\alpha}^*t_{\alpha}=q_{r(\alpha)}=q_{C(\alpha,\eword)}$. Fix $n\in\nn$ and suppose that $t_{\beta'}t_{\alpha}^*t_{\alpha}t_{\beta'}^*=q_{C(\alpha,\beta')}$ for all $\beta'\in\lang$ such that $|\beta'|=n$. Let $\beta\in\lang$ be such that $|\beta|=n+1$. If $C(\alpha,\beta)\neq\emptyset$, by Equation~\eqref{eq:rel.range}, we have that $\CL(C(\alpha,\beta)\CE^1)=\{\beta_1\}$. It follows that
\begin{align*}
    q_{C(\alpha,\beta)} &= t_{\beta_1}q_{r(C(\alpha,\beta),\beta_1)}t_{\beta_1}^*\\
    &=t_{\beta_1}q_{C(\beta_1,\eword)}q_{C(\alpha,\beta_2\ldots\beta_{n+1})}t_{\beta_1}^* \\
    &=t_{\beta_1}t_{\beta_1}^*t_{\beta_1}t_{\beta_2}\cdots t_{\beta_{n+1}}t_{\alpha}^*t_\alpha t_{\beta_{n+1}}^*\cdots t_{\beta_2}^*t_{\beta_1}^*\\
    &=t_{\beta}t_{\alpha}^*t_{\alpha}t_{\beta}^*,
\end{align*}
where in the third equality, we used the induction hypothesis. If $C(\alpha,\beta)=\emptyset$, then $F_\alpha\cap F_\beta=\emptyset$ and
\[t_{\beta}t_{\alpha}^*t_{\alpha}t_{\beta}^* =t_{\beta}t_{\beta}^*t_{\beta}t_{\alpha}^*t_{\alpha}t_{\beta}^* = t_\beta q_{F_\beta}q_{F_\alpha}t_{\beta}^*=t_\beta q_{F_\beta \cap F_\alpha}t_{\beta}^*=0=q_{C(\alpha,\beta)}.\]
By the universal property of $\ualgshift$ there is a homomorphism from $\ualgshift$ into $L_R\tlspace$ that maps $p_A$ to $q_A$, $s_a$ to $t_a$ and $s_a^*$ to $t_a^*$.

It is clear that the maps obtained above are inverses of each other and preserve the grading.
\end{proof}

\begin{remark}
When convenient, we use the above labelled space and the identification in Theorem~\ref{melaoflutuante} without further mention. In particular, we use the relations defining $L_R\lspace$ applied to elements of $\ualgshift$.
\end{remark}

\begin{remark}
    In an upcoming paper, the authors will consider the C*-algebraic setting as in Remark~~\ref{muito_frio} and connect it to the C*-algebra of the labelled space $\tlspace$ as defined by Bates and Pask in \cite{BP1}.
\end{remark}

\begin{corollary}[Graded Uniqueness Theorem] \label{GUTshift}
Let $\osf$ be a subshift. If $S$ is a $\mathbb{Z}-$graded ring and $\eta: \ualgshift\rightarrow S$ is a graded ring homomorphism with $\eta(rp_A)\neq 0$ for all non-empty $A\in \TCB$ and all non-zero $r \in R$, then $\eta$ is injective.
\end{corollary}
\begin{proof}
The result follows from the isomorphism above and \cite[Corollay~5.5]{MR4583730}.
\end{proof}

\begin{corollary}\label{cachorroquente}
Let $\osf$ be subshift and $\ualgshift$ its unital subshift algebra. If $r\in R\setminus\{0\}$ and $A\in\TCB\setminus\{\emptyset\}$, then $rp_A\neq 0$.
\end{corollary}

\begin{proof}
Notice that the isomorphism of Theorem~\ref{melaoflutuante} takes $rp_A$ in $\ualgshift$ to $rp_A$ in $L_R\tlspace$, which is non-zero by \cite[Lemma~4.12]{MR4583730}.
\end{proof}

The following proposition relates the unital subshift algebra $\ualgshift$ of a subshift $\osf_F$ with the OTW-subshift $\osf^{OTW}_F$.

\begin{proposition}\label{algebra OTW}
Let $\osf\subseteq \alf^\N$ be the subshift given by a set of forbidden words $F$, and let $\osf^{OTW}_F$ be the corresponding OTW-subshift. Suppose that $R$ is also an indecomposable ring. Then,  $\vecspan_R\{s_\alpha s_\alpha^*:\alpha \in \lang\}\cong\Lc(\osf^{OTW}_F,R)$.
\end{proposition}

\begin{proof}
Denote $\vecspan_R\{s_\alpha s_\alpha^*:\alpha \in \lang\}$ by $A$, which is unital because $1=s_{\eword}s_{\eword}^*$. By Theorem~\ref{IsoLc} and Proposition \ref{jamesbond}, $A$ is isomorphic to $\Lc(\hat{A}, R)$, where $\hat{A}$ denotes the character space of $A$. Note that $\hat{A}$ is compact because $A$ is unital. We will define a map $\Psi$ from $\hat{A}$ to $\osf^{OTW}_F$, but before we prove some properties of elements $\phi$ in $\hat A$. Because $R$ is indecomposable, $\phi(s_\alpha s_\alpha^*)\in \{0,1\}$, for each $\alpha \in \lang$, since each $s_\alpha s_\alpha^*$ is a idempotent. Moreover, we have the following.

{\bf Claim:} Let $\phi\in \hat A$. If $\phi(s_\alpha s_\alpha^*)=1 $ then $\phi(s_\beta s_\beta^*)=0$ for every $\beta\neq \alpha$ such that $|\beta|=|\alpha|$. Moreover, for every initial segment $\gamma $ of $\alpha$ we have that $\phi( s_\gamma s_\gamma^*)=1$.

To prove this claim, suppose that $\phi(s_\alpha s_\alpha^*)=1 $. Then, $\phi(s_\beta s_\beta^*)= \phi(s_\alpha s_\alpha^*)\phi( s_\beta s_\beta^*)=\phi(s_\alpha s_\alpha^* s_\beta s_\beta^*) =\phi(0)=0$. Next, let $\gamma$ be an initial segment of $\alpha$. Then, $1= \phi(s_\alpha s_\alpha^*)=\phi(s_\gamma s_\gamma^* s_\alpha s_\alpha^*) = \phi(s_\gamma s_\gamma^*)\phi( s_\alpha s_\alpha^*)$ and the claim is proved.

We now proceed to define $\Psi$. Let $\phi\in \hat{A}$ and define $N$ as the supremum of all natural numbers $n$ for which there exists $\alpha\in \lang$ such that $|\alpha|=n$ and $\phi(s_\alpha s_\alpha^*)=1$. Notice that $N$ is well-defined (and can be equal to $\infty$). Indeed, $\phi(s_\omega s_\omega^*)=1$, since otherwise, if $\phi(s_\omega s_\omega^*)=0$, then $\phi=0$. 

Suppose that $N=\infty$. In this case, we define $\Psi(\phi)=y_1y_2\ldots \in \alf^{\N}$, where for every $i\in \N$ we have that $y_1\ldots y_i = \alpha_1\ldots \alpha_i$, and $\alpha_1\ldots \alpha_i$ is the unique element of length $i$ in $\lang$ such that $\phi(s_{\alpha_1\ldots \alpha_i}s_{\alpha_1\ldots \alpha_i}^*)=1$. Using the claim, it is not hard to check that $\Psi(\phi)\in \osf^{OTW}_F$.

Next, suppose that $N<\infty$. Let $\alpha\in \lang$ be such that $|\alpha|=N$ and $\phi(s_\alpha s_\alpha^*)=1$. In this case, we define $\Psi(\phi)=\alpha$ (with the convention that $\omega=\vec{0}$, to include the case $N=0$). We have to check that $\alpha\in \osf^{fin}_F$. 

Let \[G_\alpha=\{b\in \alf: C(\alpha b, \eword) = F_{\alpha b}\neq \emptyset\}.\]
If $G_\alpha$ is finite then, since $C(\eword,\alpha)= \bigsqcup_{b\in G_\alpha} C(\eword, \alpha b)$, we have that $p_{C(\eword,\alpha)} = \sum_{b\in G_\alpha} p_{C(\eword, \alpha b)}$. So, $\phi(s_\alpha s_\alpha^*)= \sum_{b\in G_\alpha} \phi(s_{\alpha b}s_{\alpha b}^*)$ and hence there is one, and only one, $b\in G_\alpha$ such that $\phi(s_{a\alpha b}s_{\alpha b}^*) =1$. But this contradicts the maximality of $N$. Therefore $G_\alpha$ must be infinite and this implies that $\alpha\in \osf^{fin}_F$ as desired.

It remains to prove that $\Psi$ is a homeomorphism. 

If $\phi_1\neq \phi_2$ are characters, then they must differ on some generator $s_\alpha s_\alpha^*$, where $\alpha=\alpha_1\ldots \alpha_{|\alpha|}$. Without loss of generality, suppose that $\phi_1(s_\alpha s_\alpha^*)=1$. Then, for every initial segment $\beta$ of $\alpha$, we have that $\phi_1(s_\beta s_\beta^*)=1$. This implies that $\Psi(\phi_1)_i = \alpha_i$ for $1\leq i \leq |\alpha|$. On the other hand, since $\phi_2(s_\alpha s_{\alpha}^*)=0$ we have that either $\Psi(\phi_2)$ has length less than $|\alpha|$ or $\Psi(\phi_2)_{|\alpha|} \neq \alpha_{|\alpha|}$. So, $\Psi$ is injective.

To see that $\Psi$ is surjective, let $y\in \osf^{OTW}_F$. For each $\alpha \in \lang$ define $\phi^y(s_\alpha s_\alpha^*)$ as $1$ if $\alpha$ is an initial segment of $y$ and zero otherwise (again we are identifying $\omega$ with $\vec{0}$). Extend $\phi^y$ linearly to $A$. Clearly $\Psi(\phi^y)= y$, but it is necessary to check that $\phi^y$ is well-defined and multiplicative.

Suppose that $\sum_{i=1}^n \lambda_i s_{\alpha_i}s_{\alpha_i}^* = 0$. We have to prove that $\phi^y(\sum_{i=1}^n \lambda_i s_{\alpha_i}s_{\alpha_i}^*) = 0$. 

By the definition of $\phi^y$, we may assume without loss of generality that each $\alpha_i$ is an initial segment of $y$ and $\alpha_n$ has the greatest length among all of them. In this case, notice that $0=\sum_{i=1}^n \lambda_i s_{\alpha_i}s_{\alpha_i}^* s_{\alpha_n}s_{\alpha_n}^* = \left(\sum_{i=1}^n \lambda_i\right)s_{\alpha_n}s_{\alpha_n}^*$. By Corollary~\ref{cachorroquente}, we obtain that $ \sum_{i=1}^n \lambda_i=0$, and hence 
\[\phi^y\left(\sum_{i=1}^n \lambda_i s_{\alpha_i}s_{\alpha_i}^*\right) = \sum_{i=1}^n\lambda_i=0,\] as desired.

By linearity, it is enough to check the multiplicativity of $\phi^y$ on elements of the form $s_\alpha s_\alpha^*$, $\alpha\in \lang$. Notice that, for $\alpha, \beta \in \lang$, we have that \[s_\alpha s_\alpha^* s_\beta s_\beta^* = p_{C(\omega,\alpha)\cap C(\omega,\beta)}= p_{Z_\alpha \cap Z_\beta}.\] Suppose, without loss of generality, that $|\alpha|\geq |\beta|$. If the product $s_\alpha s_\alpha^* s_\beta s_\beta^*$ is zero, this means that $\beta$ is not an initial segment of $\alpha$. Hence, either $\alpha$ or $\beta$ (or both) is not an initial segment of $y$, and so   $\phi^y(s_\alpha s_\alpha^*)=0$ or $\phi^y(s_\beta s_\beta^*)=0$.
If $s_\alpha s_\alpha^* s_\beta s_\beta^*\neq 0$, then $\beta$ is an initial segment of $\alpha$ and $s_\alpha s_\alpha^* s_\beta s_\beta^*=s_\alpha s_\alpha^*$. If $\alpha$ is an initial segment of $y$, then $\beta$ is also an initial segment of $y$ and $\phi^y(s_\alpha s_\alpha^* s_\beta s_\beta^*)=\phi^y(s_\alpha s_\alpha^*)=1=\phi^y(s_\alpha s_\alpha^*)\phi^y(s_\beta s_\beta^*)$. If $\alpha$ is not an initial segment of $y$, then $\phi^y(s_\alpha s_\alpha^* s_\beta s_\beta^*)=\phi^y(s_\alpha s_\alpha^*)=0=\phi^y(s_\alpha s_\alpha^*)\phi^y(s_\beta s_\beta^*)$. Hence, the multiplicativity of $\phi^y$ follows.

Next, we prove that $\Psi $ is continuous. Suppose that $(\phi_j)_{j\in J}$ is a net converging to $\phi$ in $\hat{A}$, that is, $(\phi_j)_{j\in J}$ converges pointwise to $\phi$ (see Subsection \ref{s:stone}).

Suppose that $y:=\Psi(\phi)\in \osf_F^{inf}$. Given $k>0$, let $j_0\in J$ be such that $\phi_j(s_{y_1\ldots y_k}s_{y_1\ldots y_k}^*)=1=\phi(s_{y_1\ldots y_k}s_{y_1\ldots y_k}^*)$ for every $j\geq j_0$. So, $\Psi(\phi_j)$ agrees with $\Psi(\phi)$ in the first $k$ letters for every $j\geq j_0$ and this implies that $(\Psi(\phi_j))_{j\in J}$ converges to $\Psi(\phi)$.

We are left with the case where $\Psi(\phi)\in \osf_F^{fin}$, say $\Psi(\phi)=y_1\ldots y_m$. Let $G\subseteq \alf$ be finite. Take $j_0\in J$ such that, for every $j\geq j_0$, we have $\phi_j(s_{y_1\ldots y_m}s_{y_1\ldots y_m}^*)=\phi(s_{y_1\ldots y_m}s_{y_1\ldots y_m}^*)$ and $\phi_j(s_{y_1\ldots y_m a}s_{y_1\ldots y_ma}^*)=\phi(s_{y_1\ldots y_m a}s_{y_1\ldots y_m a}^*)$ for every $a\in G$. It follows that, for $j\geq j_0$, the first $m$ letters of $\Psi(\phi_j)$ agree with $y_1\ldots y_m $ and the $m+1$ entry of $\Psi(\phi_j)$ is not in $G$, that is, $(\Psi(\phi_j))_{j\in J}$ converges to $\Psi(\phi)$ as desired. 

Finally, since $\hat{A}$ is compact and $\osf^{OTW}_F$ is Hausdorff, it follows that $\Psi$ is a homeomorphism and the proof is finished.
\end{proof}

\begin{definition}\label{diagonalalgebra}
    Let $\osf$ be subshift and $\ualgshift$ its unital subshift algebra. The \emph{diagonal subalgebra} of $\ualgshift$ is the subalgebra $\vecspan_R\{s_\alpha p_A s_\alpha^*: A\in\TCB,\ \alpha \in \lang\}$.
\end{definition}

Next, we want to identify the diagonal subalgebra of $ \ualgshift$. For this, recall that $\widehat{\TCB}$ represents the Stone dual of the Boolean algebra $\TCB$ (see Subsection~\ref{s:stone}).

\begin{proposition}\label{ohcoconuts} Let $\osf\subseteq \alf^\N$ be a subshift and suppose that $R$ is also an indecomposable ring. Then,
$\vecspan_R\{s_\alpha p_A s_\alpha^*: A\in\TCB,\ \alpha \in \lang\}=\vecspan_R\{p_A: A\in\TCB\}\cong\Lc(\widehat{\TCB},R)$. 
\end{proposition}

\begin{proof}
We start observing that for $\alpha\in\lang$ and $A,B\in\TCB$, we have that $s_\alpha p_{A\cap B}s_\alpha^*=s_\alpha p_{A}s_\alpha^*s_\alpha p_{B}s_\alpha^*$, $s_\alpha p_{A\cup B}s_\alpha^*=s_\alpha p_{A}s_\alpha^*+s_\alpha p_{B}s_\alpha^*-s_\alpha p_{A\cap B}s_\alpha^*$ and $s_\alpha p_{A\setminus B}s_\alpha^*=s_\alpha p_{A}s_\alpha^*-s_\alpha p_{A\cap B}s_\alpha^*$. Therefore, to prove the required equality, since $\TCB$ is generated by sets of the form $C(\alpha,\beta)$, it is sufficient to prove that for $\alpha,\beta,\gamma\in\lang$, we have that $s_\alpha p_{C(\beta,\gamma)}s_\alpha^*=p_A$ for some $A\in\TCB$. Indeed
\[s_\alpha p_{C(\beta,\gamma)}s_\alpha^*=s_\alpha s_\gamma s_\beta^*s_\beta s_\gamma^* s_\alpha^*=\begin{cases}
p_{C(\beta,\alpha\gamma)}, & \text{if }\alpha\gamma\in\lang \\
p_{\emptyset}, & \text{otherwise.}
\end{cases}\]

For the isomorphism, by \cite[Example 7.8]{MR4583730}, $\Lc(\widehat{\TCB},R)$ can be seen as a Leavitt labelled path algebra, with a trivial $\mathbb{Z}$-grading. The universal property of $\Lc(\widehat{\TCB},R)$ then implies that the map that sends $1_A\in \Lc(\widehat{\TCB},R)$ to $p_A\in\ualgshift$ extends to a graded $R$-algebra homomorphism $\phi:\Lc(\widehat{\TCB},R)\to\ualgshift$, which is injective by Corollary~\ref{cachorroquente} and the Graded Uniqueness Theorem for Leavitt labelled path algebras \cite[Corollary 5.5]{MR4583730}. It is clear that the image of $\phi$ is $\vecspan_R\{p_A: A\in\TCB\}$ from where the isomorphism $\vecspan_R\{p_A: A\in\TCB\}\cong\Lc(\widehat{\TCB},R)$ follows.
\end{proof}

\section{Algebras of subshifts}\label{s:nonunital}

In this section, we define another $R$-algebra $\algshift$ associated with a subshift $\osf$, which may be non-unital. It coincides with $\ualgshift$ when it is unital, and its unitization coincides with $\ualgshift$ when it is not unital (see Proposition~\ref{jabutirapido}).

To define $\algshift$, we let $\CB$ be the Boolean algebra of subsets of $\osf$ generated by all $C(\alpha,\beta)$, for $\alpha,\beta\in\lang$ not both simultaneously equal to $\eword$. Comparing $\CB$ with $\TCB$ from Definition~\ref{diachuvoso}, the only difference is that we are removing $\osf$ as a generator for the Boolean algebra $\CB$. In some instances, $\CB$ and $\TCB$ agree. For example, when the alphabet is finite (since then $\osf=\bigcup_{a\in\alf}Z_a$), or when there is a letter such that its follower set is the whole $\osf$ (see, for instance, Example~\ref{renovando}). However, there are other instances where they are different, as in Example~\ref{theorPropfail}. 

\begin{definition}\label{nonunitalalgebra}
We define the \emph{subshift algebra} $\algshift$ as the universal  $R$-algebra with generators $\{p_A: A\in\CB\}$ and $\{s_a,s_a^*: a\in\alf\}$ subject to the relations:
\begin{enumerate}[(i)]
    \item $p_{A\cap B}=p_Ap_B$, $p_{A\cup B}=p_A+p_B-p_{A\cap B}$ and $p_{\emptyset}=0$, for every $A,B\in\CB$;
    \item $s_as_a^*s_a=s_a$ and $s_a^*s_as_a^*=s_a^*$ for all $a\in\alf$;
    \item $s_{\beta}s^*_{\alpha}s_{\alpha}s^*_{\beta}=p_{C(\alpha,\beta)}$ for all $\alpha,\beta\in\lang\setminus\{\eword\}$, where for $\alpha=\alpha_1\ldots\alpha_n\in\lang \setminus \{\omega\}$, $s_\alpha:=s_{\alpha_1}\cdots s_{\alpha_n}$ and $s_\alpha^*:=s_{\alpha_n}^*\cdots s_{\alpha_1}^*$; 
    \item $s_\alpha^* s_\alpha = p_{C(\alpha,\omega)}$ for all $\alpha \in \lang \setminus \{\omega\}$;
    \item $s_\beta s_\beta^* = p_{C(\omega,\beta)}$ for all $\beta \in\lang \setminus \{\omega\}$.
\end{enumerate}
\end{definition}

\begin{remark}
 Note that $s_\eword$ does not appear in the definition of $\algshift$. However, to ease the notational burden, we often include terms of the form $s_\alpha p_A s_\eword^*$, $s_\eword p_A s_\beta^*$ and $s_\eword p_As_\eword^*$, which should be interpreted as $s_\alpha p_A$, $p_A s_\beta^*$ and $p_A$, respectively.
\end{remark}

The following results up to Corollary~\ref{GUTshift no unit} are analogues of the unital case. Their proofs follow the same line of thought, with minor modifications. For Theorem~\ref{soumavez}, as with the unital case, if $(\CE,\CL)$ is the labelled graph defined in Section \ref{unital}, then $(\CE,\CL,\CB)$ is a normal labelled space.

\begin{proposition}\label{gen no unit}
    Let $\osf$ be a subshift. Then:
    \begin{enumerate}[(i)]
        \item $s_a^*s_b= \delta_{a,b} p_{F_a}$, for all $a,b\in \alf$;
        \item $\algshift$ is generated by $\{s_a,s_a^*:a\in\alf\}$.
    \end{enumerate}
\end{proposition}

\begin{proposition} Let $\osf$ be a subshift. The subshift algebra $\algshift$ is $\zn$-graded, with grading given by
    \[\algshift_n = \vecspan_R\{s_\alpha p_A s_\beta^* : \alpha,\beta \in \lang,\ A\in\CB \ \mbox{and} \ |\alpha|-|\beta|=n\}.\]
\end{proposition}

\begin{theorem}\label{soumavez}
Let $\osf$ be a subshift and $(\CE,\CL,\CB)$ be the labelled space as above. Then, $\algshift\cong L_R(\CE,\CL,\CB)$ as $\zn$-graded $R$-algebras. In particular, if $r\in R\setminus\{0\}$ and $A\in\CB\setminus\{\emptyset\}$, then $rp_A\neq 0$.
\end{theorem}

\begin{remark}
As with the unital case, when convenient, we use the above labelled space and the identification in Theorem~\ref{soumavez} without further mentioning it.
\end{remark}
 
\begin{corollary}[Graded Uniqueness Theorem] \label{GUTshift no unit}
Let $\osf$ be a subshift. If $S$ is a $\mathbb{Z}-$graded ring and $\eta: \algshift\rightarrow S$ is a graded ring homomorphism with $\eta(rp_A)\neq 0$ for all non-empty $A\in \CB$ and all non-zero $r \in R$, then $\eta$ is injective.
\end{corollary}

For a non-unital $R$-algebra $A$, we understand the unitization of $A$ as the $R$-algebra $A\oplus R$ with coordinate-wise addition and multiplication given by $(a,r)(b,s)=(ab+sa+rb,rs)$.

\begin{proposition}\label{jabutirapido}
Let $\osf$ be a subshift. If $\algshift$ is unital, then it is isomorphic to $\ualgshift$. If $\algshift$ is not unital, then its unitization is isomorphic to $\ualgshift$.
\end{proposition}

\begin{proof}
By the definitions of the algebras, Proposition~\ref{gen no unit}, Corollary~\ref{GUTshift no unit} and the inclusion $\CB\scj\TCB$, we have that $\algshift$ is isomorphic to the subalgebra of $\ualgshift$ generated by $\{s_a,s_a^*:a\in\alf\}$. Therefore, we do not need to make the distinction between the generators of  $\algshift$ and $\ualgshift$ and we can consider $\algshift\scj\ualgshift$. Moreover, this inclusion preserves the grading.

Suppose first that $\algshift$ is unital. By Theorem~\ref{soumavez} and \cite[Corollary~6.5]{MR4583730}, $\CB$ has a top element $I$. Suppose that $I\neq\osf$. Then there exists $x=(x_0x_1\cdots)\in\osf\setminus I$, which implies that $C(\eword,x_0)\nsubseteq I$. This is a contradiction, therefore, $I=\osf$ and $\CB=\TCB$. It follows that $\algshift=\ualgshift$.

We are left with the case in which $\algshift$ is not unital. By the universal property of $\algshift\oplus R$ as an $R$-module, there exists a linear map $\Phi:\algshift\oplus R\to\ualgshift$ given by $\Phi(a,r)=a+r1$, which is surjective by Proposition~\ref{lemma.algebra.unital}. It is straightforward to check that $\Phi$ is multiplicative. To show injectivity, suppose that $\Phi(a,r)=0$ for some $(a,r)\in\algshift\oplus R$. Then $r1=-a\in\algshift$. Since $r1$ is of degree 0 in the $\mathbb{Z}-$grading of $\ualgshift$, $-a$ is also of degree 0 (in both the $\mathbb{Z}-$gradings of $\ualgshift$ and $\algshift$). Therefore, we can write
\[-a = \sum_{j=1}^m \lambda_j p_{B_j} + \sum_{i=1}^n \gamma_i s_{\alpha_i}p_{A_i}s_{\beta_i}^*,\]
where $|\alpha_i|=|\beta_i|>0$, $\lambda_j,\gamma_i\in R\setminus\{0\}$ and $B_j, A_i\in\CB$ for each $i,j$. 
Set $A = \osf\setminus(\bigcup_{i=1}^n C(\omega,\beta_i)\cup \bigcup_{j=1}^m B_j)$ and observe that $A\neq\emptyset$, since $\algshift$ is not unital. By Definition~\ref{nonunitalalgebra}, items (i) and (v), we have
\[-ap_A = \left(\sum_{j=1}^m \lambda_j p_{B_j} + \sum_{i=1}^n \gamma_i s_{\alpha_i}p_{A_i}s_{\beta_i}^*\right)p_A = 0.\]
Then $rp_A = -ap_A=0$ and by Theorem~\ref{soumavez}, we conclude that $r=0$. Since $a=-r1$, we have $a=0$, proving that $\Phi$ is injective, as desired.
\end{proof}

\subsection{Graph algebras}
In this subsection, we show that a large class of graph algebras can be seen as subshift algebras. For the reader's convenience, we recall the definition of the Leavitt path algebra associated with a graph. 

\begin{definition}\label{def:graph_LPA}
Let $\CE=(\CE^0,\CE^1,r,s)$ be a directed graph. The \emph{Leavitt path algebra associated with $\CE$ with coefficients in $R$}, denoted by $L_R(\CE)$, is the universal $R$-algebra with generators $\{v : v\in  \CE^0\}$ and $\{e,e^* : e\in \CE^1\}$ subject to the relations
\begin{itemize}
	\item[(V)] $v v' = \delta_{v,v'} v$, for all $v\in \CE^0$;
	\item[(E1)] $s(e)e=er(e)=e$, for all $e\in \CE^1$;
	\item[(E2)] $r(e) e^*=e^*s(e)=e^* $, for all $e\in \CE^1$;
	\item[(CK1)] $e^*e'=\delta_{e,e'}r(e)$, for all $e, e'\in \CE^1$;
	\item[(CK2)] $v=\sum_{e:s(e)=v} e e^*$ for all $v\in \CE^0_{reg}$, that is, the set of vertices $v\in\CE^0$ such that $0<|s^{-1}(v)|<\infty$, called \emph{regular vertices}.
\end{itemize}
\end{definition}

For a graph $\CE$, its associated one-sided edge subshift is the set of all infinite paths, which is the subshift over the alphabet $\alf=\CE^1$ given by the family of forbidden words $\{ef\in\alf^2:r(e)\neq s(f)\}$.

\begin{proposition}\label{LPA}
Let $\CE$ be a graph with no sinks and with no vertex that is simultaneously a source and an infinite emitter. Let $\osf$ be the associated one-sided edge subshift of $\CE$. Then, $\CA_R(\osf)\cong L_R(\CE)$.
\end{proposition}

\begin{proof}
Notice that for $e\in\CE^1$, we have that $C(e,\eword)=\{x \in\osf: s(x)=r(e)\}$. Hence, for $v\in\CE^0$ that is not a source, there exists $e\in r^{-1}(v)$. In this case, $\{x\in\osf: s(x)=v\} = C(e,\eword)$.

With this in mind, we build a map from $L_R(\CE)$ to $\CA_R(\osf)$ as follows. For $e\in\CE$, we set $t_e:=s_e$ and $t_{e^*}:=s_e^*$. For $v\in\CE^0$ that is not a source, we let $e\in r^{-1}(v)$ and set $q_v:=p_{C(e,\eword)}$. This does not depend on $e$ since $C(e,\eword)$ depends only on $r(e)=v$. If $v\in\CE^0$ is a source, then it is not an infinite emitter (by hypothesis) and we define $q_v := \sum_{e\in s^{-1}(v)} s_e s_e^*$. By Items~(i) and (v) of Definition~\ref{nonunitalalgebra}, we can also write $q_v=\sum_{e\in s^{-1}(v)}p_{C(\eword,e)}=p_{\bigcup_{e\in s^{-1}(v)}C(\eword,e)}$.

We prove that the families $\{t_e,t_{e^*}: e\in\CE^1\}$ and $\{q_v: v\in\CE^0\}$ satisfy the relations of $L_R(\CE)$ in Definition~\ref{def:graph_LPA}.

(V) For $v\in\CE^0$, $q_vq_v=q_v$ because $p_A$ is idempotent for all $A\in\CB$.

Next, let $v,v'\in\CE^0$ be such that $v\neq v'$ and note that $r^{-1}(v)\cap r^{-1}(v')=\emptyset$. If $v$ and $v'$ are not sources, then, for $e\in r^{-1}(v)$ and $e'\in r^{-1}(v')$, we have that $C(e,\eword)\cap C(e',\eword)=\emptyset$, from where we get $q_vq_{v'}=0$. Suppose next that $v$ is a source and $v'$ is not a source. Let $e'\in r^{-1}( v')$. Then, $q_v q_{v'} = \sum_{e\in s^{-1}(v)} p_{C(\eword,e)} p_{C(e',\eword)} = 0$, since if $e\in s^{-1}(v)$ then $C(e',\eword)\cap C(\eword,e)=\emptyset$, otherwise there exists an element $e'ex\in\osf$, implying $v$ is not a source. For the last case, suppose that both $v$ and $v'$ are sources. Since $v\neq v'$, the result follows from Items (i) and (v) in Definition~\ref{nonunitalalgebra}.

(E1) Let $e\in\CE^1$. If $s(e)$ is a source, then it follows from Item (ii) in Definition~\ref{nonunitalalgebra} and Proposition~\ref{gen no unit} that $q_{s(e)} t_e = t_e$. Suppose that $s(e)$ is not a source. Take $f\in\CE^1$ such that $r(f)=s(e)$. Notice that $C(e,\eword)=C(fe,\eword)$. We then have that
\[q_{s(e)}t_e=p_{C(f,\eword)}s_e=s_ep_{r(C(f,\eword),e)}=s_ep_{C(fe,\eword)}=s_ep_{C(e,\eword)}=s_es_e^*s_e=s_e=t_e.\]
Since $q_{r(e)}=p_{C(e,\eword)}$, the last part of the above computation also shows that $t_eq_{r(e)}=t_e$.

(E2) This is analogous to (E1).

(CK1) For $e,f\in\CE^1$, by Proposition~\ref{gen no unit}, we have that
\[t_e^*t_f=s_e^*s_f=\delta_{e,f}p_{C(e,\eword)}=\delta_{e,f}q_{r(e)}.\]

(CK2) If $v\in\CE^0_{reg}$ is a source then (CK2) is clearly satisfied. Suppose that $v\in\CE^0_{reg}$ and $e\in r^{-1}(v)$. Then, $C(e,\eword)$ can be written as a finite disjoint union as $C(e,\eword)=\bigsqcup_{f\in s^{-1}(v)} C(\eword,f)$. Hence
\[q_v=p_{C(e,\eword)}=\sum_{f\in s^{-1}(v)}p_{C(\eword,f)}=\sum_{f\in s^{-1}(v)}s_fs_f^*=\sum_{f\in s^{-1}(v)}t_ft_f^*.\]

By the universal property of $L_R(\CE)$, we obtain an $R$-algebra homomorphism $\Phi:L_R(\CE)\to\CA_R(\osf)$ which is surjective because the set $\{s_e,s_e^*: e\in\CE^1\}$ generates $\CA_R(\osf)$ by Proposition~\ref{gen no unit}. It is easy to see that this homomorphism is $\zn$-graded. Moreover, if $r\in R\setminus\{0\}$ and $v\in\CE^0$ is not a source, then $rq_v=rp_{C(e,\eword)}$ for some $e\in r^{-1}(v)$. Because the graph has no sinks, $C(e,\eword)\neq\emptyset$, which implies that $rq_v\neq 0$ by Theorem~\ref{soumavez}. Similarly, if $r\in R\setminus\{0\}$ and $v\in\CE^0$ is a source, $rq_v=rp_{\bigcup_{e\in s^{-1}(v)}C(\eword,e)}\neq 0$. Using the Graded Uniqueness Theorem for Leavitt path algebras \cite[Theorem~5.3]{LPA_ring}, we conclude that $\Phi$ is an isomorphism.
\end{proof}

Proposition~\ref{LPA} may fail to be true if there is a vertex that is simultaneously a source and an infinite emitter.

\begin{example}\label{theorPropfail}
Consider a graph $\CE$ such that $\CE^0=\{v,w\}$, $\CE^1=\{e_n\}_{n\in\nn}\cup\{f\}$, $s(e_n)=v$ and $r(e_n)=w=s(f)=r(f)$ for all $n\in\nn$. We have that $L_R(\CE)$ is unital because $\CE^0$ is finite \cite[Section~4.2]{LPA_ring}. On the other hand, we note that for every $\alpha,\beta\in\lang$ both not simultaneously the empty word, the set $C(\alpha,\beta)$ is either empty or a singleton. This implies that $\CB$ is the family of finite subsets of $\osf$, so that $\osf\notin\CB$. By Proposition~\ref{jabutirapido}, $\algshift$ is not unital. Therefore, we cannot have that $L_R(\CE)$ is isomorphic to $\algshift$.
\end{example}

\subsection{Ultragraph algebras} In this subsection, we focus on ultragraph algebras, which include algebras associated with infinite matrices (see \cite{imanfar2017leavitt}).

\begin{definition}\label{def of ultragraph}
	An \emph{ultragraph} is a quadruple $\mathcal{G}=(G^0, \mathcal{G}^1, r,s)$ consisting of two countable sets $G^0, \mathcal{G}^1$, a map $s:\mathcal{G}^1 \to G^0$, and a map $r:\mathcal{G}^1 \to \powerset{G^0}\setminus \{\emptyset\}$, where $\powerset{G^0}$ is the power set of $G^0$.
\end{definition}

\begin{definition}\label{prop_vert_gen}
	Let $\mathcal{G}$ be an ultragraph. Define $\mathcal{G}^0$ to be the smallest subset of $\powerset{G^0}$ that contains $\{v\}$ for all $v\in G^0$, contains $r(e)$ for all $e\in \mathcal{G}^1$, contains $\emptyset$, and is closed under finite unions and finite intersections. Elements of $\mathcal G^0$ are called \emph{generalised vertices}.
\end{definition}

For an ultragraph $\CG$, its associated one-sided edge subshift, $\osf_\CG$, is the set of all infinite paths, which is the subshift over the alphabet $\alf=\CG^1$ given by the family of forbidden words $\{ef\in\alf^2:s(f)\notin r(e)\}$.

\begin{example}
We can recode the vertex subshift $\osf_A$ associated with a matrix $A$ as the edge subshift of an ultragraph. Indeed, given a matrix $A$, let $\mathcal G$ be the associated ultragraph (as defined by Tomforde in \cite{MR2050134}). Then the map $(e_i)\in \osf_\mathcal G \rightarrow (s(e_i))\in \osf_A$ is a bijective 1-step code between the subshifts, so the shift algebras associated are the same.
\end{example}

The following description of $\mathcal G^0$ will be useful.

\begin{lemma}\label{description}\cite[Lemma~2.12]{MR2050134}
If $\mathcal{G}$ is an ultragraph, then \begin{align*} \mathcal{G}^0 = \{
\bigcap_{e
\in X_1} r(e)
\cup \ldots 
\cup \bigcap_{e \in X_n} r(e) \cup F : & \ \text{$X_1,
\ldots, X_n$ are finite subsets of $\mathcal{G}^1$} \\ & \text{ and $F$
is a finite subset of $G^0$} \}.
\end{align*}  
Furthermore, $F$ may be chosen to
be disjoint from $\bigcap_{e \in X_1} r(e) \cup \ldots 
\cup \bigcap_{e \in X_n} r(e)$.
\end{lemma}

\begin{example}\label{renovando} Let $\mathcal G$ be the ultragraph associated with the renewal shift (see, for instance, \cite{MR1822107}), that is, the ultragraph with a countable set of vertices, say $G^0=\{1,2,\ldots\}$, and a countable set of edges, say $\{e_1, e_2,\ldots\}$, such that $s(e_i)=i$, for all $i$, $r(e_1)=G^0$ and $r(e_j)=j-1$ for $j>1$. This ultragraph is depicted below. Let $\osf_\CG$ be the associated subshift and note that $\TCB=\CB$ because $C(e_1,\eword)=\osf_\CG$. Hence $\CA_R(\osf_\CG)=\TCA_R(\osf_\CG)$ by Proposition~\ref{jabutirapido}.

\begin{figure}[!htb]
\begin{center}
\begin{tikzpicture}[scale=1.0]
\draw   (0,0) -- (7,0);
\node[circle, draw=black, fill=white, inner sep=1pt,minimum size=5pt] (1) at (0,0) {1};
\node[circle, draw=black, fill=white, inner sep=1pt,minimum size=5pt] (2) at (1,0) {2};
\node[circle, draw=black, fill=white, inner sep=1pt,minimum size=5pt] (3) at (2,0) {3};
\node[circle, draw=black, fill=white, inner sep=1pt,minimum size=5pt] (4) at (3,0) {4};
\node[circle, draw=black, fill=white, inner sep=1pt,minimum size=5pt] (5) at (4,0) {5};
\node[circle, draw=black, fill=white, inner sep=1pt,minimum size=5pt] (6) at (5,0) {6};
\node[circle, draw=black, fill=white, inner sep=1pt,minimum size=5pt] (7) at (6,0) {7};
\node (8) at (7.5,0) {$\cdots$};
\draw[->, >=stealth] (2)  to (1);
\draw[->, >=stealth] (3)  to (2);
\draw[->, >=stealth] (4)  to (3);
\draw[->, >=stealth] (5)  to (4);
\draw[->, >=stealth] (6)  to (5);
\draw[->, >=stealth] (7)  to (6);
\node [minimum size=10pt,inner sep=0pt,outer sep=0pt] {} edge [in=200,out=100,loop, >=stealth] (0);
\draw[->, >=stealth] (1)  to [out=90,in=90, looseness=1] (2);
\draw[->, >=stealth] (1)  to [out=90,in=90, looseness=1] (3);
\draw[->, >=stealth] (1)  to [out=90,in=90, looseness=1] (4);
\draw[->, >=stealth] (1)  to [out=90,in=90, looseness=1] (5);
\draw[->, >=stealth] (1)  to [out=90,in=90, looseness=1] (6);
\draw[->, >=stealth] (1)  to [out=90,in=90, looseness=1] (7);
\end{tikzpicture}
\end{center}
\end{figure}

\end{example}

Next, we will show that, for an ultragraph with only regular vertices (that is, vertices $v$ such that $0<|s^{-1}(v)|<\infty$), the subshift algebra associated with its one-sided edge subshift is isomorphic to the associated ultragraph Leavitt path algebra. In particular, this will include algebras associated with infinite matrices, as the ultragraph associated with an infinite matrix has only regular vertices. Before we proceed, we recall the definition of ultragraph Leavitt path algebras below; see \cite{GDD2, goncalves_royer_2019, imanfar2017leavitt}.

\begin{definition}\label{def of ultragraph algebra}
	Let $\mathcal{G}$ be an ultragraph. The Leavitt path algebra of $\mathcal{G}$, denoted by $L_R(\mathcal{G})$, is the universal $R$-algebra with generators $\{s_e,s_e^*:e\in \mathcal{G}^1\}\cup\{p_A:A\in \mathcal{G}^0\}$ and relations
	\begin{enumerate}
		\item $p_\emptyset=0,  p_Ap_B=p_{A\cap B},  p_{A\cup B}=p_A+p_B-p_{A\cap B}$, for all $A,B\in \mathcal{G}^0$;
		\item $p_{s(e)}s_e=s_ep_{r(e)}=s_e$ and $p_{r(e)}s_e^*=s_e^*p_{s(e)}=s_e^*$ for each $e\in \mathcal{G}^1$;
		\item $s_e^*s_f=\delta_{e,f}p_{r(e)}$ for all $e,f\in \mathcal{G}$;
		\item $p_v=\sum\limits_{s(e)=v}s_es_e^*$ whenever $0<\vert s^{-1}(v)\vert< \infty$.
	\end{enumerate}
\end{definition}

\begin{proposition}\label{LPAU}
Let $\mathcal G$ be an ultragraph such that every vertex is regular and let  $\osf_\CG$ be its one-sided edge shift. Then $\CA_R(\osf_\CG)\cong L_R(\CG)$.
\end{proposition}
\begin{proof}

Notice that for $e\in\mathcal G^1$, we have that $F_e= C(e,\eword)=\{x \in\osf_\CG: s(x)\in r(e)\}$.

For each $A\in P(G^0)$, we let $A'=\{x\in\osf_\CG: s(x)\in A\}$. Clearly, for $A,B\in P(G^0)$, we have that $(A\cup B)'=A'\cup B'$ and $(A\cap B)'=A'\cap B'$. Also, for $v\in G^0$, because $v$ is regular, we have $\{v\}'=\bigcup_{e\in s^{-1}(v)}C(\eword,e)\in\CB$. And for $e\in \CG^1$, we have $r(e)'=C(e,\eword)\in\CB$. It follows from Lemma~\ref{description} that $A'\in\CB$ for all $A\in\CG^0$.

With this in mind, we build a map from $L_R(\mathcal G)$ to $\CA_R(\osf_\CG)$ as follows. For $e\in\mathcal G^1$, we set $t_e:=s_e$ and $t_{e}^*:=s_e^*$. For each $A\in\CG^0$, we set $q_A:=p_{A'}$.

We prove that the families $\{t_e,t_{e}^*: e\in\mathcal G^1\}$ and $\{q_A: A\in\mathcal G^0\}$ satisfy the relations defining $L_R(\mathcal G)$.

(1) For the empty set, we have $q_{\emptyset}=p_{\emptyset'}=p_{\emptyset}=0$. If $A,B\in\CG^0$, then $q_Aq_B=p_{A'}p_{B'}=p_{A'\cap B'}=p_{(A\cap B)'}=q_{A\cap B}$. Similarly, we see that $q_{A\cup B}=q_A+q_B-q_{A\cap B}$.

(2)  Let $e\in\mathcal G^1$ and set $v=s(e)$. Then,
\[q_{s(e)}t_e=\sum_{f\in s^{-1}(v)}p_{C(\eword,f)} s_e=\sum_{f\in s^{-1}(v)}s_fs_f^* s_e =s_e=t_e, \] and
\[t_eq_{r(e)}=s_e p_{C(e,\eword)}= s_e s_e^* s_e=s_e=t_e. \]
The relations involving $t_{e^*}$ follow analogously.

(3) For $e,f\in\mathcal G_1$, using Proposition~\ref{gen no unit}, we have that
\[t_e^*t_f=s_e^*s_f=\delta_{e,f}p_{C(e,\eword)}=\delta_{e,f}q_{r(e)}.\]

(4) If $v$ is a regular vertex, then 
\[q_v=\sum_{f\in s^{-1}(v)}p_{C(\eword,f)}=\sum_{f\in s^{-1}(v)}s_fs_f^*=\sum_{f\in s^{-1}(v)}t_ft_f^*.\]

By the universal property of $L_R(\mathcal G)$, we obtain an $R$-algebra homomorphism $\Phi:L_R(\mathcal G)\to\CA_R(\osf_\CG)$ which is surjective because the set $\{s_e,s_e^*: e\in\mathcal G^1\}$ generates $\CA_R(\osf_\CG)$ by Proposition~\ref{gen no unit}. It is easy to see that this homomorphism is $\zn$-graded. Moreover, if $r\in R\setminus\{0\}$ and $A \in\mathcal G^0$ is non-empty, then $rq_A\neq 0$ by Theorem~\ref{soumavez}. Applying the Graded Uniqueness Theorem for ultragraphs \cite[Theorem 5.4]{GDD2}, we conclude that $\Phi$ is an isomorphism.
\end{proof}

The class of Leavitt path algebras in Proposition \ref{LPAU} contains algebras that cannot be obtained as the Leavitt path algebra of a graph. For example, the ultragraph with vertices $\{v\}\cup\{w_i: i\in \N\}$, and edges $\{e\}\cup\{f_i: i\in \N\}$, with $s(e)=v$, $r(e)=\{w_i:i\in\N\}$, and $s(f_i)=r(f_i)=\{w_i\}$ belongs to the aforementioned class and, by  \cite[Proposition~2.7]{NamGonc} the associated Leavitt path algebra is not isomorphic to the Leavitt path algebra of any graph.

\section{Unital algebras of subshifts via partial actions}\label{skewrings}

In this section we give two descriptions of $\ualgshift$ as partial skew group rings, one arising from a set-theoretic partial action (Subsection \ref{theoretic.pa}) and the other from a topological partial action (Subsection \ref{topological.pa}). We refer the reader to \cite{ExelBook} for background on partial actions. See also \cite[Definition 4.2]{ExelLaca03} for the definitions of semi-saturated and orthogonal partial actions.

Let $\osf$ be a subshift over an alphabet $\alf$.

\subsection{The partial skew group ring $\udalgshift\rtimes_{\tau}\F$}\label{theoretic.pa}
 Let $\CF(\osf,R)$ denote the $R$-algebra of functions from $\osf$ to $R$ with pointwise operations. Then we let $\udalgshift$ be the subalgebra of $\CF(\osf,R)$ generated by the characteristic functions of the sets $C(\alpha,\beta)$, where $\alpha,\beta\in\lang$. Let $\F$ be the free group generated by $\alf$ with the empty word $\eword$ as the identity of $\F$.

Our goal in this subsection is to show that there is a partial action $\tau$ of $\F$ on $\udalgshift$ such that $\ualgshift$ is isomorphic to the partial skew group ring $\udalgshift\rtimes_{\tau}\F$.

For $a\in\alf$, we define $\tauh_{a}:C(a,\eword)\to C(\eword,a)$ by 
\begin{equation}\label{eq:partial1}
    \tauh_{a}(x)=a x
\end{equation}
and $\tauh_{a^{-1}}:C(\eword,a)\to C(a,\eword)$ by 
\begin{equation}\label{eq:partial2}
    \tauh_{a^{-1}}(ax)= x.
\end{equation}

\begin{proposition}\label{set.partial.action}
 The maps $\tauh_a$ and $\tauh_{a^{-1}}$, with $a\in\alf$, define a unique orthogonal and semi-saturated partial action
$\tauh=\left(\{W_t\}_{t\in\F},\{\tauh_t\}_{t\in\F}  \right)$ of $\F$ on $\osf$ such that  $W_{\beta\alpha^{-1}} = C(\alpha,\beta)$ and
\[\tauh_{\alpha\beta^{-1}}(\beta x)=\alpha x\]
for every $\beta x\in C(\alpha,\beta)$ and $\alpha,\beta\in\lang$ with $\beta\alpha^{-1}$ in reduced form. Moreover, if $t\neq \alpha\beta^{-1}$ for every $\alpha,\beta\in\lang$, then $W_t=\emptyset$.
\end{proposition}
\begin{proof}
The partial action is essentially that of \cite[Section~4]{MishaRuy}, but for an arbitrary alphabet. We leave the details to the reader. 
\end{proof}

Next, we associate an algebraic partial action with $\tauh$, similar to the dual action of a topological partial action. For $\alpha,\beta\in \lang$ such that $\alpha\beta^{-1}$ is in reduced form, let $1_{\alpha\beta^{-1}}$ denote the characteristic function of $W_{\alpha\beta^{-1}}=C(\beta,\alpha)$ and $D_{\alpha\beta^{-1}}$ the ideal of $\udalgshift$ generated by $1_{\alpha\beta^{-1}}$. Note that $D_{\alpha\beta^{-1}}$ is the ideal of functions in $\udalgshift$ that vanish on $C(\beta,\alpha)^c$ and has unit $1_{\alpha\beta^{-1}}$. Define $\tau_{\alpha\beta^{-1}}:D_{\beta\alpha^{-1}}\to D_{\alpha\beta^{-1}}$ by $\tau_{\alpha\beta^{-1}}(f)=f\circ\tauh_{\beta\alpha^{-1}}$, where $f\in D_{\beta\alpha^{-1}}$. Since $\tauh_{\alpha\beta^{-1}}$ is a bijection that maps $W_{\beta\alpha^{-1}}=C(\alpha,\beta)$ onto $W_{\alpha\beta^{-1}}=C(\beta,\alpha)$, it follows that $\tau_{\alpha\beta^{-1}}$ is an isomorphism that maps the ideal $D_{\beta\alpha^{-1}}$ onto the ideal $D_{\alpha\beta^{-1}}$. In particular, $\tau_{\alpha\beta^{ -1}}(1_{\beta\alpha^{-1}})=1_{\alpha\beta^{-1}}$. If $t$ cannot be expressed in the form $\alpha\beta^{-1}$, we define $D_t=\{0\}$ and $\tau_t$ equals the zero function. Hence, we have an algebraic partial action $\tau=\left( \{D_t\}_{t\in \F}, \{\tau_t\}_{t\in\F} \right)$  of $\F$ on $\udalgshift$. This partial action is semi-saturated; to see this, apply \cite[Proposition~4.10]{ExelBook} to $\{\tau_{a}, a\in\alf\}$, and appeal to the uniqueness of the partial action given by \cite[Proposition~4.10]{ExelBook}. Moreover,  if $a,b\in\alf$ and $a\neq b$, then $C(\eword,a)\cap C(\eword,b) =\emptyset$, which implies that $D_a\cap D_{b} = \{0\}$. That is, $\tau$ is an orthogonal partial action.

\begin{remark}\label{lem:reduced_forms}
  Note that, for $t\in\F$, we have that $D_t\neq\{0\}$ if and only if there exist $\alpha,\beta\in\lang$ such that $t=\alpha\beta^{-1}$ and $C(\beta,\alpha)\neq\emptyset$.
\end{remark}

The partial skew group ring associated with $\tau$ is defined as 
\[\udalgshift\rtimes_{\tau}\F= \bigoplus_{t\in\F} D_t=\left\{\sum_{} f_t\delta_t: t\in \F, f_t\in D_t \right\}, \]
where it is understood that $f_t$ is non-zero for finitely many terms and $\delta_t$ merely serves as a placeholder to remind us that $f_t\in D_t$. Multiplication is defined by 
\begin{equation}\label{eq:partial_multiplication}
    (f_s\delta_s)(g_t\delta_t) = \tau_s(\tau_s^{-1}(f_s)g_t )\delta_{st}.
\end{equation}

The following three lemmas describe some products, a $\mathbb{Z}$-grading, and a set of generators of $\udalgshift\rtimes_{\tau}\F$. Since the techniques used in the proofs are very similar to those used in \cite{GillesDanie}, we omit them here.

\begin{lemma}\label{siames} Consider the partial skew group ring $\udalgshift\rtimes_{\tau}\F$. For every $\alpha\in\lang\setminus\{\eword\}$, we have that
\begin{enumerate}[(i)]
    \item\label{break alpha} $(1_{\alpha_1}\delta_{\alpha_1})\cdots(1_{\alpha_{|\alpha|}}\delta_{\alpha_{|\alpha|}})=1_\alpha\delta_\alpha$,
    \item\label{break alpha-1} $(1_{\alpha_{|\alpha|}^{-1}}\delta_{\alpha_{|\alpha|}^{-1}})\cdots(1_{\alpha_1^{-1}}\delta_{\alpha_1^{-1}})=1_{\alpha^{-1}}\delta_{\alpha^{-1}}$,
    \item\label{alpha alpha-1} $(1_\alpha\delta_\alpha) (1_{\alpha^{-1}}\delta_{\alpha^{-1}})=1_{\alpha}\delta_{\eword}$,
    \item\label{alpha-1 alpha} $(1_{\alpha^{-1}}\delta_{\alpha^{-1}}) (1_\alpha\delta_\alpha) = 1_{\alpha^{-1}}\delta_{\eword}$,
    \item\label{alpha beta-1} $(1_{\alpha}\delta_{\alpha})(1_{\beta^{-1}}\delta_{\beta^{-1}}) = 1_{\alpha\beta^{-1}}\delta_{\alpha\beta^{-1}}$ for $\beta\in\lang$ such that $\alpha\beta^{-1}$ is in reduced form.
\end{enumerate}
\end{lemma}

\begin{lemma}
    The algebra $\udalgshift\rtimes_{\tau}\F$ has a $\mathbb{Z}$-grading, with the homogeneous component of degree $n$ given by 
    \[ (\udalgshift\rtimes_{\tau}\F)_n = \vecspan _R\{f_{\alpha\beta^{-1}}\delta_{\alpha\beta^{-1}} 
    : \alpha,\beta\in \lang \text{ and } |\alpha| -|\beta|= n\}.\]
\end{lemma}

\begin{lemma}\label{lem:partial_skew_gen}
The partial skew group ring $\udalgshift\rtimes_{\tau}\F$ is generated as an $R$-algebra by  \[\{1_{a}\delta_{a},1_{a^{-1}}\delta_{a^{-1}}: a\in \alf\cup\{\eword\}\}.\]
\end{lemma}

\begin{theorem}\label{thm:set-theoretic-partial-action}
Let $\osf$ be a subshift. Then, $\ualgshift\cong\udalgshift\rtimes_{\tau}\F$ via an isomorphism that sends $s_a$ to $1_a\delta_a$ and $s^*_a$ to $1_{a^{-1}}\delta_{a^{-1}}$.
\end{theorem}
\begin{proof}
We use the universal property of $\ualgshift$ in order to build a homomorphism $\Phi:\ualgshift\to \udalgshift\rtimes_{\tau}\F$. Let $A\in\TCB$. By \cite[Lemma~2.2]{GDD2}, we have that $1_A\in \udalgshift$. Then, $P_A:=1_A\delta_{\eword}\in \udalgshift\rtimes_{\tau}\F$. For $a\in\alf$, let $S_a:=1_a\delta_{a}\in \udalgshift\rtimes_{\tau}\F$ and let $S_a^*:=1_{a^{-1}}\delta_{a^{-1}}\in \udalgshift\rtimes_{\tau}\F$. We claim that $P_A$ and $S_a$ satisfy the relations in Definition~\ref{gelado}. Indeed, Item (i) follows from basic properties of algebraic operations on characteristic functions, and Item (ii) follows from Lemma~\ref{siames}. For Item (iii), let $\alpha,\beta\in\lang$ and put $S_{\alpha} = S_{\alpha_1}\cdots S_{\alpha_{|\alpha|}}$, $S_{\beta} = S_{\beta_1}\cdots S_{\beta_{|\beta|}}$, $S_{\alpha}^* = S_{\alpha_{|\alpha|}}^*\cdots S_{\alpha_1}^*$ and $S_{\beta}^* = S_{\beta_{|\beta|}}^*\cdots S_{\beta_1}^*$. Write $\alpha=\alpha'\gamma$ and $\beta=\beta'\gamma$ in such a way that $\beta'\alpha'^{-1}$ is in reduced form. Note that $C(\eword,\beta)\cap C(\alpha',\beta')=C(\alpha,\beta)$. Then, by Lemma~\ref{siames}, we have that
\begin{align*}
S_{\beta}S_{\alpha}^*S_{\alpha}S_{\beta}^* & = (1_\beta\delta_\beta)(1_{\alpha^{-1}}\delta_{\alpha^{-1}})(1_\alpha\delta_\alpha)(1_{\beta^{-1}}\delta_{\beta^{-1}}) \\
& = (1_\beta\delta_\beta)(1_{\alpha^{-1}}\delta_\eword)(1_{\beta^{-1}}\delta_{\beta^{-1}}) \\
& = \tau_\beta(1_{\beta^{-1}}1_{\alpha^{-1}})\delta_\eword \\
&= 1_{\beta}1_{\beta\alpha^{-1}}\delta_\eword \\
& = 1_\beta1_{\beta'\alpha'^{-1}}\delta_\eword \\
&= 1_{C(\eword,\beta)}1_{C(\alpha',\beta')}\delta_{\eword} \\
&= 1_{C(\alpha,\beta)}\delta_\eword\\
&= P_{C(\alpha,\beta)},
\end{align*}
which shows that Item (iii) holds. Then, by the universal property of $\ualgshift$, there exists a homomorphism $\Phi:\ualgshift\to \udalgshift\rtimes_{\tau}\F$ taking $s_a\mapsto S_a$, $s_a^*\mapsto S_a^*$ and $p_A\mapsto P_A$. It follows from Lemma~\ref{lem:partial_skew_gen} that $\Phi$ is surjective, and it is not hard to see that  $\Phi$ is $\mathbb{Z}$-graded.

Also, note that for every non-zero $r\in R$ and non-empty $A\in\TCB$,
\[\Phi(rp_A) = r1_A\delta_\eword \neq 0,\]
since $1_A$ is a characteristic function of a non-empty set. By Corollary~\ref{GUTshift}, $\Psi$ is injective.
\end{proof} 

We can use Theorem~\ref{thm:set-theoretic-partial-action} to also characterise $\algshift$ as a partial skew group ring. For that, we let $\CB$ be the Boolean algebra defined in Section \ref{s:nonunital} and $\dalgshift$ the subalgebra of $\udalgshift$ generated by $\{1_{A}\}_{A\in\CB}$. In fact, $\dalgshift$ is an ideal of $\udalgshift$, so we can restrict the partial action $\tau$ to $\dalgshift$. More precisely, for each $t\in\F$, we let \[D'_t=\{f\in D_t\cap\dalgshift:\tau_{t^{-1}}(f)\in\dalgshift\}\]
and $\tau'_t:D'_{t^{-1}}\to D'_t$ the restriction of $\tau_t$.

\begin{theorem}
    Let $\osf$ be a subshift. Then $\algshift\cong \dalgshift\rtimes_{\tau'}\F$.
\end{theorem}

\begin{proof}
    The proof follows the same line of thought as Lemma~\ref{lem:partial_skew_gen} and Theorem~\ref{thm:set-theoretic-partial-action}. The main difference is that $1_\eword\delta_\eword$ does not appear as a generator of $\dalgshift\rtimes_{\tau'}\F$.
\end{proof}

\subsection{The partial skew group ring $\Lc(\HTCB,R)\rtimes_{\varphi}\F$}\label{topological.pa}
Next, we construct a topological partial action such that its dual algebraic partial action gives a partial skew ring isomorphic to $\udalgshift\rtimes_{\tau}\F$. Recall from Subsection \ref{s:stone} that $\widehat{\TCB}$ denotes the Stone dual of $\TCB$ and the topology on $\widehat{\TCB}$ has a basis given by the sets $O_A=\{\xi\in \widehat{\TCB}: A\in \xi\}$, where $A\in\TCB$. 
For each $a\in \alf$ we put $V_a:=O_{Z_a}$ and $V_{a^{-1}}:=O_{F_a}$. Then we define  $\varphih_a: V_{a^{-1}} \to V_{a}$ by
\begin{equation}\label{eq:top_part_act_1}
 \varphih_{a}(\eta)=\{A\in \TCB: r(A,a)\in \eta \} ,
\end{equation}
and we define $\varphih_{a^{-1}}: V_{a} \to V_{a^{-1}}$  by 
\begin{equation}\label{eq:top_part_act_2}
\varphih_{a^{-1}}(\xi)=\{B\in \TCB: r(A,a)\subseteq B \text{ for some } A\in \xi \}.
\end{equation}

We need the following lemma to prove that the maps $\varphih_a$ and $\varphih_{a^{-1}}$ are well-defined homeomorphisms.

\begin{lemma}\label{brokenmic}
For each $a\in\alf$ and $A\in\TCB$, let $aA=\{ax\in\osf: x\in A\}$. Then, $aA\in\TCB$. Moreover, $r(aA,a)\scj A$ and if $B\in\TCB$ is such that $B\scj Z_a$ and $r(B,a)\scj A$, then $B\scj aA$.
\end{lemma}

\begin{proof}
Suppose first that $A=C(\alpha,\beta)$. If $a\beta\in\lang$, then $aA=C(\alpha,a\beta)\in\TCB$. And if $a\beta\notin\lang$, then $aA=\emptyset\in\TCB$. We claim that $aC(\alpha,\beta)^c=Z_a\setminus C(\alpha,a\beta)$ if $a\beta\in\lang$. Suppose first that $ax\in aC(\alpha,\beta)^c$, then clearly $ax\in Z_a$. Supposing that $ax\in C(\alpha,a\beta)$, we would get $ax=a\beta y$ for some $y\in\osf$ such that $\alpha y\in\osf$. This would imply that $x\in C(\alpha,\beta)$, which is a contradiction. Similarly, if $ax\in Z_a\setminus C(\alpha,a\beta)$, we get that $x\in C(\alpha,\beta)^c$ so that $ax\in aC(\alpha,\beta)^c$.

It is easy to see that for $A_1,A_2\in\TCB$, we have that $a(A_1\cup A_2)=aA_1\cup aA_2$ and $a(A_1\cap A_2)=aA_1\cap aA_2$.

Since sets of the form $C(\alpha,\beta)$ generate $\TCB$, the first part of the result follows.

For the second part, we have
\[r(aA,a)=\{x\in\osf: ax\in aA\}\scj A.\]
Now, let $B\in\TCB$ be such that $B\scj Z_a$ and $r(B,a)\scj A$. For $x\in B$, we have that $x=ax'$ for some $x'\in\osf$. Hence $x'\in r(B,a)\scj A$ so that $x=ax'\in aA$. 
\end{proof}
 
\begin{proposition} \label{prop:partial-homeom}
For every $a\in \alf$ the map $\varphih_a$ is a homeomorphism of $V_{a^{-1}}$ onto $V_a$, with inverse $\varphih_{a^{-1}}$.  
\end{proposition}
\begin{proof}
We begin by showing $\varphih_{a^{-1}}$ is well-defined. Let $\xi\in V_{a}$. Then $Z_a\in \xi$ and, 
since $r(Z_a,a) = C(a,\omega)=F_a$, we have that $F_a\in \varphih_{a^{-1}}(\xi)$. Hence $\varphih_{a^{-1}}(\xi)\in V_{a^{-1}}$. To see that $\varphih_{a^{-1}}(\xi)$ is a filter, suppose that $B_1, B_2 \in \varphih_{a^{-1}}(\xi)$. Then, there exists $A_1$ and $A_2$ in $\xi$ such that $r(A_i,a)\subseteq B_i$, $i=1,2$. Then, $r(A_1\cap A_2, a) = r(A_1,a)\cap r(A_2,a) \subseteq B_1\cap B_2$. So, $B_1\cap B_2\in \varphih_{a^{-1}}(\xi). $ It is clear that $\varphih_{a^{-1}}(\xi)$ is upward-closed. For $B_1,B_2\in\TCB$, suppose that $r(A,a)\scj B_1\cup B_2$ for some $A\in\xi$. Then, $A\cap Z_a\scj aB_1\cup aB_2$, by Lemma~\ref{brokenmic}, and hence $aB_1\cup aB_2\in\xi$. Because $\xi$ is an ultrafilter, either $aB_1\in\xi$ or $aB_2\in\xi$. Note that $r(aB_i,a)\scj B_i$, $i=1,2$, so that either $B_1\in\varphih_{a^{-1}}(\xi)$ or $B_2\in\varphih_{a^{-1}}(\xi)$. Also, for $A\in\xi$, because $Z_a\in\xi$, we have that $\emptyset\neq r(A\cap Z_a,a)= r(A,a)$ so that $\emptyset\notin \varphih_{a^{-1}}(\xi)$. Hence $\varphih_{a^{-1}}(\xi)$ is an ultrafilter, which proves that $\varphih_{a^{-1}}$ is well-defined. 

Next, we show that $\varphih_{a}$ is well-defined.
Let $\eta\in\HTCB$ such that $F_a\in \eta$. Then $r(Z_a,a)=F_a$ implies that $Z_a\in\ \varphih_{a}(\eta)$ and thus $\varphih_{a}(\eta)\in O_{Z_a}$. To see that  $\varphih_{a}(\eta)$ is an ultrafilter, notice that $\emptyset\notin\varphih_{a}(\eta)$, because $r(\emptyset,a)=\emptyset\notin \eta$. Let $A,B\in\varphih_{a}(\eta)$. Then $r(A\cap B,a)=r(A,a)\cap r(B,a)\in\eta$ so that $A\cap B\in\varphih_{a}(\eta)$. If $A\in\varphih_{a}(\eta)$ and $B\in\TCB$ is such that $A\scj B$, then $r(A,a)\scj r(B,a)$ so that $r(B,a)\in\eta$. Hence $B\in\varphih_{a}(\eta)$. Finally, if $A,B\in\TCB$ are such that $A\cup B\in\varphih_{a}(\eta)$, then $r(A,a)\cup r(B,a)=r(A\cup B,a)\in\eta$. Because $\eta$ is an ultrafilter, $r(A,a)\in\eta$ or $r(B,a)\in\eta$, from where we get $A\in\varphih_{a}(\eta)$ or $B\in\varphih_{a}(\eta)$. Hence $\varphih_{a}(\eta)$ is an ultrafilter and well-defined.

We show that $\varphih_{a^{-1}}$ and $\varphih_{a}$ are bijections and inverses of each other. First we show that $\varphih_{a^{-1}}(\varphih_{a}(\eta))=\eta$. Because these maps take ultrafilters to ultrafilters, it  suffices to show $\varphih_{a^{-1}}(\varphih_{a}(\eta))\scj\eta$. Let $B\in\varphih_{a^{-1}}(\varphih_{a}(\eta))$. Then there exists $A\in\varphih_{a}(\eta)$ such that $r(A,a)\scj B$. By the definition of $\varphih_{a}(\eta)$, we have that $r(A,a)\in\eta$ and hence $B\in\eta$, proving that $\varphih_{a^{-1}}(\varphih_{a}(\eta))\scj\eta$. Secondly, to see that $\varphih_{a}(\varphih_{a^{-1}}(\xi))=\xi$ for all $\xi\in V_a=O_{Z_a}$, let $\xi\in\HTCB$ be such that $Z_a\in \xi$. For $A\in\xi$, we have that $r(A,a)\in\varphih_{a^{-1}}(\xi)$ and hence $A\in\varphih_{a}(\varphih_{a^{-1}}(\xi))$, proving that $\xi\scj\varphih_{a}(\varphih_{a^{-1}}(\xi))$. Since these are ultrafilters, it follows that   $\varphih_{a}(\varphih_{a^{-1}}(\xi))=\xi$. Hence, $\varphih_{a^{-1}}$ and $\varphih_{a}$ are bijections and inverses of each other. 

Finally we show that $\varphih_{a^{-1}}$ is a homeomorphism of $V_a=O_{Z_a}$ onto $V_{a^{-1}}=O_{F_a}$. If $A\scj Z_a$, then  $\varphih_{a^{-1}}(O_A)=O_{r(A,a)}\scj O_{F_a}$. If $B\scj F_a$, then we claim that $\varphih_{a}(O_B)=O_{aB}$. For $\xi\in O_{aB}$, we have that $\xi\in O_{Z_a}$, because $aB\scj Z_a$. Also, $\varphih_{a^{-1}}(\xi)\in O_B$, because $r(aB,a)\scj B$, and hence $\xi=\varphih_{a}(\varphih_{a^{-1}}(\xi))\in\varphih_{a}(O_B)$. On the other hand, let $\xi\in \varphih_{a}(O_B)$ so that $\xi\in O_{Z_a}$ and $\varphih_{a^{-1}}(\xi)\in O_B$. Then there exists $A\in\xi$ such that $r(A,a)\scj B$. Also, $r(A\cap Z_a,a)=r(A,a)\scj B$. By Lemma~\ref{brokenmic}, $aB$ is the largest subset of $Z_a$ such that $r(aB,a)\scj B$, so that $A\cap Z_a\scj aB$ and $aB\in\xi$. Hence $\xi\in O_{aB}$. Therefore, since the sets $O_A$ form a basis for the topology in $\HTCB$, $\varphih_{a^{-1}}$ is a homeomorphism of $V_a=O_{Z_a}$ onto $V_{a^{-1}}=O_{F_a}$ and the proof is complete. 
\end{proof}

\begin{proposition}\label{top.partial.action}
The maps $\varphih_a$ and $\varphih_{a^{-1}}$, with $a\in\alf$, define an orthogonal semi-saturated topological partial action $\varphih=\left( \{V_t\}_{t\in \F}, \{\varphih_t\}_{t\in\F} \right)$ of $\F$ on $\HTCB$ such that $V_{\beta\alpha^{-1}}=O_{C(\alpha,\beta)}$ and
\[\varphih_{\alpha\beta^{-1}}(\xi)=\{A\in\TCB:r(B,\beta)\scj r(A,\alpha)\text{ for some }B\in\xi\}\] 
for every $\xi\in V_{\beta\alpha^{-1}}$ and $\alpha,\beta\in\lang$ such that $\alpha \beta^{-1}$ is in reduced form.
\end{proposition}
\begin{proof}
Since each $\varphih_a$, with $a\in \alf$, is a bijection of $V_{a^{-1}}$ onto $V_{a}$ by Proposition~\ref{prop:partial-homeom}, it follows from \cite[Proposition 4.10]{ExelBook} that there is a unique semi-saturated set-theoretic partial action $\varphih=\left( \{V_t\}_{t\in \F}, \{\varphih_t\}_{t\in\F} \right)$ of $\F$ on $\HTCB$. Moreover, this partial action is orthogonal, since $V_a\cap V_b=O_{Z_a}\cap O_{Z_b}=O_{Z_a\cap Z_b}=O_{\emptyset}=\emptyset$ for $a,b\in \alf$ and $a\neq b$. Appealing again to Proposition~\ref{prop:partial-homeom} and the fact that $\varphih$ is semi-saturated, we see that $\varphih$ is a topological partial action.

Note that for every $\alpha,\beta\in\lang$ and $A\in\TCB$, $r(r(A,\alpha),\beta)=r(A,\alpha\beta)$. By induction, applying \eqref{eq:top_part_act_1} and \eqref{eq:top_part_act_2} iteratively and using that the partial action is semi-saturated, for every $\alpha\in\lang$, we have that $V_{\alpha}=O_{Z_\alpha}$, $V_{\alpha^{-1}}=O_{F_\alpha}$, and for every $\eta\in V_{\alpha^{-1}}$ and $\xi\in V_{\alpha}$, we have that 
\begin{equation}\label{eq:partial_phi_on_word}
    \varphih_{\alpha}(\eta)=\{A\in \TCB: r(A,\alpha)\in \eta \}
\end{equation}
and
\[\varphih_{\alpha^{-1}}(\xi)=\{B\in \TCB: r(A,\alpha)\subseteq B \text{ for some } A\in \xi \}.\]
Thus, for $\alpha,\beta\in\lang$ such that $\alpha\beta^{-1}$ is in reduced form, again using that the partial action is semi-saturated, we have that
\[\varphih_{\alpha\beta^{-1}}(\xi)=\varphih_{\alpha}\circ\varphih_{\beta^{-1}}(\xi)=\{A\in\TCB:r(B,\beta)\scj r(A,\alpha)\text{ for some }B\in\xi\}\] 
for every $\xi\in V_{\beta\alpha^{-1}}$.

We now prove that $V_{\beta\alpha^{-1}}=O_{C(\alpha,\beta)}$. By \cite[Proposition~2.6]{ExelBook}, we have that $V_{\beta\alpha^{-1}}\cap V_{\beta} = \varphih_{\beta}(V_{\alpha^{-1}}\cap V_{\beta^{-1}})$. Because the partial action is semi-saturated, $V_{\beta\alpha^{-1}}\scj V_{\beta}$, so that $V_{\beta\alpha^{-1}} = \varphih_{\beta}(V_{\alpha^{-1}}\cap V_{\beta^{-1}})$. Thus, 
\begin{eqnarray*}
V_{\beta\alpha^{-1}} & = & \{\varphih_{\beta}(\eta): \eta\in                                         V_{\alpha^{-1}}\cap V_{\beta^{-1}} \} \\
 &=& \{\varphih_{\beta}(\eta): C(\alpha,\eword)\cap C(\beta,\eword) \in \eta\}. 
\end{eqnarray*} 
It follows from Equation~(\ref{eq:rel.range}) that 
$r(C(\alpha,\beta),\beta) = C(\alpha,\eword)\cap C(\beta,\eword)$. Therefore,
$\varphih_{\beta}(\eta) \in V_{\beta\alpha^{-1}}$ if and only if $r(C(\alpha,\beta),\beta)\in \eta$, which holds if and only if $C(\alpha,\beta)\in \varphih_{\beta}(\eta)$, by Equation~(\ref{eq:partial_phi_on_word}). That is, $\varphih_{\beta}(\eta) \in V_{\beta\alpha^{-1}}$ if and only if $\varphih_{\beta}(\eta)\in O_{C(\alpha,\beta)}$. This completes the proof.
\end{proof}

The partial action of Proposition \ref{top.partial.action} can be seen as an extension of the partial action of Proposition \ref{set.partial.action}. For that, notice that for each $x\in\osf$, the set $\xi_x=\{A\in\TCB:x\in A\}$ is an element of $\HTCB$. We can then define a map $\iota:\osf\to\HTCB$ by $\iota(x)=\xi_x$.

\begin{proposition}\label{pa.extension}
The map $\iota:\osf\to\HTCB$ is injective and equivariant with respect to $\tauh$ and $\varphih$. Moreover, the image of $\iota$ is dense in $\HTCB$.
\end{proposition}

\begin{proof}
To see that $\iota$ is injective, let $x,y\in\osf$ be such that $x\neq y$. Then, there exists $n\in\nn$ such that $x_{0,n}\neq y_{0,n}$ and hence $Z_{x_{0,n}}\in\xi_x\setminus\xi_y$, which implies that $\xi_x\neq\xi_y$.

To show that $\iota$ is equivariant, it is enough to consider $t=\alpha\beta^{-1}$ for $\alpha,\beta\in\lang$, where $t$ is written in reduced form. Let $\beta x\in C(\alpha,\beta)$. On the one hand, by Proposition \ref{top.partial.action}, we have that
\[\varphih_{\alpha\beta^{-1}}(\iota(\beta x))=\{A\in\TCB:r(B,\beta)\scj r(A,\alpha)\text{ for some }B\ni \beta x\}.\]
On the other hand, by Proposition~\ref{set.partial.action},
\[\iota(\tauh_{\alpha\beta^{-1}}(\beta x))=\{A\in\TCB:\alpha x\in A\}.\]
Let $A\in \varphih_{\alpha\beta^{-1}}(\iota(\beta x))$ and let $B$ be such that $\beta x\in B$ and $r(B,\beta)\scj r(A,\alpha)$. By definition, we have that $x\in r(B,\beta)$ and hence $\alpha x\in A$, that is, $A\in \iota(\tauh_{\alpha\beta^{-1}}(\beta x))$. Because we are dealing with ultrafilters, we conclude that $\varphih_{\alpha\beta^{-1}}(\iota(\beta x))=\iota(\tauh_{\alpha\beta^{-1}}(\beta x))$.

For the last part, let $\xi\in\HTCB$. A basic open neighborhood of $\xi$ is a set of the form $O_A$ for some non-empty $A\in\TCB$. Then, for any $x\in A$, we have that $\iota(x)\in O_A$.
\end{proof}

We now build the dual algebraic partial action of $\varphih$ as follows. For $t\in\F$, we let $I_t=\Lc(V_{t},R)$. Then $I_{t}$ is the unital ideal in $\Lc(\HTCB,R)$ generated by the characteristic function $1_{V_{t}}$. Define $\varphi_t:I_{t^{-1}}\to I_{t}$ by 
\[\varphi_t(g)= g\circ\varphih_{t^{-1}}\]
for $g\in I_{t^{-1}}$. Then $\varphi=\left( \{I_t\}_{t\in\F},\{\varphi_t\}_{t\in\F} \right)$ is an algebraic partial action of $\F$ on $\Lc(\HTCB,R)$. Our next goal is to prove that $\algshift$ is isomorphic to the partial skew group ring $\Lc(\HTCB,R)\rtimes_{\varphi}\F$ associated with $\varphi$.

If $Y$ is a set and $\mathcal{C}\subseteq \mathcal{P}(Y)$, then we let $\mathcal{F}_{\mathcal{C}}$ denote the subalgebra of $\CF(Y,R)$ generated by $\{1_{C}\}_{C\in\mathcal{C}}$. 

\begin{lemma}\label{lem:DRX_gen}
The algebras  $\udalgshift$ and  $\mathcal{F}_{\TCB}$ coincide.
\end{lemma}
\begin{proof}
Since $\TCB$ is the algebra of sets generated by $\{C(\alpha,\beta): \alpha,\beta\in\lang\}$, it follows from \cite[Lemma 2.2]{GDD2} that $\udalgshift$ is generated as an algebra by $\{1_{A}: A\in\TCB\}$.
\end{proof}

\begin{lemma}\label{sleep}
Let $\alpha,\beta,\gamma,\delta\in\lang$ be such $\delta=\beta\delta'$ for some $\delta'\in\lang$ and $\alpha\beta^{-1}$ is in reduced form in $\F$. Then,
\begin{enumerate}[(i)]
    \item\label{sleep1} $\tauh_{\alpha\beta^{-1}}(C(\gamma,\delta)\cap C(\alpha,\beta))=C(\gamma,\alpha\delta')\cap C(\beta,\alpha)$,
    \item\label{sleep2} $\varphih_{\alpha\beta^{-1}}(O_{C(\gamma,\delta)\cap C(\alpha,\beta)})=O_{C(\gamma,\alpha\delta')\cap C(\beta,\alpha)}$.
\end{enumerate}
\end{lemma}

\begin{proof}
(i) Let $x\in C(\gamma,\delta)\cap C(\alpha,\beta)$. Then $x=\delta y=\beta\delta' y$ for some $y\in\osf$ such that $\gamma y,\alpha\delta' y\in\osf$. Applying $\tauh_{\alpha\beta^{-1}}$, we have that
\[\tauh_{\alpha\beta^{-1}}(x) = \alpha\delta' y \in C(\gamma,\alpha\delta')\cap C(\beta,\alpha).\]
Hence, $\tauh_{\alpha\beta^{-1}}(C(\gamma,\delta)\cap C(\alpha,\beta))\subseteq C(\gamma,\alpha\delta')\cap C(\beta,\alpha)$. Analogously, $\tauh_{\beta\alpha^{-1}}(C(\gamma,\alpha\delta')\cap C(\beta,\alpha))\subseteq C(\gamma,\delta)\cap C(\alpha,\beta)$. Applying $\tauh_{\alpha\beta^{-1}}$ to both sides of the latter inclusion, we have that $C(\gamma,\alpha\delta')\cap C(\beta,\alpha)\subseteq \tauh_{\alpha\beta^{-1}}(C(\gamma,\delta)\cap C(\alpha,\beta))$, which proves that $\tauh_{\alpha\beta^{-1}}(C(\gamma,\delta)\cap C(\alpha,\beta))=C(\gamma,\alpha\delta')\cap C(\beta,\alpha)$. 
(ii) Let $\xi\in O_{C(\gamma,\delta)\cap C(\alpha,\beta)}$. Then, \[\varphih_{\alpha\beta^{-1}}(\xi) = \{A\in \TCB: r(B,\beta)\subseteq r(A,\alpha) \text{ for some } B\in\xi\},\] by Proposition~\ref{top.partial.action}. Note that $C(\gamma,\delta)\cap C(\alpha,\beta)\in \xi$. Thus, to demonstrate that $\varphih_{\alpha\beta^{-1}}(O_{C(\gamma,\delta)\cap C(\alpha,\beta)}) \subseteq O_{C(\gamma,\alpha\delta')\cap C(\beta,\alpha)}$, it suffices to show that $r(C(\gamma,\delta)\cap C(\alpha,\beta),\beta) \subseteq r(C(\gamma,\alpha\delta')\cap C(\beta,\alpha),\alpha)$. 
Applying Equation~(\ref{eq:rel.range}) we obtain
\[r(C(\gamma,\delta)\cap C(\alpha,\beta),\beta)=C(\gamma,\delta')\cap C(\alpha,\eword)\cap C(\beta,\eword)=r(C(\gamma,\alpha\delta')\cap C(\beta,\alpha),\alpha).\]
Therefore, we have that $C(\gamma,\alpha\delta')\cap C(\beta,\alpha)\in \varphih_{\alpha\beta^{-1}}(O_{C(\gamma,\delta)\cap C(\alpha,\beta)})$, which implies that $\varphih_{\alpha\beta^{-1}}(O_{C(\gamma,\delta)\cap C(\alpha,\beta)}) \subseteq O_{C(\gamma,\alpha\delta')\cap C(\beta,\alpha)}$. For the reverse inclusion, it suffices to see that $\varphih_{\beta\alpha^{-1}}(O_{C(\gamma,\alpha\delta')\cap C(\beta,\alpha)}) \subseteq O_{C(\gamma,\delta)\cap C(\alpha,\beta)}$. For this it is sufficient to see that  $r(C(\gamma,\alpha\delta')\cap C(\beta,\alpha),\alpha) \subseteq r(C(\gamma,\delta)\cap C(\alpha,\beta),\beta)$. But, this latter inclusion also follows from the observation that $r(C(\gamma,\delta)\cap C(\alpha,\beta),\beta) = r(C(\gamma,\alpha\delta')\cap C(\beta,\alpha),\alpha)$. Hence, $\varphih_{\alpha\beta^{-1}}(O_{C(\gamma,\delta)\cap C(\alpha,\beta)})=O_{C(\gamma,\alpha\delta')\cap C(\beta,\alpha)}$.
\end{proof}

\begin{proposition} \label{isodiagonals}
There exists an isomorphism $\Phi:\udalgshift\to \Lc(\HTCB,R)$ defined by $\Phi(1_A)=1_{O_{A}}$ for every $A\in \TCB$. Moreover, this isomorphism is equivariant with respect to the partial actions $\tau$ on $\udalgshift$ and $\varphi$ on $\Lc(\HTCB,R)$.
\end{proposition}
\begin{proof}
Recall that $\CK(\HTCB)$ denotes the set of all compact-open subsets of $\HTCB$ and note that $\Lc(\HTCB,R)=\mathcal{F}_{\CK(\HTCB)}$. By the Stone Duality (Theorem~\ref{thm:stone}), the map $A\mapsto O_{A}$ is a Boolean algebra isomorphism of $\TCB$ onto $\CK(\HTCB)$. Thus, $\mathcal{F}_\TCB$ is isomorphic to $\mathcal{F}_{\CK(\HTCB)} = \Lc(\HTCB,R)$ as $R$-algebras, with the isomorphism given by $\Phi(1_A)= 1_{O_{A}}$, for each $A\in\TCB$. But $\udalgshift= \mathcal{F}_{\TCB}$, by Lemma~\ref{lem:DRX_gen}. Hence, $\Phi$ is an isomorphism of $\udalgshift$ onto $\Lc(\HTCB,R)$.

We now prove that $\Phi$ is equivariant. By Remark~\ref{lem:reduced_forms}, it is enough to consider $t=\alpha\beta^{-1}$ for $\alpha,\beta\in\lang$ such that $C(\beta,\alpha)\neq\emptyset$. We may also assume that $t$ is written in reduced form. For $f\in D_{\beta\alpha^{-1}}$, we need to prove that $\Phi(f)\in I_{\beta\alpha^{-1}}$ and $\Phi(\tau_{\alpha\beta^{-1}}(f))=\varphi_{\alpha\beta^{-1}}(\Phi(f))$. By the definition of $\udalgshift$, we have that $D_{\beta\alpha_{-1}}$ is generated by functions of the form $1_{C(\alpha,\beta)\cap C(\gamma,\delta)}$ for $\gamma,\delta\in\lang$. Because $\Phi$ is an isomorphism and $I_{\alpha\beta^{-1}}$ is an ideal, it is enough to consider $f=1_{C(\alpha,\beta)\cap C(\gamma,\delta)}$ for some $\gamma,\delta\in\lang$. If $C(\alpha,\beta)\cap C(\gamma,\delta)=\emptyset$, then $f=0$ and the result is trivial. We may then assume that $C(\alpha,\beta)\cap C(\gamma,\delta)\neq\emptyset$. In this case, notice that either $\beta=\delta\beta'$ for some $\beta'\in\lang$ or $\delta=\beta\delta'$ for some $\delta'\in\lang$. In the first case, we have that $C(\alpha,\beta)\cap C(\gamma,\delta)=C(\alpha,\beta)\cap C(\gamma\beta',\beta)$, so we may assume without loss of generality that $\delta=\beta\delta'$ for some $\delta'\in\lang$. Firstly,
\[\Phi(f)=\Phi(1_{C(\alpha,\beta)\cap C(\gamma,\delta)})=1_{O_{C(\alpha,\beta)\cap C(\gamma,\delta)}}\in I_{\beta\alpha^{-1}}.\]
Secondly, by Lemma~\ref{sleep}, we have that
\begin{align*}
    \Phi(\tau_{\alpha\beta^{-1}}(f)) &=\Phi(1_{C(\alpha,\beta)\cap C(\gamma,\delta)}\circ\tauh_{\beta\alpha^{-1}}) \\
    &=\Phi(1_{\tauh_{\alpha\beta^{-1}}(C(\alpha,\beta)\cap C(\gamma,\delta))})\\
    &=\Phi(1_{C(\beta,\alpha)\cap C(\gamma,\alpha\delta')}) \\
    &=1_{O_{C(\beta,\alpha)\cap C(\gamma,\alpha\delta')}} \\
    &=1_{\varphih_{\alpha\beta^{-1}}(O_{C(\alpha,\beta)\cap C(\gamma,\delta)})} \\
    &=1_{O_{C(\alpha,\beta)\cap C(\gamma,\delta)}}\circ\varphih_{\beta\alpha^{-1}} \\
    &=\varphi_{\alpha\beta^{-1}}(1_{O_{C(\alpha,\beta)\cap C(\gamma,\delta)}}) \\
    &=\varphi_{\alpha\beta^{-1}}(\Phi(1_{C(\alpha,\beta)\cap C(\gamma,\delta)}))\\
    &=\varphi_{\alpha\beta^{-1}}(\Phi(f)).
\end{align*}
Hence, $\Phi$ is equivariant with respect to $\tau$ and $\varphi$, and the proof is complete.
\end{proof}

\begin{theorem} \label{thm:shift_alg_partil_skew}
   The partial skew group rings $\udalgshift\rtimes_{\tau}\F$ and $\Lc(\HTCB,R)\rtimes_{\varphi}\F$ are isomorphic. In particular, $\ualgshift\cong \Lc(\HTCB,R)\rtimes_{\varphi}\F$ via an isomorphism that sends $s_a$ to $1_{O_{C(\eword,a)}}\delta_a$ and $s^*_a$ to $1_{O_{C(a,\eword)}}\delta_{a^{-1}}$.
\end{theorem}
\begin{proof}
The first part follows immediately from Proposition~\ref{isodiagonals}, and the second from Theorem~\ref{thm:set-theoretic-partial-action} and the isomorphism found in Proposition~\ref{isodiagonals}.
\end{proof}

\begin{remark}
Since $\CB$ is an ideal of $\TCB$, in the sense of Boolean algebras, there exists an embedding of $\HCB$ into $\HTCB$, whose image is open in $\HTCB$. We can then restrict $\varphih$ to a partial action on $\HCB$, which in turn gives a restriction of $\varphi$ to a partial action $\varphi'$ on $\Lc(\HCB,R)$. One can then show that $\algshift\cong \Lc(\HCB,R)\rtimes_{\varphi'}\F$. Since our main focus is on the unital algebra $\ualgshift$, we will not go through the details of the proof.
\end{remark}

\section{Groupoid models for unital subshift algebras}\label{s:groupoid picture}
In this section, we study the unital subshift algebra as a Steinberg algebra of a groupoid. We define two (isomorphic) groupoids associated with a subshift. The first is the transformation groupoid of the partial action defined in Section \ref{topological.pa}. The second is the Deaconu-Renault groupoid associated with a singly generated dynamical system \cite{Renault00}. The latter is used in Section~\ref{s:conjugacy} to study conjugacy of subshifts.

Recall that a Hausdorff topological groupoid $\mathcal{G}$ is an ample groupoid if its topology has a basis of compact-open bisections. Then the Steinberg algebra associated with $\mathcal{G}$ (see \cite{BenGroupoid}) is defined as 
\[\mathcal{A}_R(\mathcal{G}) = \vecspan_R \{1_{U}: U\subseteq \mathcal{G} \text{ is a compact-open set} \}.\]
Equivalently, $\mathcal{A}_R(\mathcal{G})$ consists of all locally constant functions $f:\mathcal{G}\to R$ with compact support. Multiplication in $\mathcal{A}_R(\mathcal{G})$ is defined by the convolution product 
\[ f*g(\gamma) = \sum_{s(\eta)=s(\gamma)} f(\eta)g(\eta^{-1}\gamma).\]

Having realised $\ualgshift$ as the partial skew group ring $\Lc(\HTCB,R)\rtimes_{\varphi}\F$ (Theorem~\ref{thm:shift_alg_partil_skew}), we now have a groupoid model for $\ualgshift$, given by the transformation groupoid 
\[ \F\ltimes_{\varphi} \HTCB = \{(\xi,t,\eta)\in \HTCB\times\F\times \HTCB: \eta\in V_{t^{-1}} \text{ and } \xi = \varphih_t (\eta) \},\]
where multiplication and inverses are defined by
\[(\xi,s,\eta)(\eta,t,\gamma) = (\xi,st,\gamma) \text{ and } (\xi,t,\eta)^{-1}=(\eta,t^{-1},\xi),\]
respectively. Note that $\F\ltimes_{\varphih} \HTCB$ is an ample groupoid, since the topology on $\HTCB$ has a basis of compact-open sets.

Next, we describe a Deaconu-Renault groupoid associated with a subshift. For the reader's convenience, we recall the key definitions first.

\begin{definition}\label{d:SGDS}
A \emph{singly generated dynamical system} is a pair $(X,\hsig)$ consisting of a locally compact Hausdorff space $X$, and a local homeomorphism $\hsig: \dom(\hsig)\to \Img(\hsig)$ from an open set $\dom (\hsig)\subseteq X$ onto an open set $\Img(\hsig)\subseteq X$. Inductively, define $D_0=X$ and $D_n:=\dom(\hsig^n)=\hsig^{-1}(D_{n-1})$ for $n>0$. The \emph{Deaconu-Renault groupoid} associated with $(X,\hsig)$ is defined as \[\CG(X,\hsig)=\displaystyle\bigcup_{n,m\in \mathbb N} \Big \{(x,n-m,y)\in D_n\times \{n-m\}\times D_m :  \hsig^n(x)=\hsig^m(y)\Big \},\]
equipped with the topology with basic open sets \[\CZ(U,V,n,m):=\big \{(x,n-m,y):x\in U,y\in V,\text{ and }\hsig^n(x)=\hsig^m(y)\big\},\] 
indexed by quadruples $(U,V,n,m)$, where $n,m\in \mathbb N$, $U\subseteq D_n$ and $V\subseteq D_m$ are open and $\hsig^n|_U$ and $\hsig^m|_V$ are homeomorphisms. The operations are given by $(x,k,y)(y,l,z)=(x,k+l,z)$ and $(x,k,y)^{-1}=(y,-k,x)$, for $(x,k,y),(y,l,z)\in \CG(X,\hsig)$. Note that $r(x,k,y)=(x,0,x)$ and $s(x,k,y)=(y,0,y)$. Therefore, we identify the unit space $\CG(X,\hsig)^{(0)}$ with $X$ as topological spaces. Also, there is a one-cocycle $c:\CG(X,\hsig)\to\zn$ given by $c(x,k,y)=k$.
\end{definition}

Given a singly generated dynamical system $(X,\hsig)$, its Steinberg algebra $\CA_R(\CG(X,\hsig))$ has a natural $\zn$-grading given by $\CA_R(\CG(X,\hsig))_n:=\{f\in\CA_R(\CG(X,\hsig)):\supp(f)\scj c^{-1}(n)\}$, see \cite{GroupoidLeavitt}.

In our case, we consider the pair $(\HTCB, \hsig)$, where $\HTCB$ is the Stone dual of $\TCB$ (see Section~\ref{s:stone}), and $\hsig$ has domain  $\dom(\hsig):=\bigcup_{a\in\alf} V_a = \bigcup_{a\in\alf} O_{Z_a}$ and, for $\xi \in V_a$, is defined by \[\hsig(\xi):=\{B\in \TCB: r(A,a)\subseteq B \text{ for some } A\in \xi \}.\]

Before studying the groupoid $\CG(\HTCB,\hsig)$, we first show how the map $\hsig$ relates to the shift map on the OTW subshift $\osf^{OTW}$.  For that we define a map $\pi:\HTCB\to\osf^{OTW}$ by 
\[\pi(\xi)=\begin{cases}
\alpha, &\text{if }\alpha\in \lang\text{ is such that }Z_\alpha\in\xi\text{ and }Z_{\alpha b}\notin\xi\text{ for all }b\in\alf; \\
\alpha, &\text{if }\alpha\in \osf^{inf}\text{ and }Z_{\alpha_{1,n}}\in\xi\text{ for all }n\geq 1; \\
\vec{0}, &\text{if }Z_a\notin\xi\text{ for all }a\in\alf. 
\end{cases}\]
To show that $\pi$ is well-defined we need to show that, in the first case, $\pi(\xi)=\alpha$ indeed belongs to $\osf^{OTW}$ (the other cases are straightforward). Let $L=\{b\in\alf:\exists \eta\in\osf\text{ s.t. }\alpha b\eta\in\osf\}$ and note that $L\neq\emptyset$ because $\alpha\in\lang$. Suppose that $L$ is finite. Then we can write $Z_\alpha$ as the finite union $Z_{\alpha}=\bigcup_{b\in L} Z_{\alpha b}$. Because $\xi$ is an ultrafilter, there exists $b\in\alf$ such that $Z_{\alpha b}\in\xi$, which contradicts the definition of $\pi(\xi)$. This means that $L$ is infinite, that is, $\alpha\in\osf^{fin}$.

\begin{remark}
In $\ualgshift$, there is an inclusion of commutative subalgebras generated by idempotents $\vecspan_R\{s_\alpha s_\alpha^*\}\scj\vecspan_R\{s_\alpha p_As_\alpha^*\}=\vecspan_R\{p_A\}$. When $R$ is an indecomposable ring, via the isomorphisms of Propositions~\ref{algebra OTW} and \ref{ohcoconuts}, we obtain a map as follows. Given $\xi\in\HTCB$, there is corresponding character $\phi_\xi$ on $\vecspan_R\{p_A\}$ such $\phi_\xi(p_A)=1$ if $A\in\xi$ and 0 otherwise. When we restrict $\phi_\xi$ to $\vecspan_R\{s_\alpha s_\alpha^*\}$, using that $s_\alpha s_\alpha^*=p_{Z_\alpha}$ and the isomorphism of Proposition~\ref{algebra OTW}, we see that this map is the same as $\pi$ defined above. We point out that the definition of $\pi$ itself does not depend on $R$.
\end{remark}

\begin{proposition}\label{prop:pi_shift_commute}
For all $\xi\in\dom(\hsig)$, we have that $\pi\circ\hsig(\xi)=\sigma\circ\pi(\xi)$. Moreover, $\pi$ is surjective.
\end{proposition}

\begin{proof}
Let $\xi\in\dom(\hsig)$ and $\alpha=\pi(\xi)$. Notice that $\alpha\neq \vec{0}$, otherwise $\xi \notin \dom(\hsig)$. 
Suppose that $|\alpha|=\infty$. Then, for every $n\in\nn^*$, we have that $r(Z_{\alpha_{1,n}},\alpha_1)=F_{\alpha_1}\cap Z_{\alpha_2,n}\scj Z_{\alpha_2,n}$, by Equation~\ref{eq:rel.range}. Hence $Z_{\alpha_2,n}\in\hsig(\xi)$ for all $n\in\nn^*$, that is, $\pi(\hsig(\xi))=\alpha_{2,\infty}=\sigma(\pi(\xi))$.

Suppose now that $|\alpha|<\infty$. The above argument shows that $Z_{\alpha_{2,|\alpha|}}\in\hsig(\xi)$. Suppose that for some $b\in\alf$, we have that $Z_{\alpha_{2,|\alpha|}b}\in\hsig(\xi)$. This would imply that there exists $A\in\TCB$ such that $r(A,\alpha_1)\scj Z_{\alpha_{2,|\alpha|}b}$ so that $A\scj Z_{\alpha b}$. In this case, $Z_{\alpha b}\in\xi$, which is a contradiction, since $\pi(\xi)=\alpha$. It follows that $\pi(\hsig(\xi))=\alpha_{2,|\alpha|}=\sigma(\pi(\xi))$.

Next, we show that $\pi$ is surjective. Let $\alpha \in \osf^{OTW}$. If $\alpha\in\osf^{inf}$, then $\alpha=\pi(\xi)$, where  $\xi=\{A\in\TCB: \alpha\in A\}$. If $\alpha\in\osf^{fin}$, then $\alpha$ is the image of an ultrafilter that contains the upper set $\uparrow\{Z_\alpha\cap Z_{\alpha b_1}^c\cap \ldots \cap Z_{\alpha b_n}^c: b_i\in \alf, n\in \N \}$ (which exists because the family $\{Z_\alpha\cap Z_{\alpha b_1}^c\cap \ldots \cap Z_{\alpha b_n}^c: b_i\in \alf, n\in \N \}$ is closed under intersection and does not contain the empty set).
\end{proof}

We now prove that the transformation groupoid of the partial action of Section~\ref{topological.pa} and the Deaconu-Renault groupoid of $\hsig$ are isomorphic. As a consequence, we describe the subshift algebra as a Steinberg algebra.

\begin{theorem}\label{isogroupoid}
    Let $\osf$ be a subshift. The map $\Theta:\F\ltimes_{\varphih} \HTCB  \to  \CG(\HTCB,\hsig)$ defined by $\Theta(\xi,\alpha\beta^{-1},\eta)=(\xi,|\alpha|-|\beta|,\eta)$
    is an isomorphism of topological groupoids.
\end{theorem}

\begin{proof}
    Notice that $\varphih_{a^{-1}}$ (see Equation~(\ref{eq:top_part_act_2})) is the restriction of $\hsig$ to $V_a$. The remainder of the proof is straightforward and follows the same steps of \cite[Theorem~5.5]{GillesDanie}, \cite[Theorem~5.12]{MR4443753} and \cite[Theorem~4.4]{GillesEunJi2}.
\end{proof}

\begin{theorem}\label{Steinberg} Let $\osf$ be a subshift. Then, $\ualgshift\cong \CA_R(\CG(\HTCB,\hsig))$ as $\zn$-graded algebras via an isomorphism that takes $s_a$ to $S_a:=1_{\CZ(O_{Z_a},O_{F_a},1,0)}$ and $s_a^*$ to $S_a^*:=1_{\CZ(O_{F_a}, O_{Z_a},0,1)}$. Moreover, this isomorphism sends the diagonal subalgebra of $\ualgshift$ to $\CA_R(\CG(\HTCB,\hsig)^{(0)})$.
\end{theorem}

\begin{proof}
From Theorem~\ref{thm:shift_alg_partil_skew} and \cite[Theorem 3.2]{GonBeu18}, we get an isomorphism between $\ualgshift$ and  $\mathcal{A}_R(\F\ltimes_{\varphih} \HTCB)$ that sends $s_a$ to $1_{V_a\times\{a\}\times V_{a^{-1}}}$ and $s^*_a$ to $1_{V_{a^{-1}}\times\{a^{-1}\}\times V_a}$. Applying the map $\Theta$ from Theorem~\ref{isogroupoid}, we see that $\Theta(V_a\times\{a\}\times V_{a^{-1}})=\CZ(O_{Z_a},O_{F_a},1,0)$ and similarly, $\Theta(V_{a^{-1}}\times\{a^{-1}\}\times V_a)=\CZ(O_{F_a}, O_{Z_a},0,1)$. The result the follows from the fact that we have an isomorphism from $\mathcal{A}_R(\F\ltimes_{\varphih} \HTCB)$ to $\CA_R(\CG(\HTCB,\hsig))$ that sends $f$ to $f\circ\Theta^{-1}$.

The last part follows from Proposition~\ref{ohcoconuts} and the fact that $\CA_R(\CG(\HTCB,\hsig)^{(0)})\cong\Lc(\HTCB,R)$.
\end{proof}

Next, we define a map $\epsilon:\dom(\epsilon)\scj \CG(\HTCB,\hsig)\to\CG(\HTCB,\hsig)$ that will play a important role when discussing conjugacy of OTW-subshifts in Section~\ref{s:conjugacy}. This map is inspired by the one found in \cite{BrixCarlsen} and is defined as follows. The domain of $\epsilon$ is given by
\[\dom(\epsilon)=\{(\xi,n,\eta)\in\CG(\HTCB,\hsig): \xi,\eta\in\dom(\hsig)\}\]
and, for $(\xi,n,\eta)\in\dom(\epsilon)$, set
\[\epsilon(\xi,n,\eta)=(\hsig(\xi),n,\hsig(\eta)).\]

We would like an analogue of \cite[Lemma 4.2]{BrixCarlsen}, in the sense that we want to define a map $\tau:\CA_R(\CG(\HTCB,\hsig))\to\CA_R(\CG(\HTCB,\hsig))$ by $\tau(f)=f\circ\epsilon$. However, if the alphabet $\alf$ is infinite, then there are a few problems that we need to overcome. The first is that $\dom(\epsilon)$ is not necessarily all $\CG(\HTCB,\hsig)$ and, secondly, may fail to be clopen in $\CG(\HTCB,\hsig)$, so that we cannot continuously extend $f\circ\epsilon$ to be zero outside $\dom(\epsilon)$. Lastly, even if $\dom(\epsilon)$ is clopen, there is no guarantee that $f\circ\epsilon$ has compact support. To overcome these issues, we will work with restrictions of $\epsilon$.

Fix $M\scj \alf$ finite and let $V_M=\cup_{a\in M}V_a$. Note that $V_M$ is compact-open in $\HTCB$. We let \begin{equation}\label{dom(epsilonM)}
\dom(\epsilon_M)=\{(\xi,n,\eta)\in\CG(\HTCB,\hsig): \xi,\eta\in V_M\}=s^{-1}(V_M)\cap r^{-1}(V_M),\end{equation}
which is clopen in $\CG(\HTCB,\hsig)$, and we define $\epsilon_M$ as the restriction of $\epsilon$ to $\dom(\epsilon_M)$.

\begin{lemma}\label{epsilonM}
The map $\epsilon_M$ is continuous and proper.
\end{lemma}

\begin{proof}
Consider a basic open set $\CZ(U,V,k,l)$ of $\CG(\HTCB,\hsig)$, where $U,V$ are compact-open subsets of $\HTCB$. Then,
\begin{align*}
    \epsilon^{-1}_M&(\CZ(U,V,k,l))=\\ &=\{(\xi,m,\eta)\in\dom(\epsilon_M): \hsig(\xi)\in U,\hsig(\eta)\in V,m=k-l,\hsig^{k+1}(\xi)=\hsig^{l+1}(\eta)\}\\
    &=\bigcup_{a,b\in M}\{(\xi,m,\eta)\in\dom(\epsilon_M): \xi\in\varphih_a(O_{F_a}\cap U),\eta\in \varphih_b(O_{F_b}\cap V),\\
    &\hphantom{\bigcup_{a,b\in M}\{(\xi,m,\eta)\in\dom(\epsilon_M): \xi\in\varphih_a(O_{F_a}\cap U),}m=k-l,\hsig^{k+1}(\xi)=\hsig^{l+1}(\eta)\} \\
    &=\bigcup_{a,b\in M}\CZ(\varphih_a(O_{F_a}\cap U),\varphih_b(O_{F_b}\cap V),k+1,l+1).
\end{align*}

This shows that the preimage of each basic compact-open is a finite union of basic compact-open sets. We immediately get that $\epsilon_M$ is continuous. And because $\CG(\HTCB,\hsig)$ is Hausdorff and $\epsilon_M$ is continuous, the above equality also implies that $\epsilon_M$ is proper. Indeed the preimage of a compact set $K$ by $\epsilon_M$ is then a closed subset contained in a finite union of compact-open basics open sets and therefore $\epsilon_M^{-1}(K)$ is compact.
\end{proof}

\begin{lemma}\label{l:tauM}
The map $\tau_M:\CA_R(\CG(\HTCB,\hsig))\to\CA_R(\CG(\HTCB,\hsig))$ defined by
\[\tau_M(f)(\xi,m,\eta)=\begin{cases}
f(\epsilon_M(\xi,m,\eta)), & \text{if }(\xi,m,\eta)\in\dom(\epsilon_M) \\
0, &\text{otherwise,}
\end{cases}\]
where $f\in\CA_R(\CG(\HTCB,\hsig))$ and $(\xi,m,\eta)\in \CG(\HTCB,\hsig)$, is a well-defined $R$-linear map. Moreover, for every $f\in\CA_R(\CG(\HTCB,\hsig))$, we have that
\[\tau_M(f)=\sum_{a,b\in M}S_a*f*S_b^{*},\]
where the product $*$ is the convolution product in $\CA_R(\CG(\HTCB,\hsig))$, $S_a=1_{\CZ(O_{Z_a},O_{F_a},1,0)}$, and $S_a^*=1_{\CZ(O_{F_a}, O_{Z_a},0,1)}$.
\end{lemma}

\begin{proof}
Because $\dom(\epsilon_M)$ is clopen in $\CG(\HTCB,\hsig)$, $\epsilon_M$ is continuous (by Lemma \ref{epsilonM}) and $f$ is locally constant, we have that $\tau_M(f)$ is locally constant. Moreover, Lemma \ref{epsilonM} says that $\epsilon_M$ is proper, so that $\supp(\tau_M(f))=\epsilon_M^{-1}(\supp(f))$ is compact. This implies that $\tau_M$ is well-defined. That $\tau_M$ is $R$-linear follows from the fact that addition and scalar multiplication in $\CA_R(\CG(\HTCB,\hsig))$ are defined pointwise.

For the second part, let $(\xi,m,\eta)\in\CG(\HTCB,\hsig)$. Then
\[\sum_{a,b\in M}S_a*f*S_b^{*}(\xi,m,\eta)=\sum_{a,b\in M}\sum S_a(\xi,k,\zeta) f(\zeta,l,\theta) S_b^*(\theta,n,\eta),\]
where the second summation is over all triples $(\xi,k,\zeta),(\zeta,l,\theta),(\theta,n,\eta)\in\CG(\HTCB,\hsig)$ such that $(\xi,k,\zeta)(\zeta,l,\theta)(\theta,n,\eta)=(\xi,m,\eta)$.
If $(\xi,m,\eta)\notin \dom(\epsilon_M)$, then either $\xi\notin V_M$ or $\eta\notin V_M$. In the first case, we have that $S_a(\xi,k,\zeta)=0$ for all $a\in M$ independently of $k$ and $\zeta$. Similarly, in the second case $S_b^*(\theta,n,\eta)=0$ for all $b\in M$ independently of $n$ and $\theta$. This means that if $(\xi,m,\eta)\notin \dom(\epsilon_M)$, and thus 
\[\sum_{a,b\in M}S_a*f*S_b^{*}(\xi,m,\eta)=0=\tau_M(f)(\xi,m,\eta).\]
Suppose now that $(\xi,m,\eta)\in\dom(\epsilon_M)$. In this case there is a unique $a_0\in M$ such that $\xi\in V_{a_0}$ and a unique $b_0\in M$ such that $\eta\in V_{b_0}$. In this case, we have
\begin{align*}
    \sum_{a,b\in M}S_a*f*S_b^{*}(\xi,m,\eta)&= \sum_{a,b\in M}\sum S_{a}(\xi,k,\zeta) f(\zeta,l,\theta) S_{b}^*(\theta,n,\eta) \\
    &=\sum S_{a_0}(\xi,k,\zeta) f(\zeta,l,\theta) S_{b_0}^*(\theta,n,\eta)\\
    &=S_{a_0}(\xi,1,\hsig(\xi)) f(\hsig(\xi),m,\hsig(\eta)) S_{b_0}^*(\hsig(\eta),-1,\eta)\\
    &=f(\hsig(\xi),m,\hsig(\eta)) \\
    &=f(\epsilon_M(\xi,m,\eta)) \\
    &=\tau_M(f)(\xi,m,\eta),
\end{align*}
and the result then follows.
\end{proof}

\begin{proposition}\label{p:restrictionM}
    Let $\osf_1$ and $\osf_2$ be two subshifts. Consider the corresponding groupoids $\CG(\HTCB_i,\hsig_i)$ with maps $\epsilon^i$ as constructed above, for $i=1,2$. Also, let $\Phi:\CG(\HTCB_1,\hsig_1)\to\CG(\HTCB_2,\hsig_2)$ be an isomorphism of topological groupoids. The following are equivalent:
    \begin{enumerate}[(i)]
        \item\label{itg1} $\Phi(\dom(\epsilon^1))\scj\dom(\epsilon^2)$ and $\epsilon^2(\Phi(\xi,m,\eta))=\Phi(\epsilon^1(\xi,m,\eta))$ for all $(\xi,m,\eta)\in\dom(\epsilon^1)$.
        \item\label{itg2} For every $M\scj\alf_1$ finite there exists $N\scj\alf_2$ finite such that $\Phi(\dom(\epsilon^1_M))\scj\dom(\epsilon^2_N)$ and $\epsilon^2_N(\Phi(\xi,m,\eta))=\Phi(\epsilon^1_M(\xi,m,\eta))$ for all $(\xi,m,\eta)\in\dom(\epsilon^1_M)$.
    \end{enumerate}
\end{proposition}

\begin{proof}
\eqref{itg1}$\Rightarrow$\eqref{itg2} Note that for each $i=1,2$, using the identification $\CG(\HTCB_i,\hsig_i)^{(0)}=\HTCB_i$, we have that
\[\bigcup_{a\in\alf_i}V^i_a=\dom(\epsilon^i)\cap \CG(\HTCB_i,\hsig_i)^{(0)}.\]
The above equality and the hypothesis that $\Phi(\dom(\epsilon^1))\scj\dom(\epsilon^2)$ imply that for $M\scj\alf_1$ finite, we have that $\Phi(V^1_M)\scj\bigcup_{b\in\alf_2} V^2_b$. Because $V^1_M$ is compact and $\Phi$ is a homeomorphism, there exists a finite set $N\scj\alf_2$ such that $\Phi(V^1_M)\scj V^2_N$. Also, because $\Phi$ preserves $s$ and $r$, by Equation \eqref{dom(epsilonM)}, we get $\Phi(\dom(\epsilon^1_M))\scj\dom(\epsilon^2_N)$. For $(\xi,m,\eta)\in\dom(\epsilon^1_M)$, we have that
\[\epsilon^2_N(\Phi(\xi,m,\eta))=\epsilon^2(\Phi(\xi,m,\eta))=\Phi(\epsilon^1(\xi,m,\eta))=\Phi(\epsilon^1_M(\xi,m,\eta)).\]

\eqref{itg2}$\Rightarrow$\eqref{itg1} Let $(\xi,m,\eta)\in\dom(\epsilon^1)$. Then there exists $a,b\in\alf_1$ such that $\xi\in V_a$ and $\eta\in V_b$. Take $M=\{a,b\}$ and let $N\scj \alf_2$ be as in the hypothesis. Then $\Phi(\xi,m,\eta)\in\dom(\epsilon^2_N)\scj\dom(\epsilon^2)$ and
\[\epsilon^2(\Phi(\xi,m,\eta))=\epsilon^2_N(\Phi(\xi,m,\eta))=\Phi(\epsilon^1_M(\xi,m,\eta))=\Phi(\epsilon^1(\xi,m,\eta)).\]
\end{proof}

\section{Conjugacy of OTW subshifts}\label{s:conjugacy}
In this section, we describe a conjugacy of OTW-subshifts in terms of an isomorphism of the associated groupoids and in terms of an isomorphism of the associated subshift algebras, see Theorem~\ref{theone}. We retain the notation of Section~\ref{s:groupoid picture} and start recalling the definition of a conjugacy between OTW-subshifts. 

\begin{definition}
    Let $ \osf_1^{OTW}$ and $ \osf_2^{OTW}$ be OTW-subshifts over alphabets $\alf_1$ and $\alf_2$, respectively. A map $h:  \osf_1^{OTW}\rightarrow \osf_2^{OTW}$ is a \emph{conjugacy} if it is a homeomorphism, commutes with the shift and is length-preserving.
\end{definition}

We point out that because a conjugacy $h:\osf_1^{OTW}\rightarrow \osf_2^{OTW}$ is length-preserving, we can restrict it to a bijection between the corresponding subshifts $\osf_1$ and $\osf_2$, which we also denote by $h$, see Remark~\ref{computer}.

\begin{remark}
Notice that for OTW-subshifts, a conjugacy is not necessarily given by a usual sliding block code. Instead, one has to use the notion of a generalised sliding block code, see \cite{GSS}.
\end{remark}

Before we prove the main theorem of the section, we need a few auxiliary results.

\begin{lemma}\label{l:hCF}
Let $h:\osf_1^{OTW}\to \osf_2^{OTW}$ be a conjugacy. Let $\alpha\in\CL_{\osf_1}$ and $F\scj\alf_1$ a finite set. Then
\begin{enumerate}[(i)]
    \item\label{i:hcyl} $h(\CZ_1(\alpha,F))=\bigcup_{i=1}^m \CZ_2(\mu^i,M^i)$ for some $\mu^i\in\CL_{\osf_2}$ with $|\mu^i|\geq|\alpha|$ and $M^i\scj \alf_2$ finite, for all $i=1,\ldots,m$. Moreover, if $\alpha\in\osf_1^{fin}$, then there exists $j\in\{1,\ldots,n\}$ such that $\mu^j=h(\alpha)$, and for all $i=1,\ldots,m$ such that $|\mu^i|=|\alpha|$ and $\mu^i\in\osf_2^{fin}$, we have that $\mu^i=h(\alpha)$.
    \item\label{i:hfol} $h(\CF_1(\alpha))=\bigcup_{i=1}^m \CZ_2(\mu^i,M^i)\cap\CF_2(\nu^i)$ for some $\mu^i,\nu^i\in\CL_{\osf_2}$ and $M^i\scj \alf_2$ finite, for all $i=1,\ldots,m$.
    \item\label{i:hcylinf} $h(Z^1_\alpha)\scj \bigcup_{i=1}^m Z^2_{\mu^i}$ for some $\mu^i\in\CL_{\osf_2}$ with $|\mu^i|\geq|\alpha|$, for all $i=1,\ldots,m$. Moreover, if $\alpha\in\osf_1^{fin}$, then there exists $j\in\{1,\ldots,n\}$ such that $\mu^j=h(\alpha)$, and for all $i=1,\ldots,m$ such that $|\mu^i|=|\alpha|$ and $\mu^i\in\osf_2^{fin}$, we have that $\mu^i=h(\alpha)$.
\end{enumerate}
\end{lemma}

\begin{proof}
\eqref{i:hcyl} Since $\CZ_1(\alpha,F)$ is compact-open and $h$ is a conjugacy, we have that $h(\CZ_1(\alpha,F))$ is compact-open in $\osf_2^{OTW}$. The result follows from the fact that $h$ is length-preserving and by the description of a neighbourhood base for a point in $\osf_2^{OTW}$ given in \cite[Theorem~2.15, Remark~3.24]{OTW}. For the second part, suppose that $\alpha\in\osf_1^{fin}$. In this case, if $\mu^i\in\osf_2^{fin}$ and $|\mu^i|=|h^{-1}(\mu^i)|=|\alpha|$, then we must have $h^{-1}(\mu^i)=\alpha$ because $\alpha$ is the only element in $\CZ_1(\alpha,F)$ with length $|\alpha|$. Hence $\mu^i=h(\alpha)$. On the other hand, $\alpha\in \CZ_1(\alpha,F)$ and hence $h(\alpha)\in \CZ_2(\mu^j,M^j)$ for some $j$. Because $|h(\alpha)|=|\alpha|\leq|\mu^j|$, the only possibility for $\mu^j$ is that it is equal to $h(\alpha)$.

\eqref{i:hfol} Notice that $\CF_1(\alpha)=\sigma^{|\alpha|}(\CZ_1(\alpha))$. Because $h$ commutes with the shift, the result follows from Item \eqref{i:hcyl} and Lemma~\ref{lem:backandforth}.

\eqref{i:hcylinf} Using Item \eqref{i:hcyl} with $F=\emptyset$, the fact that $h$ preserves length, and Lemma~\ref{l:CFZ}(i), we obtain that
\[h(Z^1_\alpha)=h(\CZ_1(\alpha)\cap\osf_1^{inf})=\bigcup_{i=1}^m \CZ_2(\mu^i,M^i)\cap\osf_2^{inf}\scj\bigcup_{i=1}^m Z^2_{\mu^i}\]
for some $\mu^i\in\CL_{\osf_2}$ with $|\mu^i|\geq|\alpha|$ and $M^i\scj \alf_2$ finite, for all $i=1,\ldots,m$. The second part follows immediately from Item \eqref{i:hcyl} and the above computation.
\end{proof}

\begin{proposition}\label{p:isob}
Let $h:\osf_1^{OTW}\to \osf_2^{OTW}$ be a conjugacy. The map $\overline{h}:\TCB_1\to\TCB_2$ given by $\overline{h}(A)=h(A)$ is an isomorphism of Boolean algebras.
\end{proposition}

\begin{proof}
Because $h$ is a bijection, it preserves unions, intersections, and relative complements. To prove that $\overline{h}$ is a homomorphism, it is then sufficient to show that $h(C_1(\alpha,\beta))\in\TCB_2$ for every $\alpha,\beta\in\CL_{\osf_1}$. Because $h$ is length-preserving, by Lemma~\ref{l:CFZ}, we have that
\[h(C_1(\alpha,\beta))=h(\CZ_1(\beta)\cap\sigma^{-|\beta|}(\CF_1(\alpha))\cap\osf_1^{inf})=h(\CZ_1(\beta))\cap h(\sigma^{-|\beta|}(\CF_1(\alpha)))\cap\osf_2^{inf}.\]
Using that $h$ commutes with the shift and Lemma~\ref{l:hCF}, we conclude that $h(C_1(\alpha,\beta))$ is a union of sets of the form $\CZ_2(\mu,M)\cap\sigma^{-|\beta|}(\CZ_2(\nu,N)\cap\CF_2(\rho))\cap \osf_2^{inf}$, where $\mu,\nu,\rho\in\CL_{\osf_2}$ with $|\mu|\geq |\beta|$, and $M,N\scj\alf_2$ are finite sets. By Lemmas~\ref{lem:backandforth} and \ref{l:CFZ}, for some $\tau\in\CL_{\osf_2}$ and $P\scj\alf_2$ finite, we have that
\begin{align*}
    \CZ_2(\mu,M)&\cap\sigma^{-|\beta|}(\CZ_2(\nu,N)\cap\CF_2(\rho))\cap \osf_2^{inf} \\ &=\CZ_2(\mu,M)\cap\sigma^{-|\beta|}(\CZ_2(\nu,N))\cap\CZ_2(\mu_{1,|\beta|})\cap\sigma^{-|\beta|}(\CF_2(\rho))\cap \osf_2^{inf}\\
    &=\CZ_2(\tau,P)\cap\CZ_2(\mu_{1,|\beta|})\cap\sigma^{-|\beta|}(\CF_2(\rho))\cap \osf_2^{inf}\\
    &=\left(C_2(\eword,\tau)\setminus \bigcup_{p\in P}C_2(\eword,\tau p)\right)\cap C_2(\rho,\mu_{1,|\beta|}),
\end{align*}
which is an element of $\TCB_2$. If follows that $h(C_1(\alpha,\beta))$ is a union of element of $\TCB_2$, and therefore it is also in $\TCB_2$.

To see that $\overline{h}$ is an isomorphism, just apply the above argument to $h^{-1}$ so that we obtain $\overline{h}^{-1}=\overline{h^{-1}}$.
\end{proof}

Next, we show that a conjugacy between OTW-subshifts lifts to a homeomorphism between the Stone dual of the associated Boolean algebras which commute with the corresponding $\hsig$.

\begin{proposition}\label{hlifts}
Let $h:\osf_1^{OTW}\to \osf_2^{OTW}$ be a conjugacy. The map $\hh:\widehat{\TCB_1}\to \widehat{\TCB_2}$ given by $\hh(\xi)=\{h(A): A\in\xi\}$ is a homeomorphism such that $\hh(\dom(\hsig_1))=\dom(\hsig_2)$, $\hh\circ\hsig_1=\hsig_2\circ\hh|_{\dom(\hsig_1)}$, and $h\circ\pi_1=\pi_2\circ\hh$.
\end{proposition}

\begin{proof}
By Proposition~\ref{p:isob} and the Stone duality (Theorem~\ref{thm:stone}), we have that $\hh$ is a well-defined homeomorphism.

Given $\xi\in\widehat{\TCB_1}$, we claim that $\pi_1(\xi)$ and $\pi_2\circ\hh(\xi)$ have the same length. By Lemma~\ref{l:hCF}\eqref{i:hcylinf}, the fact that ultrafilters are prime filters in Boolean algebras and the definitions of $\pi_1$ and $\pi_2$, we have that $|\pi_1(\xi)|\leq|\pi_2\circ\hh(\xi)|$. Similarly, since $\xi=\{h^{-1}(B):B\in \hh(\xi)\}$, using the same argument, we obtain that $|\pi_1(\xi)|\geq|\pi_2\circ\hh(\xi)|$. It follows that $\xi\in\dom(\hsig_1)$ if and only if $\hh(\xi)\in\dom(\hsig_2)$.

We show that $\hh\circ\hsig_1=\hsig_2\circ\hh|_{\dom(\hsig_1)}$. Let $\xi\in\dom(\hsig_1)$, $a\in\alf_1$ such that $Z^1_a\in\xi$ and $b\in\alf_2$ such that $Z^2_b\in\hh(\xi)$. Then
\[\hh(\hsig_1(\xi))=\{h(B): B\in\TCB_1\text{ and }r(A,a)\scj B\text{ for some }A\in\xi\}\]
and
\[\hsig_2(\hh(\xi))=\{C\in \TCB_2: r(h(A),b)\scj C\text{ for some }A\in\xi\}.\]
Take $B\in\TCB_1$ such that $r(A,a)\scj B$ for some $A\in\xi$. Note that $A\cap Z^1_a\cap h^{-1}(Z^2_b)\in\xi$, so we may assume without loss of generality that $A\scj Z^1_a\cap h^{-1}(Z^2_b)$. In this case, we have that $r(A,a)=\sigma_1(A)$ and $r(h(A),b)=\sigma_2(h(A))$. Because $h$ is shift commuting, we conclude that
\[r(h(A),b)=\sigma_2(h(A))=h(\sigma_1(A))=h(r(A,a))\scj h(B)\]
and hence $h(B)\in\hsig_2(\hh(\xi))$. Since we are dealing with ultrafilters, it follows that $\hh(\hsig_1(\xi))=\hsig_2(\hh(\xi))$.

Finally, we prove that $h\circ\pi_1=\pi_2\circ\hh$. Let $\xi\in\widehat{\TCB_1}$. As proved above, $\pi_1(\xi)$ and $\pi_2\circ\hh(\xi)$ have the same length. Since $\vec{0}_2$ is the only element of $\osf_2^{OTW}$ with length zero, if $\pi_1(\xi)=\vec{0}_1$ then $h(\pi_1(\xi))=h(\vec{0}_1)=\vec{0}_2=\pi_2(\hh(\xi))$. Suppose now that $\pi_1(\xi)=\alpha$ with $0<|\alpha|<\infty$. Then, by Lemma~\ref{l:hCF}\eqref{i:hcylinf}, the fact that $\hh(\xi)$ is an ultrafilter, and the definition of $\pi_2$, we have that $\pi_2(\hh(\xi))=h(\alpha)$. Lastly, suppose that $\pi_1(\xi)=\alpha$ with $|\alpha|=\infty$, so that $\beta=\pi_2(\hh(\xi))$ is such that $|\beta|=\infty$. By the definition of $\pi_1$, we have that $Z^1_{\alpha_{1,n}}\in\xi$ for each $n\in\nn^*$.  By Lemma~\ref{l:hCF}\eqref{i:hcylinf}, $h(Z^1_{\alpha_{1,n}})\scj \bigcup_{i=1}^m Z^2_{\mu^i}$ for some $\mu^i\in\CL_{\osf_2}$ with $|\mu^i|\geq n$, for all $i=1,\ldots,m$. Because $\hh(\xi)$ is an ultrafilter, by the definition of $\pi_2$, there exists $j\in\{1,\ldots,m\}$ such that $\mu^j$ is a beginning of $\beta$. This implies that for $m_n:=|\mu^j|\geq n$ we can find $x^n\in\osf_1^{inf}$ and $y^n\in\osf_2^{inf}$ such that $h(\alpha_{1,n}x^n)=\beta_{1,m_n}y^n$. By the continuity of $h$, taking the limit as $n$ goes to infinity, we get $h(\pi_1(\xi))=h(\alpha)=\beta=\pi_2(\hh(\xi))$.
\end{proof}

For a subshift $\osf$ over an alphabet $\alf$ and $M\scj\alf$ finite, we define
\[e_M=\sum_{a\in M}s_as_a^*\in\ualgshift.\]
Observe that $e_M$ is idempotent and, via the isomorphism given in Theorem~\ref{Steinberg}, it corresponds to the characteristic function $1_{V_M}$, where $V_M$ is the same as in Section \ref{s:groupoid picture}. Using Theorem~\ref{Steinberg} and $\tau_M$ of Lemma~\ref{l:tauM} we obtain a map from $\TCA_R(\osf)$ to $\TCA_R(\osf)$, also denoted by $\tau_M$, given by  
\[\tau_M(f)=\sum_{a,b\in M}s_afs_b^*.\]
Both $e_M$ and $\tau_M$ will play an important role in the algebraic characterisation of conjugacy for OTW-subshifts, as we see below.

In order to simplify notation, we use the isomorphism of Theorem~\ref{Steinberg} as an equality in the next theorem. Also, recall that an isomorphism $\Psi:\CA_R(\CG_1)\to\CA_R(\CG_2)$ between Steinberg algebras is said to be diagonal-preserving if $\Psi(\CA_R(\CG_1^{(0)}))=\CA_R(\CG_2^{(0)})$.

\begin{theorem}\label{theone}
    Let $h:\osf_1^{OTW}\to \osf_2^{OTW}$ be a homeomorphism and suppose that $R$ is also an indecomposable ring. The following are equivalent:
    \begin{enumerate}[(i)]
        \item\label{iconj} $h$ is a conjugacy.
        \item\label{ihh} There exists a homeomorphism $\hh:\widehat{\TCB_1}\to \widehat{\TCB_2}$ such that $h\circ\pi_1=\pi_2\circ\hh$, $\hh(\dom(\hsig_1))=\dom(\hsig_2)$ and $\hh\circ\hsig_1=\hsig_2\circ\hh|_{\dom(\hsig_1)}$.
        \item\label{iisogrpcocyle} There exists an isomorphism of topological groupoids $\Phi:\CG(\HTCB_1,\hsig_1)\to\CG(\HTCB_2,\hsig_2)$ such that $c_2\circ\Phi=c_1$, $\pi_2\circ \Phi^{(0)}=h\circ \pi_1$,  $\Phi(\dom(\epsilon^1))=\dom(\epsilon^2)$ and $\epsilon^2\circ\Phi|_{\dom(\epsilon^1)}=\Phi\circ\epsilon^1$, 
        where $\Phi^{(0)}$ is the restriction of $\Phi$ to the unit spaces, and $c_1$ and $c_2$ are the one-cocycles from Definition~\ref{d:SGDS}.
        \item\label{iisogrp} There exists an isomorphism of topological groupoids $\Phi:\CG(\HTCB_1,\hsig_1)\to\CG(\HTCB_2,\hsig_2)$ such that $\pi_2\circ \Phi^{(0)}=h\circ \pi_1$, $\Phi(\dom(\epsilon^1))=\dom(\epsilon^2)$ and $\epsilon^2\circ\Phi|_{\dom(\epsilon^1)}=\Phi\circ\epsilon^1$, where $\Phi^{(0)}$ is the restriction of $\Phi$ to the unit spaces.
        \item\label{iisoalggraded} There exists a $\zn$-graded diagonal-preserving isomorphism $\Psi:\TCA_R(\osf_1)\to \TCA_R(\osf_2)$ such that  \begin{itemize}
            \item $\Psi(f\circ\pi_1)=f\circ h^{-1}\circ\pi_2$ for all $f\in\Lc(\osf_1^{OTW},R)$,
            \item for all finite $M\scj\alf_1$, there exists a finite $N\scj\alf_2$ such that $\Psi(e_M)e_N=\Psi(e_M)$ and
            \[\Psi(\tau_M(f))=\Psi(e_M)\tau_N(\Psi(f))\Psi(e_M)\]
            for every $f\in\TCA_R(\osf_1)$,
            \item for all $N'\scj\alf_2$ finite, there exists $M'\scj\alf_1$ finite such that $\Psi^{-1}(e_{N'})e_{M'}=\Psi^{-1}(e_{N'})$.
        \end{itemize}
        \item\label{iisoalg} There exists a diagonal-preserving isomorphism $\Psi:\TCA_R(\osf_1)\to \TCA_R(\osf_2)$ such that \begin{itemize}
            \item $\Psi(f\circ\pi_1)=f\circ h^{-1}\circ\pi_2$ for all $f\in\Lc(\osf_1^{OTW},R)$,
            \item for all finite $M\scj\alf_1$, there exists a finite $N\scj\alf_2$ such that $\Psi(e_M)e_N=\Psi(e_M)$ and
            \[\Psi(\tau_M(f))=\Psi(e_M)\tau_N(\Psi(f))\Psi(e_M)\]
            for every $f\in\TCA_R(\osf_1)$,
            \item for all $N'\scj\alf_2$ finite, there exists $M'\scj\alf_1$ finite such that $\Psi^{-1}(e_{N'})e_{M'}=\Psi^{-1}(e_{N'})$.
        \end{itemize}
    \end{enumerate}
\end{theorem}

\begin{proof}
\eqref{iconj}$\Rightarrow$\eqref{ihh} This follows from Proposition~\ref{hlifts}.

\eqref{ihh}$\Rightarrow$\eqref{iisogrpcocyle} Supposing the existence of $\hh$ as in \eqref{ihh}, it is straightforward to check that the map $\Phi:\CG(\HTCB_1,\hsig_1)\to\CG(\HTCB_2,\hsig_2)$ given by $\Phi(\xi,m,\eta)=(\hh(\xi),m,\hh(\eta))$ is a well-defined isomorphism of topological groupoids satisfying the conditions of \eqref{iisogrpcocyle}.

\eqref{iisogrpcocyle}$\Rightarrow$\eqref{iisogrp} It is immediate.

\eqref{iisogrp}$\Rightarrow$\eqref{iconj} Let $x\in\osf_1^{OTW}$. Using the surjectivity of $\pi_1$ given by Proposition~\ref{prop:pi_shift_commute}, choose $\xi\in\HTCB_1$ such that $\pi_1(\xi)=x$. By the definitions of $\pi_1$ and $\hsig_1$, we see that $|x|$ is exactly the number of times we can apply $\epsilon_1$ to $\xi$ (via the identification with $(\xi,0,\xi)$).
Analogously, $|\pi_2(\Phi^{(0)}(\xi))|$ is the number of times we can apply $\epsilon_2$ to $\Phi^{(0)}(\xi)$. By hypothesis, $|\pi_2(\Phi^{(0)}(\xi))|=|h(\pi_1(\xi))|=|h(x)|$. Since $\Phi(\dom(\epsilon^1))=\dom(\epsilon^2)$ and $\epsilon^2\circ\Phi|_{\dom(\epsilon^1)}=\Phi\circ\epsilon^1$, the number of times we can apply $\epsilon_1$ to $\xi$ is the same as the number of times we can apply $\epsilon_2$ to $\Phi^{(0)}(\xi)$. Hence $|x|=|h(x)|$.

Now suppose that $x\neq\vec{0}_1$. In this case $\xi\in\dom(\hsig_1)$ and $(\xi,0,\xi)\in\dom(\epsilon_1)$. With the identification of $\HTCB_i$ and the unit space of $\CG(\HTCB_i,\hsig_i)$, for $i=1,2$, we see that the hypothesis on $\epsilon_i$ implies that $\hsig_2(\Phi^{(0)}(\xi))=\Phi^{(0)}(\hsig_1(\xi))$. Then
\begin{align*}
    \sigma_2(h(x))&=\sigma_2(h(\pi_1(\xi)))=\sigma_2(\pi_2(\Phi^{(0)}(\xi)))=\pi_2(\hsig_2(\Phi^{(0)}(\xi)))\\ &=\pi_2(\Phi^{(0)}(\hsig_1(\xi)))=h(\pi_1(\hsig_1(\xi)))=h(\sigma_1(\pi_1(\xi)))=h(\sigma_1(x)).
\end{align*}
Hence $h$ is a length-preserving, shift-commuting homeomorphism, that is, $h$ is a conjugacy.

\eqref{iisogrpcocyle}$\Rightarrow$\eqref{iisoalggraded} We use the Steinberg algebra picture for the subshift algebras as in Theorem \ref{Steinberg}. Below, we use $*$ for the convolution product and $\cdot$ for the pointwise product. If $\Phi$ is an isomorphism of the groupoids, then it is easy to see that the map $\Psi:\TCA_R(\osf_1)\to\TCA(\osf_2)$ given by $\Psi(f)=f\circ \Phi^{-1}$ is a diagonal-preserving algebra isomorphism. By Theorem \ref{Steinberg} and the hypothesis that $c_2\circ \Phi=c_1$, we have that this isomorphism is $\zn$-graded. As explained in the beginning of Section \ref{s:groupoid picture}, we view $\Lc(\osf_i^{OTW},R)$ inside $\TCA_R(\osf_i)$ via the map $\pi_i$, for $i=1,2$. Then, because $\Phi^{(0)}$ and $h$ are homeomorphisms, we have that
\[\Psi(f\circ\pi_1)=f\circ\pi_1\circ\Phi^{-1}=f\circ\pi_1\circ(\Phi^{(0)})^{-1}=f\circ h^{-1}\circ\pi_2.\]

Now, let $M\scj\alf_1$ be finite and let $N\scj\alf_2$ be as in Proposition \ref{p:restrictionM}. Notice that $e_M=1_{V^1_M}$ and $e_N=1_{V^2_N}$. Also, $V^1_M=s(\dom(\epsilon^1_M))=r(\dom(\epsilon^1_M))$, and similarly for $V^2_N$. Then
\[\Psi(e_M)e_N=(1_{V^1_M}\circ\Phi^{-1})1_{V^2_N}=1_{\Phi(V^1_M)}1_{V^2_N}=1_{\Phi(V^1_M)}=\Psi(e_M),\]
where the second to last equality follows from the fact that $\Phi(\dom(\epsilon^1_M))\scj\dom(\epsilon^2_N)$. Because $\dom(\epsilon^1_M)=s^{-1}(V^1_M)\cap r^{-1}(V^1_M)$, for any $g\in\TCA_R(\osf_2)$ and $(\xi,m,\eta)\in \CG(\HTCB_2,\hsig_2)$, we have that
\begin{align*}
    \Psi(e_M)*g*\Psi(e_M)(\xi,m,\eta)&=1_{\Phi(V^1_M)}*g*1_{\Phi(V^1_M)}(\xi,m,\eta)\\
    &=1_{\Phi(V^1_M)}(\xi)g(\xi,m,\eta)1_{\Phi(V^1_M)}(\eta)\\
    &=g(\xi,m,\eta)1_{\Phi(\dom(\epsilon^1_M))}(\xi,m,\eta),
\end{align*}
that is, $\Psi(e_M)*g*\Psi(e_M)=g\cdot 1_{\Phi(\dom(\epsilon^1_M))}$, where multiplication on the right hand side is pointwise. Consider now $f\in\TCA_R(\osf_1)$. On the one hand, we have
\[\Psi(\tau_M(f))=(f\circ\epsilon^1_M\circ\Phi^{-1})\cdot 1_{\Phi(\dom(\epsilon^1_M))}.\]
On the other hand
\begin{align*}
    \Psi(e_M)*(\tau_N(\Psi(f))*\Psi(e_M)&=\Psi(e_M)*((f\circ\Phi^{-1}\circ\epsilon^2_N)\cdot 1_{\dom(\epsilon^2_N)})*\Psi(e_M)\\
    &=(f\circ\Phi^{-1}\circ\epsilon^2_N)\cdot 1_{\dom(\epsilon^2_N)})\cdot 1_{\Phi(\dom(\epsilon^1_M))}\\
    &=(f\circ\Phi^{-1}\circ\epsilon^2_N)\cdot  1_{\Phi(\dom(\epsilon^1_M))},
\end{align*}
where the last equality follows from the inclusion $\Phi(\dom(\epsilon^1_M))\scj\dom(\epsilon^2_N)$.

\eqref{iisoalggraded}$\Rightarrow$\eqref{iisoalg} It is immediate.

\eqref{iisoalg}$\Rightarrow$\eqref{ihh} 
By hypothesis, the isomorphism $\Psi$ restricts to an isomorphism between its diagonal subalgebras, which in turn, by Proposition~\ref{ohcoconuts} gives an isomorphism $\Psi:\Lc(\HTCB_1,R)\to\Lc(\HTCB_2,R)$. By Proposition~\ref{givemeanh},
there exists a homeomorphism $\hh:\HTCB_1\to\HTCB_2$ such that $\Psi(f)=f\circ\hh^{-1}$ for all $f\in\Lc(\HTCB_1,R)$. Then, for all $g\in\Lc(\osf_1^{OTW},R)$, we have
\[g\circ h^{-1}\circ\pi_2=\Psi(g\circ\pi_1)=g\circ\pi_1\circ\hh^{-1}.\]
This implies that $h^{-1}\circ\pi_2=\pi_1\circ\hh^{-1}$ because $\Lc(\osf_1^{OTW},R)$ separates the points of $\osf_1^{OTW}$.

Let $M\scj\alf_1$ finite and choose $N\scj\alf_2$ finite such that $\Psi(e_M)e_N=\Psi(e_M)$. As observed above $e_M=1_{V^1_M}$ so that $\Psi(e_M)=1_{V^1_M}\circ\hh^{-1}=1_{\hh(V^1_M)}$. We then get that $\hh(V^1_M)\scj V^2_N$ and hence $\hh(\dom(\hsig_1))\scj\dom\hsig_2$. Similarly, we prove that $\hh^{-1}(\dom(\hsig_2))\scj\dom(\hsig_1)$ so that $\hh(\dom(\hsig_1))=\dom(\hsig_2)$.

For each $a\in\alf_1$, $f\in\Lc(\HTCB_1,R)$ and $\xi\in\HTCB_1$, we have that
\begin{align*}
    \tau_{\{a\}}(f)(\xi)&=s_a*f*s_a^*(\xi,0,\xi)\\
    &=s_a(\xi,1,\hsig_1(\xi))f(\hsig_1(\xi))s_a^*(\hsig_1(\xi),-1,\xi)\\
    &=1_{V_a^1}(\xi)\cdot f(\hsig_1(\xi)).
\end{align*}
Hence
\begin{equation}\label{epsi1}
    \Psi(\tau_{\{a\}}(f))=(f\circ\hsig_1\circ\hh^{-1})1_{\hh(V_a^1)}.
\end{equation}
On the other hand, there exists $N\scj\alf_2$ finite such that
\begin{align}
    \begin{split}\label{epsi2}
        \Psi(\tau_{\{a\}}(f))&=\Psi(e_{\{a\}})\tau_N(\Psi(f))\Psi(e_{\{a\}}) \\
        &=1_{\hh(V_a^1)}*(f\circ\hh^{-1}\circ\epsilon^2_N)*1_{\hh(V_a^1)}\\
        &=(f\circ\hh^{-1}\circ\hsig_2)\cdot 1_{\hh(V_a^1)}.
    \end{split}
\end{align}
By comparing \eqref{epsi1} and \eqref{epsi2}, and using that $\Lc(\HTCB_1,R)$ separates the points of $\HTCB_1$, we conclude that $\hsig_1\circ \hh^{-1}|_{\hh(V_a^1)}=\hh^{-1}\circ\hsig_2|_{\hh(V_a^1)}$, or equivalently, $\hh\circ\hsig_1|_{V_a^1}=\hsig_2\circ h|_{V_a^1}$. Taking the union over all $a\in\alf_1$, we then obtain $\hh\circ\hsig_1=\hsig_2\circ\hh|_{\dom(\hsig_1)}$.
\end{proof}

\begin{remark}
    Steinberg proved in \cite[Theorem~5.6]{BenDiagonal} that for graded groupoids $\CG_1$ and $\CG_2$ satisfying the local bisection hypothesis and $R$ an indecomposable ring, there is a graded isomorphism between $\CG_1$ and $\CG_2$ if, and only if, there is a diagonal-preserving graded isomorphism between $\CA_R(\CG_1)$ and $\CA_R(\CG_2)$. That the groupoid $\CG(\HTCB,\hsig)$ associated with a subshift satisfies the local bisection hypothesis follows from \cite[Corollary~9.4]{Reconstruction}. Although the equivalences \eqref{iisogrpcocyle}$\Leftrightarrow$\eqref{iisoalggraded} and \eqref{iisogrp}$\Leftrightarrow$\eqref{iisoalg} of Theorem~\ref{theone} are certainly connected with Steinberg's result, we need more than the existence of isomorphisms to obtain our results, since we need to keep track of the map $h$ and the subalgebras $\Lc(\osf_1^{OTW},R)$ and $\Lc(\osf_2^{OTW},R)$, which might be smaller than the diagonals of $\CG(\HTCB_1,\hsig_1)$ and $\CG(\HTCB_2,\hsig_2)$, respectively (see Propositions~\ref{algebra OTW} and \ref{ohcoconuts}). 
\end{remark}

\bibliographystyle{abbrv}
\bibliography{ref}

\begin{thebibliography}{10}

\bibitem{AbrAraMol}
G.~Abrams, P.~Ara, and M.~Siles~Molina.
\newblock {\em Leavitt path algebras}, volume 2191 of {\em Lecture Notes in
  Mathematics}.
\newblock Springer, London, 2017.

\bibitem{Reconstruction}
B.~Armstrong, G.~G. de~Castro, L.~Orloff~Clark, K.~Courtney, Y.-F. Lin,
  K.~McCormick, J.~Ramagge, A.~Sims, and B.~Steinberg.
\newblock Reconstruction of twisted {Steinberg} algebras.
\newblock {\em Int. Math. Res. Not.}, 2023(3):2474--2542, 2023.

\bibitem{MR3614028}
T.~Bates, T.~M. Carlsen, and D.~Pask.
\newblock {$C^*$}-algebras of labelled graphs {III}---{$K$}-theory
  computations.
\newblock {\em Ergodic Theory Dynam. Systems}, 37(2):337--368, 2017.

\bibitem{BP1}
T.~Bates and D.~Pask.
\newblock {$C\sp *$}-algebras of labelled graphs.
\newblock {\em J. Operator Theory}, 57(1):207--226, 2007.

\bibitem{GonBeu18}
V.~M. Beuter and D.~Gon\c{c}alves.
\newblock The interplay between {S}teinberg algebras and skew rings.
\newblock {\em J. Algebra}, 497:337--362, 2018.

\bibitem{BoavaDeCastroMortari1}
G.~Boava, G.~G. de~Castro, and F.~de~L.~Mortari.
\newblock Inverse semigroups associated with labelled spaces and their tight
  spectra.
\newblock {\em Semigroup Forum}, 94(3):582--609, 2017.

\bibitem{Gil3}
G.~Boava, G.~G. de~Castro, and F.~de~L.~Mortari.
\newblock Groupoid {M}odels for the {C}*-{A}lgebra of {L}abelled {S}paces.
\newblock {\em Bull. Braz. Math. Soc. (N.S.)}, 51(3):835--861, 2020.

\bibitem{MR4583730}
G.~Boava, G.~G. de~Castro, D.~Gon\c{c}alves, and D.~W. van Wyk.
\newblock Leavitt path algebras of labelled graphs.
\newblock {\em J. Algebra}, 629:265--306, 2023.

\bibitem{BrixCarlsen1}
K.~A. Brix and T.~M. Carlsen.
\newblock Cuntz-{K}rieger algebras and one-sided conjugacy of shifts of finite
  type and their groupoids.
\newblock {\em J. Aust. Math. Soc.}, 109(3):289--298, 2020.

\bibitem{BrixCarlsen}
K.~A. Brix and T.~M. Carlsen.
\newblock {$\rm C^*$}-algebras, groupoids and covers of shift spaces.
\newblock {\em Trans. Amer. Math. Soc. Ser. B}, 7:134--185, 2020.

\bibitem{BCW}
N.~Brownlowe, T.~M. Carlsen, and M.~Whittaker.
\newblock Graph algebras and orbit equivalence.
\newblock {\em Ergodic Theory Dynam. Systems}, 37:389--417, 2017.

\bibitem{CarlsenShift}
T.~M. Carlsen.
\newblock Cuntz--{Pimsner} {{\(C^*\)}}-algebras associated with subshifts.
\newblock {\em Int. J. Math.}, 19(1):47--70, 2008.

\bibitem{COP}
T.~M. Carlsen, E.~Ortega, and E.~Pardo.
\newblock {$C^*$}-algebras associated to {B}oolean dynamical systems.
\newblock {\em J. Math. Anal. Appl.}, 450(1):727--768, 2017.

\bibitem{GroupoidLeavitt}
L.~O. Clark, C.~Farthing, A.~Sims, and M.~Tomforde.
\newblock A groupoid generalisation of {Leavitt} path algebras.
\newblock {\em Semigroup Forum}, 89(3):501--517, 2014.

\bibitem{CuntzKrieger}
J.~Cuntz and W.~Krieger.
\newblock A class of {$C^{\ast} $}-algebras and topological {M}arkov chains.
\newblock {\em Invent. Math.}, 56(3):251--268, 1980.

\bibitem{GDD2}
G.~G. de~Castro, D.~Gon\c{c}alves, and D.~W. van Wyk.
\newblock Ultragraph algebras via labelled graph groupoids, with applications
  to generalized uniqueness theorems.
\newblock {\em J. Algebra}, 579:456--495, 2021.

\bibitem{GillesEunJi2}
G.~G. de~Castro and E.~J. Kang.
\newblock {$C^*$}-algebras of generalized {B}oolean dynamical systems as
  partial crossed products.
\newblock {\em J. Algebr. Comb.}, https://doi.org/10.1007/s10801-022-01170-x,
  2022.

\bibitem{GillesEunJi1}
G.~G. de~Castro and E.~J. Kang.
\newblock Boundary path groupoids of generalized {B}oolean dynamical systems
  and their {$C^\ast$}-algebras.
\newblock {\em J. Math. Anal. Appl.}, 518(1):Paper No. 126662, 2023.

\bibitem{GillesDanie}
G.~G. de~Castro and D.~W. van Wyk.
\newblock Labelled space {$C^*$}-algebras as partial crossed products and a
  simplicity characterization.
\newblock {\em J. Math. Anal. Appl.}, 491(1):124290, 35, 2020.

\bibitem{MishaRuy}
M.~Dokuchaev and R.~Exel.
\newblock Partial actions and subshifts.
\newblock {\em J. Funct. Anal.}, 272(12):5038--5106, 2017.

\bibitem{NamGonc}
T.~T.~H. Duyen, D.~Gonçalves, and T.~G. Nam.
\newblock On the ideals of ultragraph {L}eavitt path algebras.
\newblock {\em arXiv:2109.10440 [math.RA]}, 2021.

\bibitem{ExelBook}
R.~Exel.
\newblock {\em Partial dynamical systems, {F}ell bundles and applications},
  volume 224 of {\em Mathematical Surveys and Monographs}.
\newblock American Mathematical Society, Providence, RI, 2017.

\bibitem{ExelLaca03}
R.~Exel and M.~Laca.
\newblock Partial dynamical systems and the {KMS} condition.
\newblock {\em Comm. Math. Phys.}, 232(2):223--277, 2003.

\bibitem{MR3600124}
D.~Gon\c{c}alves and D.~Royer.
\newblock Ultragraphs and shift spaces over infinite alphabets.
\newblock {\em Bull. Sci. Math.}, 141(1):25--45, 2017.

\bibitem{MR3938320}
D.~Gon\c{c}alves and D.~Royer.
\newblock Infinite alphabet edge shift spaces via ultragraphs and their {$\rm
  C^*$}-algebras.
\newblock {\em Int. Math. Res. Not. IMRN}, 2019(7):2177--2203, 2019.

\bibitem{MR4634309}
D.~Gon\c{c}alves and D.~Royer.
\newblock Properties of the gradings on ultragraph algebras via the underlying
  combinatorics.
\newblock {\em J. Algebraic Combin.}, 58(2):435--451, 2023.

\bibitem{GSS}
D.~Gon\c{c}alves, M.~Sobottka, and C.~Starling.
\newblock Sliding block codes between shift spaces over infinite alphabets.
\newblock {\em Math. Nachr.}, 289(17-18):2178--2191, 2016.

\bibitem{goncalves_royer_2019}
D.~Gonçalves and D.~Royer.
\newblock Simplicity and chain conditions for ultragraph {L}eavitt path
  algebras via partial skew group ring theory.
\newblock {\em J. Aust. Math. Soc.}, 109(3):299--319, 2020.

\bibitem{HazratClassificationLPA}
R.~Hazrat.
\newblock The graded {G}rothendieck group and the classification of {L}eavitt
  path algebras.
\newblock {\em Math. Ann.}, 355(1):273--325, 2013.

\bibitem{Hazrat98}
R.~Hazrat.
\newblock The graded structure of {L}eavitt path algebras.
\newblock {\em Israel J. Math.}, 195(2):833--895, 2013.

\bibitem{imanfar2017leavitt}
M.~Imanfar, A.~Pourabbas, and H.~Larki.
\newblock The {L}eavitt path algebras of ultragraphs.
\newblock {\em Kyungpook Math. J.}, 60(1):21--43, 2020.

\bibitem{MR3517565}
R.~Johansen and A.~P.~W. S{\o}rensen.
\newblock The {C}untz splice does not preserve {$\ast$}-isomorphism of
  {L}eavitt path algebras over {$\mathbb{Z}$}.
\newblock {\em J. Pure Appl. Algebra}, 220(12):3966--3983, 2016.

\bibitem{Keimel}
K.~Keimel.
\newblock Alg\`ebres commutatives engendr\'{e}es par leurs \'{e}l\'{e}ments
  idempotents.
\newblock {\em Canadian J. Math.}, 22:1071--1078, 1970.

\bibitem{KitchensBook}
B.~P. Kitchens.
\newblock {\em Symbolic dynamics}.
\newblock Universitext. Springer-Verlag, Berlin, 1998.
\newblock One-sided, two-sided and countable state Markov shifts.

\bibitem{LindMarcus}
D.~Lind and B.~Marcus.
\newblock {\em An Introduction to Symbolic Dynamics and Coding}.
\newblock Cambridge University Press, 1995.

\bibitem{MatsumotoConjugacy}
K.~Matsumoto.
\newblock One-sided topological conjugacy of normal subshifts and gauge actions
  on the associated {$C^*$}-algebras.
\newblock {\em Dyn. Syst.}, 36(4):586--607, 2021.

\bibitem{MatuiMatsumoto}
K.~Matsumoto and H.~Matui.
\newblock Continuous orbit equivalence of topological {M}arkov shifts and
  {C}untz-{K}rieger algebras.
\newblock {\em Kyoto J. Math.}, 54(4):863--877, 2014.

\bibitem{OTW}
W.~Ott, M.~Tomforde, and P.~N. Willis.
\newblock {\em One-sided shift spaces over infinite alphabets}, volume~5 of
  {\em New York Journal of Mathematics. NYJM Monographs}.
\newblock State University of New York, University at Albany, Albany, NY, 2014.

\bibitem{Renault00}
J.~Renault.
\newblock Cuntz-like algebras.
\newblock In {\em Operator theoretical methods ({T}imi\c soara, 1998)}, pages
  371--386. Theta Found., Bucharest, 2000.

\bibitem{MR1822107}
O.~M. Sarig.
\newblock Phase transitions for countable {M}arkov shifts.
\newblock {\em Comm. Math. Phys.}, 217(3):555--577, 2001.

\bibitem{BenGroupoid}
B.~Steinberg.
\newblock A groupoid approach to discrete inverse semigroup algebras.
\newblock {\em Adv. Math.}, 223(2):689--727, 2010.

\bibitem{BenDiagonal}
B.~Steinberg.
\newblock Diagonal-preserving isomorphisms of \'{e}tale groupoid algebras.
\newblock {\em J. Algebra}, 518:412--439, 2019.

\bibitem{MR4443753}
F.~A. Tasca and D.~Gon\c{c}alves.
\newblock K{MS} states and continuous orbit equivalence for ultragraph shift
  spaces with sinks.
\newblock {\em Publ. Mat.}, 66(2):729--787, 2022.

\bibitem{MR2050134}
M.~Tomforde.
\newblock A unified approach to {E}xel-{L}aca algebras and {$C\sp
  \ast$}-algebras associated to graphs.
\newblock {\em J. Operator Theory}, 50(2):345--368, 2003.

\bibitem{LPA_ring}
M.~Tomforde.
\newblock Leavitt path algebras with coefficients in a commutative ring.
\newblock {\em J. Pure Appl. Algebra}, 215(4):471--484, 2011.

\bibitem{Vas}
L.~Vas.
\newblock Every graded ideal of a {L}eavitt path algebra is graded isomorphic
  to a {L}eavitt path algebra.
\newblock {\em Bull. Aust. Math. Soc.}, 105(2):248--256, 2022.

\end{thebibliography}

\end{document}